\documentclass[a4paper,12pt]{amsart}
\usepackage{amsmath,amsfonts,amssymb,amsthm,graphicx}
\usepackage{overpic,verbatim}
\usepackage{bm} 
\usepackage{gensymb,color}

\usepackage[totalwidth=15.75cm,totalheight=22.275cm]{geometry}

\usepackage[shortlabels]{enumitem}

\numberwithin{equation}{section}

\usepackage[labelformat=simple]{subcaption}

\usepackage{etoolbox}
\makeatletter 
\patchcmd\caption@subtypehook{\let\label\subcaption@label}
{\let\label\subcaption@label\let\ltx@label\subcaption@label}{}{\fail}
\makeatother

\usepackage{hyperref}

\usepackage{orcidlink}

\makeatletter
\def\swappedhead#1#2#3{%
  \thmnumber{\@upn{\the\thm@headfont #2\@ifnotempty{#1}{.~}}}%
  \thmname{#1}%
  \thmnote{ {\the\thm@notefont(#3)}}}
\makeatother

\swapnumbers
\theoremstyle{plain}

\newtheorem{thm}{Theorem}[section]%
\newtheorem{lem}[thm]{Lemma}%
\newtheorem{cor}[thm]{Corollary}%
\newtheorem{prop}[thm]{Proposition}%
\newtheorem{conj}[thm]{Conjecture}%

\theoremstyle{definition}
\newtheorem{rmk}[thm]{Remark}%
\newtheorem{question}[thm]{Question}%
\newtheorem{defn}[thm]{Definition}%

 \newtheoremstyle{remarkstyle}%
   {}
   {}
   {\normalfont}
   {}
   {\bfseries}
   {.}
   { }
   {\thmnote{#3}}

\theoremstyle{remarkstyle}
\newtheorem*{varremark}{}
\newenvironment{remark}[1][Remark]{\begin{varremark}[#1]}{\end{varremark}}

 \newtheoremstyle{claimstyle}%
   {}
   {}
   {\normalfont}
   {}
   {\itshape}
   {.}
   { }
   {\thmnote{#3}}
   
\theoremstyle{claimstyle}
\newtheorem*{varclaim}{}

\newcommand{\T}{\mathcal{T}}
\newcommand{\eucl}{\operatorname{eucl}}

\newcommand{\closure}{\operatorname{cl}}

\newcommand{\Deltafn}{\bm{d}}

\newcommand{\deriv}{{\rm d}}

\newcommand{\Teich}{\mathcal{T}}

\newcommand{\supp}{\operatorname{supp}}
\newcommand{\area}{\operatorname{area}}

\newcommand{\HH}{\mathbb{H}}
\renewcommand{\AA}{\mathbb{A}}
\newcommand{\DD}{\mathbb{D}}

\newcommand{\id}{\operatorname{id}}

\newcommand{\interior}{\operatorname{int}}

\newcommand{\N}{\mathbb{N}}

\newcommand{\Z}{\mathbb{Z}}

\newcommand{\R}{\mathbb{R}}
\newcommand{\C}{\ensuremath{\mathbb{C}}}
\newcommand{\Ch}{\hat{\mathbb{C}}}

\newcommand{\D}{\mathbb{D}}

\newcommand{\eps}{\ensuremath{\varepsilon}}
\renewcommand{\theta}{\vartheta}
\renewcommand{\phi}{\varphi}

\newcommand{\Log}{\operatorname{Log}}

\newcommand{\dist}{\operatorname{dist}}
\newcommand{\diam}{\operatorname{diam}}

\newcommand*{\defeq}{\mathrel{\vcenter{\baselineskip0.5ex \lineskiplimit0pt
                     \hbox{\scriptsize.}\hbox{\scriptsize.}}}%
                     =}

\title[Non-compact surfaces are equilaterally triangulable]{Non-compact Riemann surfaces \\
   are equilaterally triangulable}

\begin{document} 

\begin{abstract}
 We show that every open Riemann surface $X$ can be obtained by glueing 
  together a countable collection of equilateral triangles, in such a way that 
  every vertex belongs to finitely many triangles. Equivalently, $X$
  is a \emph{Belyi surface}: There exists a holomorphic 
    branched covering $f\colon X\to\Ch$ that is branched only over
   $-1$, $1$ and $\infty$. 
  It follows that every Riemann surface is a branched cover of the sphere, branched
 only over finitely many points. 
\end{abstract}


\author[Christopher J.\ Bishop]{Christopher J.\ Bishop\orcidlink{0000-0002-8459-5448}}
\address{Stony Brook University \\ Stony Brook, NY, 11790 \\ USA}
\email{bishop@math.sunysb.edu}
\author[Lasse Rempe]{Lasse Rempe\orcidlink{0000-0001-8032-8580}} 
\address{\noindent Dept. of Mathematics \\ The University of Manchester \\ Manchester \\ M13 9PL \\ UK}
\email{lasse.rempe@manchester.ac.uk}

\maketitle

\section{Introduction}

This article considers the following question: which Riemann surfaces 
 can be built from equilateral triangles? Among \emph{compact}
  surfaces, only countably many surfaces can be obtained in this
  manner. These are characterized by Belyi's theorem in terms of algebraic
  number theory, and have been the subject of intense investigation
  for over forty years. In contrast, we will prove that every non-compact
  surface can be constructed by gluing together countably many equilateral
  triangles (see Theorem~\ref{thm:triangulation} below). The question was initially motivated by the 
  technique of \emph{quasiconformal folding} recently developed by 
  the first author \cite{BishopFolding}, but 
  our proof is self-contained and introduces several novel ideas and techniques.
  These allow greater control on the surfaces and associated functions constructed,
  and are required to overcome significant new difficulties that arise. 
  
  A consequence (Corollary~\ref{cor:branchedcover}) of our main result is that every Riemann surface
   is a branched cover of the Riemann sphere, branched over a finite 
   number of points (this follows from the Riemann--Roch  theorem for
   compact surfaces). Our construction also gives rise to new examples of
   finite-type  holomorphic dynamical systems, generalising known examples
   of maps  from  elliptic or parabolic surfaces  to the sphere (i.e., 
   rational maps on the sphere, and transcendental meromorphic dynamics 
   on the plane or once punctured plane).  These and other implications 
   are discussed at the end of the introduction. 
   
 To state our results formally, let us begin with the definition of a equilateral
   triangulation of a Riemann surface. Let $T$ be a closed Euclidean
   equilateral triangle.
  Starting from 
  either a finite even number or a countably infinite number of copies of 
   $T$, glue these triangles together by identifying every edge with exactly
     one edge of another triangle, in such a way that the identification map is the restriction of an orientation-reversing
   symmetry of $T$.
   Assume furthermore that
   the
   resulting space $E$ is connected, 
   and that any vertex is identified with
   only finitely many other vertices; see Figure~\ref{fig:snowsphere}. 
   Then $E$ is an orientable topological surface, which is 
   compact if and only if the number of triangles we started with was finite. 
   We say that $E$ is an \emph{equilateral surface}.
   
   Every equilateral surface comes equipped with a Riemann surface structure: On 
   the interior of a 
   face or of an edge, the complex structure is inherited from $T$. It is easy to see that
   each vertex is conformally a puncture, and therefore the complex structure
   extends to all of $E$; indeed, local charts can be defined by using appropriate
   power maps. (Recall that 
   every vertex lies on the boundary of some finite number of faces, which are necessarily arranged
   cyclically around it.)
   We say that a Riemann surface is \emph{equilaterally triangulable} if it is 
   conformally equivalent to an equilateral surface; 
   compare \cite{voevodskiishabat} and Section~\ref{sec:belyibackground}. 
 
 \begin{figure}
  \includegraphics[height=.25\textheight]{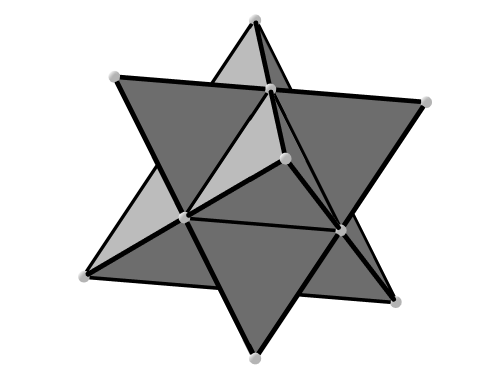}\hfill%
  \includegraphics[height=.25\textheight]{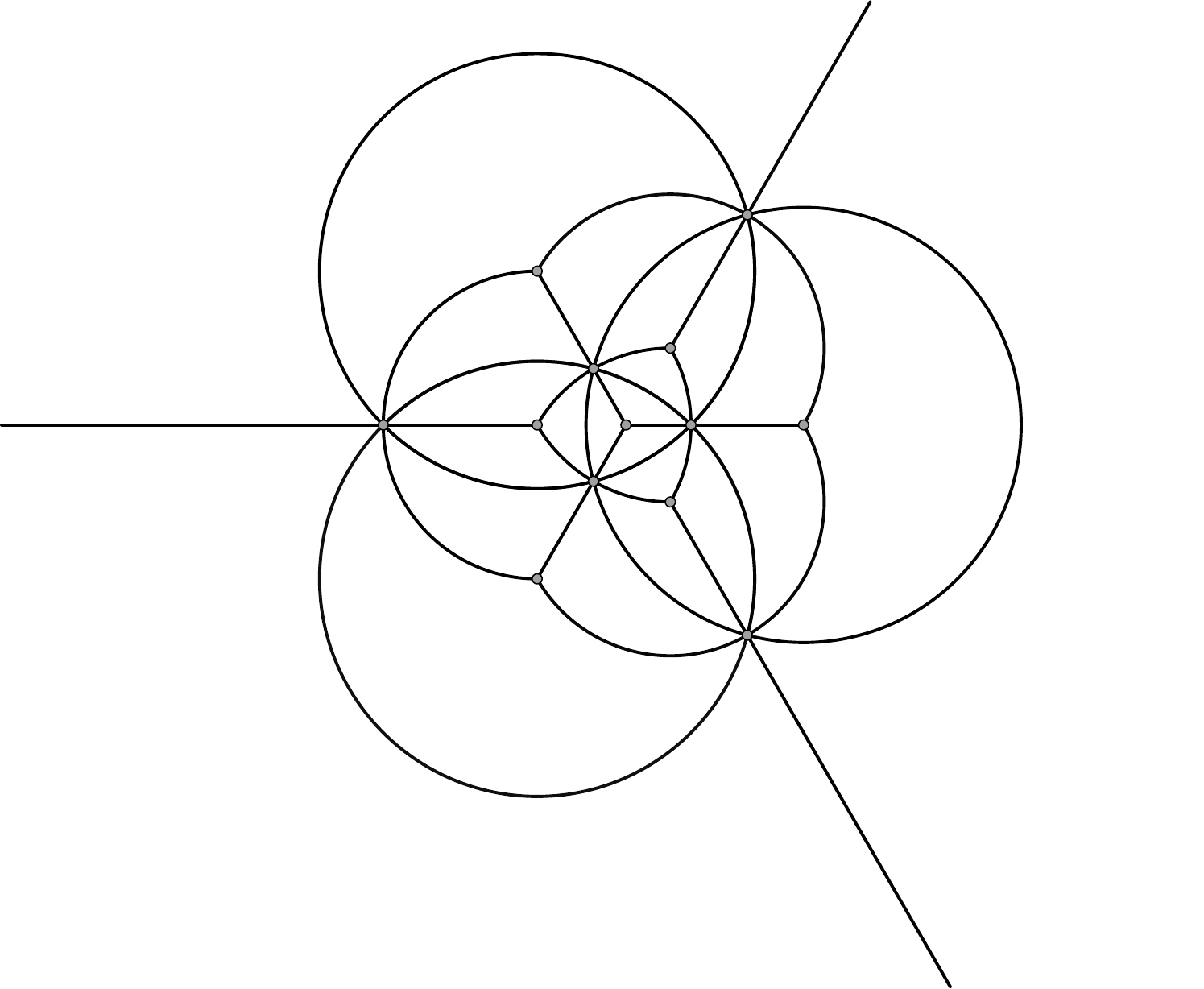}
  \caption{\label{fig:snowsphere}An equilateral surface of genus $0$, built from 24 
    equilateral triangles, and the corresponding triangulation of the Riemann sphere.}
 \end{figure}
 
   \begin{question}\label{question:belyi}
     Which Riemann surfaces are equilaterally triangulable? 
   \end{question}
   We emphasise that Question~\ref{question:belyi} concerns conformal rather than
     metric structures. That is, a conformal isomorphism from 
     a given Riemann surface $X$ to 
     an equilateral surface $E$ 
     induces 
     a flat metric on $X$ having
     isolated cone singularities; different triangulations will lead to
     different metrics. Question~\ref{question:belyi} asks whether $X$ supports
     \emph{any} such \emph{equilateral triangulation}. 
     
  There are only countably many constellations in which one may glue finitely many triangles together.
    So 
     there are only countably many compact equilateral surfaces; therefore
   most compact Riemann
   surfaces 
   can \emph{not} be equilaterally triangulated. 
    The first explicit mention of equilateral triangulations on compact surfaces in the literature
    of which we are aware is in the context of string theory~\cite{boulatovetal}. 
    In response to~\cite{boulatovetal}, and making use of
    ideas from Grothendieck's 1984 ``Esquisse d'un programme'' \cite{grothendieck} relating
    to work of Belyi~\cite{belyi},
    Shabat and Voevodskii~\cite{voevodskiishabat} point out  that
    $X$ is equilaterally triangulable if and only
   if there exists
    a \emph{Belyi function} $f\colon X\to \Ch$; that is, a meromorphic function
   whose only critical values are $-1$, $1$ and $\infty$.\footnote{%
    Often, the values $0$, $1$, and $\infty$ are used in the definition
      of Belyi functions, but our choice turns out to 
    be more convenient for explicit formulae. Either normalisation can be obtained
    from the other by postcomposition with a complex affine 
    transformation.}
     Compare Proposition~\ref{prop:equivalence}.
    
  Such a surface is called a \emph{Belyi surface}. 
    Belyi's theorem~\cite[Theorem~4]{belyi}, see also~\cite{belyi2}, 
   states that $X$ is a Belyi surface if and only if $X$ is defined over a number field. 
   (That is, $X$ can be represented as a smooth projective variety, defined by
   equations with algebraic coefficients.) In particular, this classical theorem
   gives a complete answer to Question~\ref{question:belyi} in the compact case. 
   Belyi functions on compact surfaces are the subject of intense research, 
   particularly in connection with Grothendieck's programme for studying the
   absolute Galois group. Compare~\cite{schnepscollection,landozvonkin,joneswolfart}. 

 It seems  natural to study Question~\ref{question:belyi} also for non-compact 
    surfaces. See
     below for motivations of this problem from complex dynamics,
    in terms of the existence of \emph{finite-type maps}, and from the point of
    view of \emph{conformal tilings}. 
    The answer is  trivial in the case of
     the Euclidean or hyperbolic plane or
     the bi-infinite cylinder. Indeed,
      the plane can be tesselated using equilateral triangles; since this tesselation
      is periodic, it also provides a tesselation of the cylinder $\C/\Z$. Equilateral
      triangulations of the hyperbolic plane are provided by the classical hyperbolic
      triangle groups. Furthermore, it is not difficult to obtain equilaterally 
      triangulated surfaces that are conformally equivalent to the three-punctured sphere
       or the once-punctured disc; 
      see Section~\ref{sec:belyibackground} and Figure~\ref{fig:trivialcases}.

   Every Riemann surface $X$ is triangulable by Rad\'o's theorem
    \cite{Rado1925}. Replacing each element of the triangulation by an equilateral triangle,
    we see that there is  an equilaterally triangulable surface topologically equivalent
    to $X$. However, in general the two surfaces are not \emph{conformally}
    equivalent. Indeed, Riemann surfaces are arranged in \emph{moduli spaces},
    which are nontrivial real or complex manifolds except in the finitely many 
     cases mentioned above. 
     The simplest  examples of non-trivial moduli spaces of non-compact surfaces
     are provided by
     round annuli $\{1<\lvert z\rvert < R\}$, which form a real one-dimensional
      family parameterised by $R\in(1,\infty)$, and 
      four-punctured spheres, which are organised in a one-complex-dimensional moduli 
      space, locally parameterised by the cross ratio of their punctures.  
     As far as we are aware, 
       Question~\ref{question:belyi} is open even for these two
     simple cases. We give a complete answer for all
     non-compact surfaces, which shows that this case differs fundamentally 
     from that of compact Belyi surfaces.

 \begin{thm}\label{thm:triangulation}
   Every non-compact Riemann surface is equilaterally 
     triangulable. 
 \end{thm}

As with compact surfaces, we can rephrase equilateral triangulability
   in terms of Belyi functions. 
        
 \begin{defn}
  Let $X$ be a (compact or non-compact) Riemann surface. 
  A meromorphic function $f\colon X\to\Ch$ is a 
   \emph{Belyi function} if $f$ is a branched covering whose
   branched points lie only over $-1$, $1$ and $\infty$.
\end{defn}
\begin{remark}[Remark~1]
   Here $f$ is called a \emph{branched covering} if every point 
     $w\in\Ch$ has a simply connected neighborhood $U$ such that each connected component
    $V$ of $f^{-1}(U)$ is simply connected and $f\colon V\to U$ is a proper map topologically
     equivalent to $z\mapsto z^d$ for some $d\geq 1$.
    
  Observe that, by definition of a branched covering $f\colon X\to \Ch$,
    $X$ is the natural domain of $f$. That is, 
    there is no Riemann surface $Y\supsetneq X$ such that 
    $f$ extends to a holomorphic function $\tilde{f}$ on~$Y$. 
    Indeed, otherwise let
    $z$ belong to the relative boundary of $X$ in $Y$ and set 
      $w\defeq \tilde{f}(z)$. If $U$ is a small neighbourhood of $w$, then there is
      a connected component $V$ of $f^{-1}(U)$ such that $f\colon V\to U$ is not onto, and in particular not proper.
\end{remark}
\begin{remark}[Remark 2]
  The Belyi functions $f\colon \C\to\Ch$ are precisely the transcendental meromorphic functions
   with three critical values and no asymptotic values. See~\cite{langleycriticalvalues} and~\cite{eremenkobelyi} 
   for a discussion of the function-theoretic properties of these functions.
\end{remark}

 The following is an equivalent formulation of Theorem~\ref{thm:triangulation}; see Proposition~\ref{prop:equivalence}. 

 \begin{thm}\label{thm:belyi} Every non-compact Riemann surface supports a Belyi function. 
 \end{thm}       
 
 It is a consequence of the classical Riemann--Roch theorem that every compact 
  Riemann surface is a branched cover of the Riemann sphere, branched over
  finitely many points. Hence Theorem~\ref{thm:belyi} implies a new result for all Riemann surfaces.
  
  \begin{cor}\label{cor:branchedcover}
   Every Riemann surface is a branched cover of the sphere with only finitely
   many branched values.  
  \end{cor}
\begin{remark}[Remark 1]
  Gunning and Narasimhan~\cite{gunningnarasimhan} proved that
  every open Riemann surface $X$ admits a holomorphic
  immersion into the complex plane. That is, there exists a holomorphic 
  mapping $f\colon X\to\C$ which is a local homeomorphism. However, this function cannot be a covering map if $X\neq \C$; so the inverse $f^{-1}$ 
  necessarily has some, and potentially
  infinitely many, transcendental singularities in $\C$. 
	In particular, such $f$ is not a branched covering.
\end{remark}
\begin{remark}[Remark 2]
  For a general compact Riemann surface $X$ of genus $g\geq 2$, the minimal number of branched values required in the theorem is $3g$. Indeed, 
  the moduli space of $X$ has complex dimension $3g-3$. The subset consisting of those surfaces for which there
  is a branched cover branched over only $B\geq 3$ values is a countable union of submanifolds of dimension at most $B-3$. Thus, for a general surface $X$, 
  the number of branched values in Corollary~\ref{cor:branchedcover} is at least $B=3g$. On the other hand, if $X$ is any surface of genus $g$,
   and $P$ is a Weierstrass point
  of $X$, then there is a function $f\colon X\to\Ch$ having a single pole at $P$ of degree at most $g$. By the Riemann--Hurwitz formula, $f$ has at most 
    $3g-1$ finite critical values, and hence $3g$ critical values in total. (We thank Alex Eremenko for pointing out this argument.) 
    For $g=1$, the moduli space is one-dimensional, so we need at least $B=4$ critical values in general; this is achieved by the Weierstrass 
   $\wp$-function. 
   On the
   other hand, Theorem~\ref{thm:belyi} shows that $B=3$ always suffices for \emph{non-compact} $X$.
\end{remark}
  Theorems \ref{thm:triangulation} and~\ref{thm:belyi} may 
   seem surprising since the 
   function $f$ is determined by an underlying equilateral triangulation, 
   which is described by an infinite abstract 
   graph on the surface, a discrete and non-flexible 
   object. In contrast, 
   Riemann surfaces are parameterised by complex manifolds, so 
   the triangulation in Theorem~\ref{thm:triangulation}
      and the Belyi function $f$ in Theorem~\ref{thm:belyi} 
      cannot depend continuously on $X$ as it
      varies in a given moduli
   space. A similar phenomenon appears in the 
   setting of \emph{circle packings}: Every non-compact Riemann surface of finite conformal type (see Section~\ref{sec:belyibackground})
   can be filled by a circle packing \cite{noncompactpackings}. Here a circle packing 
   is a locally finite collection of circles whose tangency graph is a triangulation, and again 
   this
   tangency graph completely determines the surface. 
    However, despite the similarity of statements,
    the techniques used in \cite{noncompactpackings} have no obvious counterpart in the 
    setting of equilateral surfaces.
    Indeed, \cite[Section~3]{noncompactpackings} discusses how one may
    modify an existing partial packing to a full packing by replacing only one 
    of the circles by another chain of circles. On the other hand, 
    an equilateral triangulation is uniquely determined by any one of its
    triangles; see Remark~\ref{rmk:rigidity}. 

  There is a long history of constructing functions with finitely many singular values
   using quasiconformal mappings. See \cite{wittich} and~\cite[Chapter~7]{goldbergostrovskii};
    for a modern example, compare 
    Bergweiler and Eremenko~\cite{BergweilerEremenkoQCSurgery}. The control of the
    geometric behaviour of the resulting functions that can be achieved with
    classical methods is limited, but recently the first author introduced the
    concept of \emph{quasiconformal folding}~\cite{BishopFolding}. This technique allows
    the very flexible construction of functions with finitely many singular values and
    prescribed behaviour. 
     It has subsequently been used by authors including
     Fagella, Godillon and Jarque~\cite{MR3339086}, Lazebnik~\cite{MR3579902},   
     Osborne and Sixsmith~\cite{MR3525384}, and the second author~\cite{arc-like}  
     to construct examples  in transcendental dynamics on the plane.
     Compare Mart\'i-Pete and Shishikura~\cite{martipeteshishikura} for a related construction
     that does not use quasiconformal folding. 
   
While quasiconformal folding has been applied mostly to construct entire functions $f\colon \C\to\C$, 
it also allows the construction of meromorphic functions on more general
Riemann surfaces. More precisely, given any Riemann surface  $X$
(compact or not), quasiconformal folding allows one to construct
a \emph{quasiregular} map $f\colon X\to \Ch$ that is branched only over 
$-1$, $1$ and $\infty$. Moreover, $f$ can be chosen to be ``almost'' holomorphic
(more formally, its maximal dilatation is bounded by a uniform
constant and supported on a 
subset of $X$ of arbitrarily small area). It follows that there is a Belyi function
on a surface $\tilde{X}$ \emph{close} to $X$, establishing that 
equilaterally triangulable surfaces are dense in every moduli space;  
compare \cite[Section~15]{BishopFolding}. 
However, in general $\tilde{X}$ and $X$ have different complex structures.
        
Establishing Theorem~\ref{thm:belyi} hence
requires substantial 
new ideas, which can be outlined as follows. We begin by subdividing
$X$ into countably many pieces of finite topological type. We construct
a finite triangulation on the first such piece $S$ 
that is almost equilateral; more precisely, it becomes equilateral after a quasiconformal
change of the complex structure on $S$. By a careful analysis we see that
this change can be kept so small that the
new surface $\tilde{S}$ re-imbeds into $X$. This allows us to
continue with our construction. 
An additional subtlety arises from the fact that 
choices made at earlier stages of the construction will influence how
small we can 
keep our change in complex structure on subsequent pieces. 
It turns out that it is possible to control this
influence by choosing the equilateral triangulation on each $S$ carefully, together with
results on the area distortion under quasiconformal mappings. 
    
The partial equilateral triangulations could be
constructed by quasiconformal folding. Instead, we use a direct and more elementary
method~-- though still motivated by the ideas of~\cite{BishopFolding}~-- which 
has the additional advantage that the number of triangles meeting at a single point 
is bounded by a universal constant. In particular, we obtain the following strengthening
of Theorem~\ref{thm:belyi}. 
   
\begin{thm}\label{thm:boundeddegree}
 There is a universal constant $D$ such that the Belyi function in 
 Theorem~\ref{thm:belyi} can be chosen to have local degree $\leq D$ at every point. 
\end{thm}
    
Our proof allows many choices at each stage of the inductive construction,
and hence even shows the existence of uncountably many different Belyi functions 
on $X$. We thus obtain a new 
characterisation of compact Riemann surfaces. 

\begin{cor}\label{cor:uncountable} A Riemann surface $X$ is compact 
if and only if supports at most countably many different Belyi functions,
up to pre-composition by conformal automorphisms. 
\end{cor}

\subsection*{Finite-type maps}
  Let $X$ and $Y$ be Riemann surfaces, where  $Y$ is compact. 
    Following Epstein~\cite{adamthesis},   
    a holomorphic function $f\colon X\to Y$ is a \emph{finite-type map}
    if there is a finite set $S$ such that 
       \[ f \colon X\setminus f^{-1}(S) \to Y\setminus S \]
       is a covering map, and furthermore
    $f$ has no removable singularities at any punctures of $X$. The smallest
     such set $S$ is called the \emph{set of singular values}, and denoted by $S(f)$. 
           
    Epstein proved that finite-type maps have
    certain transcendence properties near the boundary, reminiscent of 
    the Ahlfors five islands theorem \cite[Proposition~9]{adamthesis}. 
    In particular, he proved that, when 
    $X\subset Y$, the fundamental results of the classical iteration theory 
    of rational functions, and of entire/meromorphic functions with a 
    finite set of singular values, remain valid for finite-type maps.
    Compare also~\cite{cheritatepstein} and~\cite[Section~2]{hypdim}. 
    
    It is a natural question for which pairs of $X$ and $Y$ finite-type maps exist. 
      Corollary~\ref{cor:branchedcover} shows 
      that there are finite-type maps $X\to \Ch$ for every Riemann surface $X$. 
      In particular, when $X\subsetneq \Ch$ is a proper open subset,  we obtain 
       the existence of many new non-trivial finite-type dynamical systems. 
       
    It is also possible to prove the existence of 
      finite-type maps $f\colon X\to Y$ with $\# S(f) = 1$ 
      for every non-compact Riemann surface $X$ 
      and every torus $Y$. This is achieved by a modification of our methods
      that leads to the existence of a \emph{Shabat function} on $X$; i.e.\ 
      a branched covering map from $X$ to the complex plane $\C$ which
      is branched only over two values. Postcomposing the Shabat function 
      with a projection to
      the torus that identifies the two critical values yields the desired 
      finite-type map. The details of the construction
      will be given in a subsequent article. 
      
     The question of the existence of finite-type maps with target $Y$ becomes
        more subtle when $Y$ is hyperbolic. By Liouville's theorem, $X$ must be hyperbolic 
        if such a map is going to exist. In fact, it is possible to show that 
        the boundary of $X$ must be \emph{uniformly perfect}. That is, 
        the hyperbolic length of any non-contractible closed curve in $X$ is bounded uniformly
        from
        below. 
        
    In \cite[Section~16]{BishopFolding}, the first author uses quasiconformal folding
     to construct finite-type maps
     from certain finite Riemann surfaces $U$ 
     (see Section~\ref{sec:belyibackground}) to all compact
     hyperbolic surfaces. This is achieved by constructing a branched covering
     $U\to \DD$ with only two branched points in $\DD$, and postcomposing with
     the universal covering map. If $U'$ is any finite Riemann surface, then a refinement
     of the method of~\cite[Section~16]{BishopFolding} shows that $U$ can be chosen 
     arbitrarily close to $U'$ in its moduli space. In particular, if $U'$ is a subpiece of 
     some compact Riemann surface $Y$, bounded by disjoint analytic boundary
     circles, then the perturbed surface $U$
     is also embeddable in $Y$ (see Proposition~\ref{prop:straightening} below),
     and we obtain new examples of 
     finite-type dynamical systems with hyperbolic target $Y$. 
     The following
     appears plausible in view of our results.
     
     \begin{conj}
       On every finite Riemann surface $U$, there is a branched covering
       $f\colon U\to\DD$ branched over at most two points. 
       In particular, if $Y$ is any compact hyperbolic surface, and $\pi\colon \DD\to Y$ is its universal
       cover, then $\pi\circ f\colon U\to Y$ is a finite-type map.
     \end{conj}
     
     The method of~\cite[Section~16]{BishopFolding} can also be used to construct 
      finite type maps to hyperbolic surfaces on some infinitely-connected $U$. It
      is an interesting question whether 
      such functions exist on all hyperbolic surfaces with uniformly perfect boundary.
        
\subsection*{Conformal tilings}
  Bowers and Stephenson~\cite{bowersstephensonpentagonal,bowersstephenson,bowersstephensonII} study 
   \emph{conformal tilings} of a Riemann surface $X$, which are obtained by 
   allowing general regular polygons, of the same fixed side-length, in our 
   construction above. In particular, every equilateral triangulation of $X$ is
   also a conformal tiling. Conversely, the barycentric subdivision of a conformal tiling
   is an equilateral triangulation, so a tiling exists if and only if the surface is
   equilaterally triangulable. 
   
 Bowers and Stephenson are mainly interested in the case where $X$ is simply connected.
   As mentioned above, these surfaces are equilaterally triangulable for elementary
   reasons; the cited articles exhibit many interesting and
   beautiful different such conformal tilings. 
   However,~\cite[Appendix~B]{bowersstephenson} also
   raises the question which multiply-connected surfaces admit conformal tilings;
    this is equivalent to Question~\ref{question:belyi}, and 
    Theorem~\ref{thm:triangulation} (together with Belyi's theorem for the compact case) gives
    a complete answer.
   
 \subsection*{Random equilateral triangulations} There is an extensive literature on random equilateral
triangulations of compact surfaces; see e.g.\ 
\cite{MR2152911,MR3080483,BCP-2019}.
In statistical physics, there has been intensive study of
the metric and conformal structures on  compact surfaces built from random
equilateral triangulations, quadrangulations or more general
random maps, and especially of the limits of these  random surfaces
when the number of triangles tends to infinity but the genus is held constant. For
 example, a recent
 major result of Miller and Sheffield \cite{MR4050102,MS-MapII,MS-MapIII}
  shows that two such limiting objects~-- ``Liouville quantum gravity'' and the ``Brownian map''~--
   are essentially the same. Compare also 
\cite{MR2336042,MR4007665,Miermont-aspects}.
For 
analogous constructions on higher genus compact surfaces, see e.g. \cite{DRV-2015,MR3627425}.

In all of these cases, the distribution of the conformal structures
of the discrete random surfaces is
supported on a countable set in moduli space (Belyi surfaces in the case of random
equilateral triangulations), but for a fixed genus, the distributions
conjecturally converge to continuous distributions.
What can be said about random non-compact triangulations?  For the Euclidean plane,
this  question has been addressed by Angel and Schramm
\cite{AngelSchramm}: they show how to define a probability measure on the metric space of
rooted planar triangulations, 
called a uniform infinite planar triangulation (UIPT). Hyperbolic
versions have also been considered; compare \cite{MR3342664,MR3520011,MR4076778}.

The UIPT can be thought of as a uniformly random
surface with the topology of a plane. Can one also make sense of
the notion of a uniformly random surface with the topology of a cylinder, or
some other non-compact topology, such as a compact surface with a puncture?
Scott Sheffield suggested the following formulation of this problem.
Begin with the UIPT, which comes with a distinguished "origin" triangle, and
then cut out that origin triangle and glue in some finite genus graph.
By our results, 
it is at least possible that there is a continuous limiting distribution. Do all conformal
structures occur if we glue in a random finite genus graph? Does a neighborhood of
a point in moduli space  occur if we glue in a fixed choice?

\subsection*{Basic notation}
  The symbols $\C$ and $\Ch$ denote the complex plane and  Riemann sphere,
   respectively.
   The (Euclidean) disc of radius $\rho$ around $w\in\C$ is
   denoted by $D(w,\rho)$; the unit disc is denoted
   $\D \defeq D(0,1)$.  In a slight abuse of terminology, we also denote the 
   complement of 
   the closed unit disc by $D(\infty,1) = \Ch\setminus \overline{\D}$. 
   For $R>1$ we define the following annuli (see Figure~\ref{fig:annuli}):
          \begin{align*} &\AA(R) \defeq \{1/R < \lvert z\rvert < R\}; \\
          &\AA_-(R) \defeq \{ 1/R < \lvert z \rvert < 1 \} = \AA(R)\cap \D; \quad\text{and}\\
          &\AA_+(R) \defeq \{ 1 < \lvert z \rvert < R\} = \AA(R)\setminus\overline{\D}.
            \end{align*}

   A quasiconformal map $\psi$ on a  planar domain $\Omega$ has a 
   \emph{complex dilatation} $\mu = \psi_{\overline{z}} / \psi_z$. 
   This is a measurable function on $\Omega$ and has $L^\infty$ norm 
   equal to some $k  \in [0,1)$.  The \emph{maximal dilatation} of $\psi$ is 
   denoted $K =  (k+1)/(k-1)$; such a map is called $K$-quasiconformal. 
   Geometrically, this is the maximal eccentricity 
   of the elliptical image of a circle under a tangent map of $\psi$. 
   Note that $k=0$ and $K=1$ for conformal maps.
   The term \emph{dilatation} can refer to either of these 
   quantities; for clarity we distinguish between the ``complex dilatation'' 
   $\mu$  and the ``maximal dilatation'' $K$.  

   In general, $A=B$ denotes equality between two previously defined 
   quantities, and $A:=B$ defines $A$ in terms of $B$.

   We assume throughout that the reader is
   familiar with the theory of Riemann surfaces and quasiconformal mappings,
   and refer e.g.\ to~\cite{forsterriemannsurfaces,lehtovirtanen,hubbardteichmuller} 
   for reference. In addition, the proofs in Section~\ref{sec:teichmueller} 
   use background from Teichm\"uller theory. However, this technique
    is not required to understand the statements of
    the main results in these sections, or their applications in the proofs of
    our main theorems.

 \subsection*{Acknowledgements} This research arose from an interesting and stimulating e-mail discussion
  concerning finite-type maps with Adam Epstein and Alex Eremenko; we thank them both for the 
  initial inspiration and subsequent helpful comments and conversations on this work. We are grateful to J.\ Martel \cite{mathoverflow} for
  pointing out the result of Williams on circle packing, Curt McMullen for helpful comments on the case of compact surfaces, and Daniel Meyer 
  for highlighting the connection with conformal tilings. We thank Dmitry Chelkak and Scott Sheffield for providing helpful
  comments and references on random triangulations. We are grateful to 
  the referee for their thoughtful comments on the manuscript that improved
  the presentation. 
  
 This research was partly conducted while the second author was 
   employed at the University of Liverpool, and he gratefully acknowledges 
   support
   by the University of Liverpool for this research through 
   a number of research visits to Stony Brook University.

\section{Riemann surfaces, triangulations and Belyi functions}
\label{sec:belyibackground}

 In this section, we collect background on Riemann surfaces
  and triangulations. In particular, we recall the proof of the fact 
  that a Riemann surface is equilaterally triangulable if and only if
  it supports a Belyi function. 

\subsection*{Riemann surfaces and conformal metrics}
A   \emph{Riemann surface} $X$ is a connected one-dimensional complex
   Hausdorff manifold. 
  By a \emph{conformal metric} on a Riemann surface we mean a length
   element that takes the form
     $\deriv s = \rho(z) \lvert \deriv z\rvert$ in local coordinates (where $\rho$ is a continuous
     positive-valued function). Note that each conformal metric gives rise to
     an area element, $\rho^2(z) \lvert \deriv z\rvert^2$. 
     When such a metric $\rho$ 
     is given, we shall write $\dist_{\rho}$ for the corresponding distance
     function; i.e.\ $\dist_{\rho}(z,w)$ is the largest lower bound for the $\rho$-length of a curve
     connecting $z$ and $w$. (We  
     omit the subscript $\rho$ when it is clear from the context which metric $\rho$
     is to be used.)
     
    By the uniformisation theorem, every Riemann surface can be endowed with a 
      conformal metric of constant curvature; in the case of positive
      or negative curvature,
      this metric becomes unique by requiring that the curvature is $1$ or $-1$,
       respectively. 
      We emphasise that we use conformal metrics only in an inessential
      way, to provide a measure of smallness of area on compact pieces of
      a Riemann surface. Any two conformal metrics on a compact surface
       (or surface-with-boundary) are equivalent; indeed, the quotient of
       their densities is a continuous function and hence 
       assumes a positive and finite maximum and minimum. Thus
       the precise choice of metric
       will be irrelevant. 
       
  \subsection*{Finite pieces of Riemann surfaces}
         A Riemann surface $X$ is said to be \emph{finite} if 
    it is of finite genus with a finite number of boundary components, none of which are degenerate. In other words, $X$ is conformally equivalent to a compact
    Riemann surface with at most finitely many topological discs removed. This notion should not be confused with that of \emph{finite type}: 
   a surface has \emph{finite topological type} if it is 
    homeomorphic to a compact surface with finitely many points removed,
    and it has  \emph{finite conformal type} if this homeomorphism can be chosen analytic. 
   In particular, a non-compact finite Riemann surface has finite topological type, but is never of finite conformal type. 
    (See Figure~\ref{fig:types}.) To avoid ambiguities, we do not use the
   notion of finite conformal type in the remainder of the article. 

\begin{figure}
\begin{center}
\subcaptionbox{Finite Riemann surface}{\includegraphics[width=0.3\textwidth]{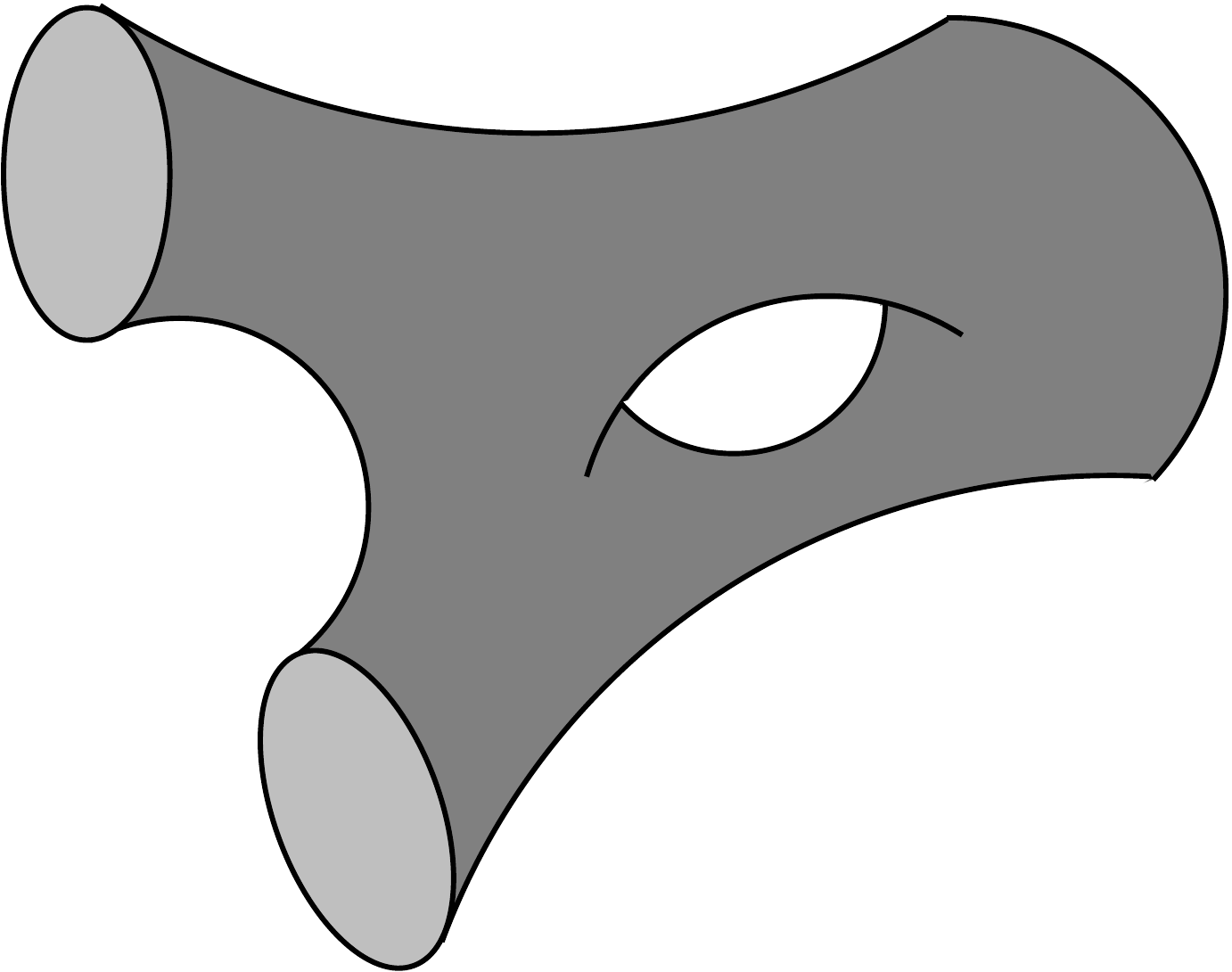}}%
\hfill%
\subcaptionbox{Finite conformal type}{\includegraphics[width=0.3\textwidth]{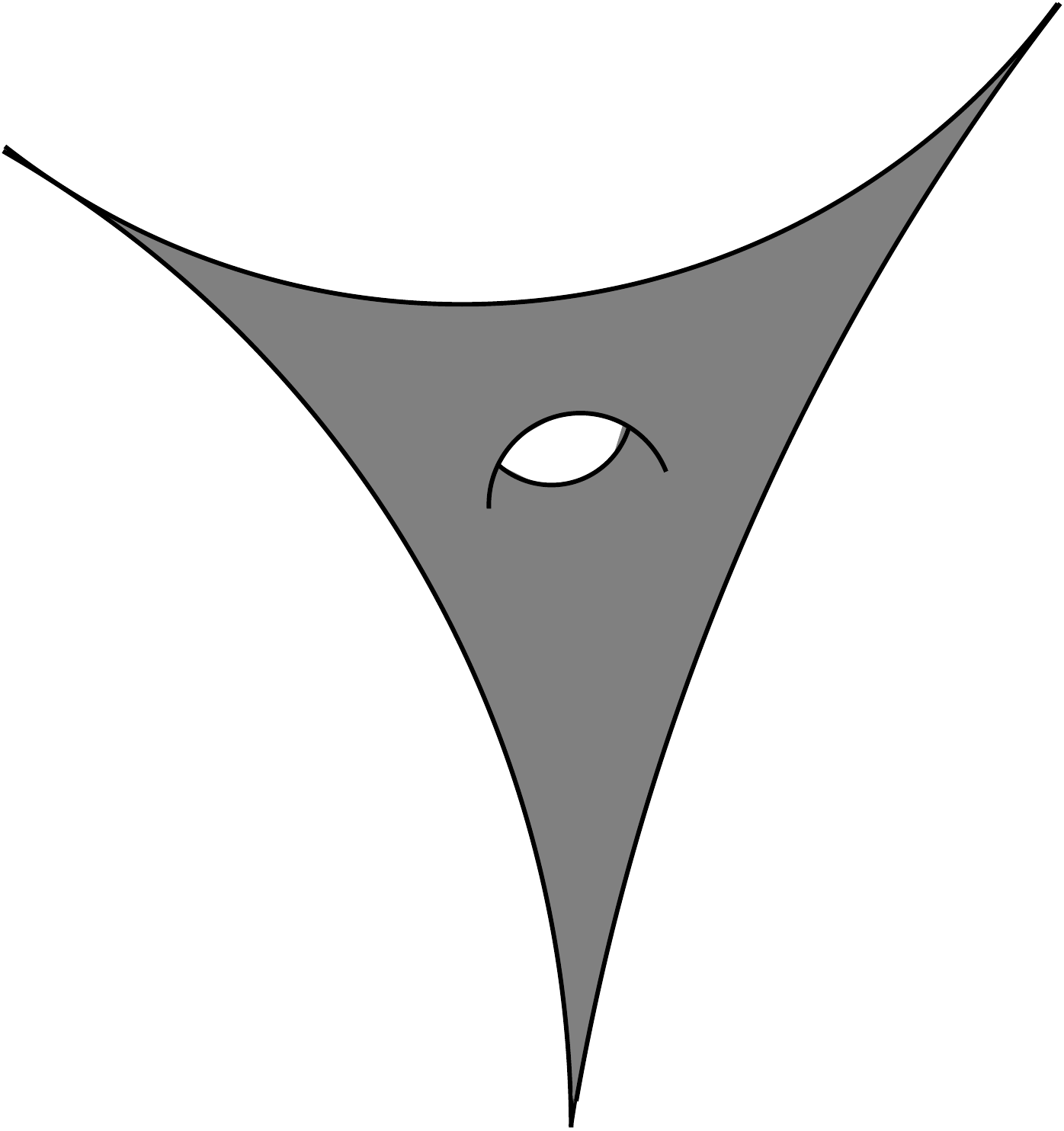}}\hfill
\subcaptionbox{Finite topological type}{\includegraphics[width=0.3\textwidth]{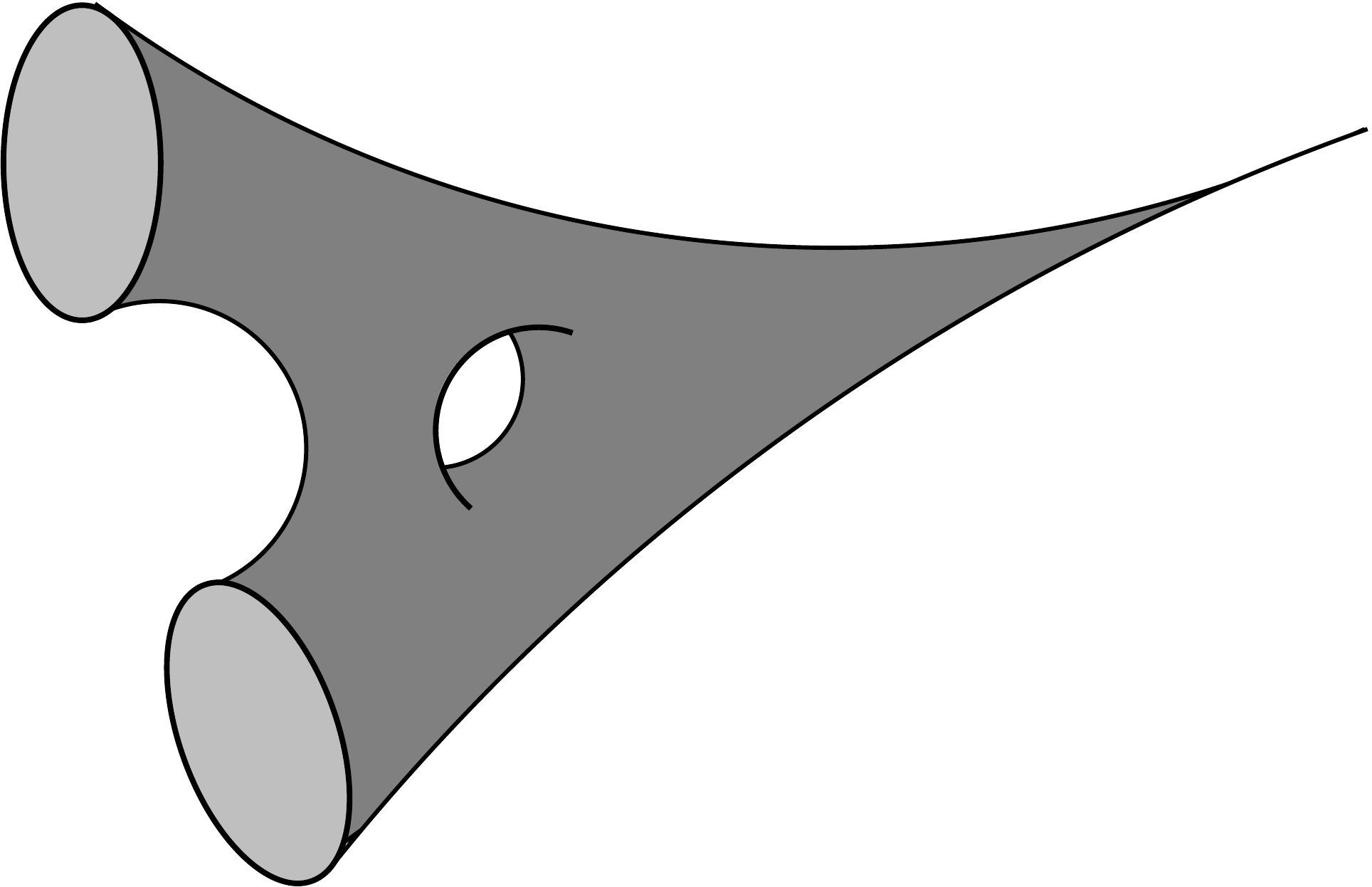}}
\end{center}
\caption{\label{fig:types}Three different notions of finiteness of Riemann
surfaces.}
\end{figure}

 In our context, finite Riemann surfaces often arise as subsets of a larger surface $X$. The following notation will be convenient. 

  \begin{defn}[Finite pieces]\label{defn:finitepiece}
     Let $X$ be a Riemann surface, and  let $U\subsetneq X$ be a finite Riemann surface. If $U$ is pre-compact in $X$, then we say that $U$ is a
     \emph{finite piece} of $X$.  If furthermore $\partial U\subset X$ consists of finitely many analytic Jordan curves (called the \emph{boundary curves} of $U$), then $U$ is said to be 
    \emph{analytically bounded}. 
  \end{defn}

 \subsection*{Boundary coordinates and hemmed surfaces}
    We shall construct triangulations on finite pieces of our Riemann surface $X$. 
      To be able to combine such partial triangulations, we also need to 
      record, for a finite piece, suitable parameterisations of its 
      boundary. We hence introduce the
      following notion. (See Figure~\ref{fig:annuli}.) 
      
    \begin{defn}\label{defn:hemmed}
        A \emph{hemmed Riemann surface} is a non-compact finite Riemann surface $U$, together with 
         analytic parameterisations of  its boundary curves.
       More precisely, let $\Gamma$ be the set of boundary curves of $U$
          (or, in other words, the set of ends of $U$). For each 
          $\gamma\in\Gamma$, let 
            \[ \phi^{\gamma} \colon \AA_-(R^{\gamma}) \to A^{\gamma},\] 
      where $R^{\gamma}>1$, be a conformal map to an annulus $A^{\gamma}\subset U$ such that 
       $\phi^{\gamma}(z)\to \gamma$ as $\lvert z\rvert \to 1$. 
       We furthermore assume that the image annuli $A^{\gamma}$ have
         pairwise disjoint closures.         
        Then we say that $U$ is a \emph{hemmed Riemann surface} with 
          boundary coordinates $(\phi^{\gamma})_{\gamma\in\Gamma}$. 
    \end{defn}

  Observe that the closure of every hemmed Riemann surface is
    a compact Riemann surface-with-boundary, with charts on the boundary
    curve $\gamma$ given by $(\phi^{\gamma})^{-1}$. 
    Conversely, any compact Riemann surface-with-boundary can be given the 
    structure of a hemmed Riemann surface by choosing an annulus $A^{\gamma}$ 
    around each boundary curve, and letting $\phi^{\gamma}$ be a conformal
    map from a round annulus to $A^{\gamma}$. 
    Different choices of annuli will lead to different
    boundary coordinates, and hence to different hemmed surfaces.

\begin{figure}[htb]
  \begin{overpic}[ percent, width = 5in]{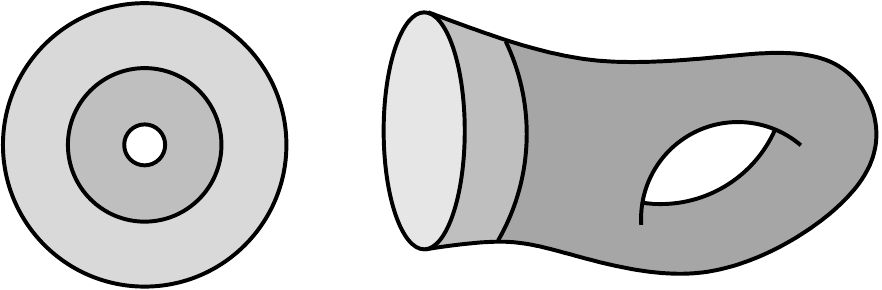}
  \put(-10,20){${\mathbb A}(R)$}
  \put(10,27){${\mathbb A}_+(R)$}
  \put(11,11){${\mathbb A}_-(R)$}
  \put(40,20){$\gamma$}
  \put(53,20){$A^\gamma$}
  \end{overpic}\\[2ex]
  \caption{
	\label{fig:annuli} The annuli $\AA_+(R)$ and $\AA_-(R)$ (left), 
 and a hemmed Riemann surface (right)
        }
\end{figure}
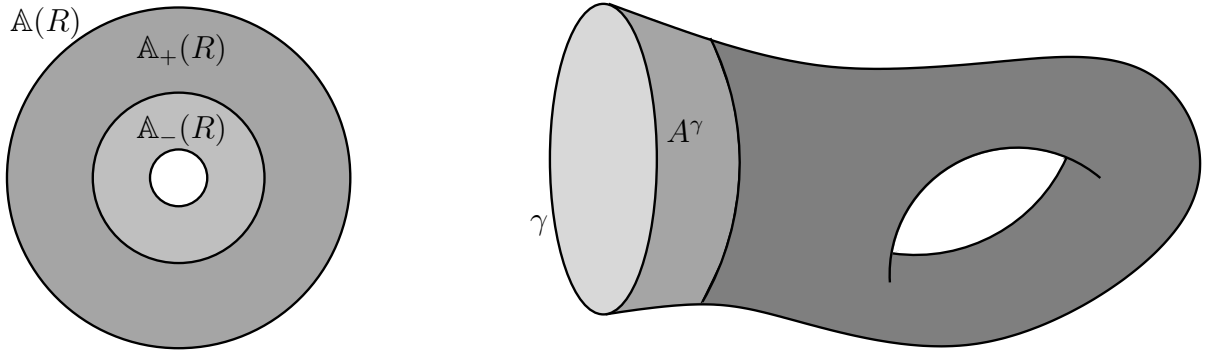

\subsection*{Triangulations} 
\begin{defn} 
  Let $X$ be a Riemann surface, or a Riemann surface-with-boundary. 
   A \emph{triangulation} of $X$ is a countable and locally finite 
   collection of closed topological 
   triangles that cover $X$, such that two triangles intersect only in 
   a full edge or in a vertex.
\end{defn} 

   In other words, a triangulation furnishes $X$ with the structure
   of a locally finite simplicial complex.    
   By a theorem of Rad\'o 
    from 1925 \cite{Rado1925} (see \cite[\S23]{forsterriemannsurfaces} or
    \cite[Theorem~1.3.3]{hubbardteichmuller}), 
     every Riemann surface is second 
     countable, and hence triangulable. 
          
   Let $\mathcal{T}$ be a triangulation
   and let $\Delta$ be the 
         Euclidean equilateral triangle inscribed in the unit circle, with a vertex at $1$.     
   For each topological triangle $T\in\mathcal{T}$, 
     let $\phi_T$ 
     denote a biholomorphic isomorphism that takes $T$ to $\Delta$,
     mapping vertices to vertices. Observe that
     $\phi_T$ is unique up to postcomposition by a rotational symmetry of $\Delta$. 

\begin{defn}
     The triangulation $\mathcal{T}$ is \emph{equilateral} if, on every edge $e$
      with two  
     adjacent triangles $T$ and $\tilde{T}$, the maps $\phi_T$ and
     $\phi_{\tilde{T}}$ agree up to a reflection symmetry of $\Delta$. 
     If such a triangulation exists, we say that $X$ is \emph{equilaterally triangulable}. 
\end{defn}
     
    It is elementary to see that this agrees with the definition given
     in the introduction, with one caveat: The triangulations mentioned there 
     allowed two triangles to intersect in more than one edge; let
     us call these \emph{generalised triangulations} in the following. 
     Given an equilateral generalised triangulation, we can
     perform a barycentric 
     subdivision of all triangles, 
     inserting a new vertex in the barycenter of each face and the mid-point 
     of each edge. In this triangulation, no two triangles intersect in more than
     one edge. The following observation shows that this 
     triangulation is also equilateral; see Figures~\ref{fig:barycentric} and~\ref{fig:trivialcases}. 
     Compare~\cite[\S1.3]{bowersstephenson}. 

\begin{figure}
\begin{center}
 \includegraphics[width=.7\textwidth]{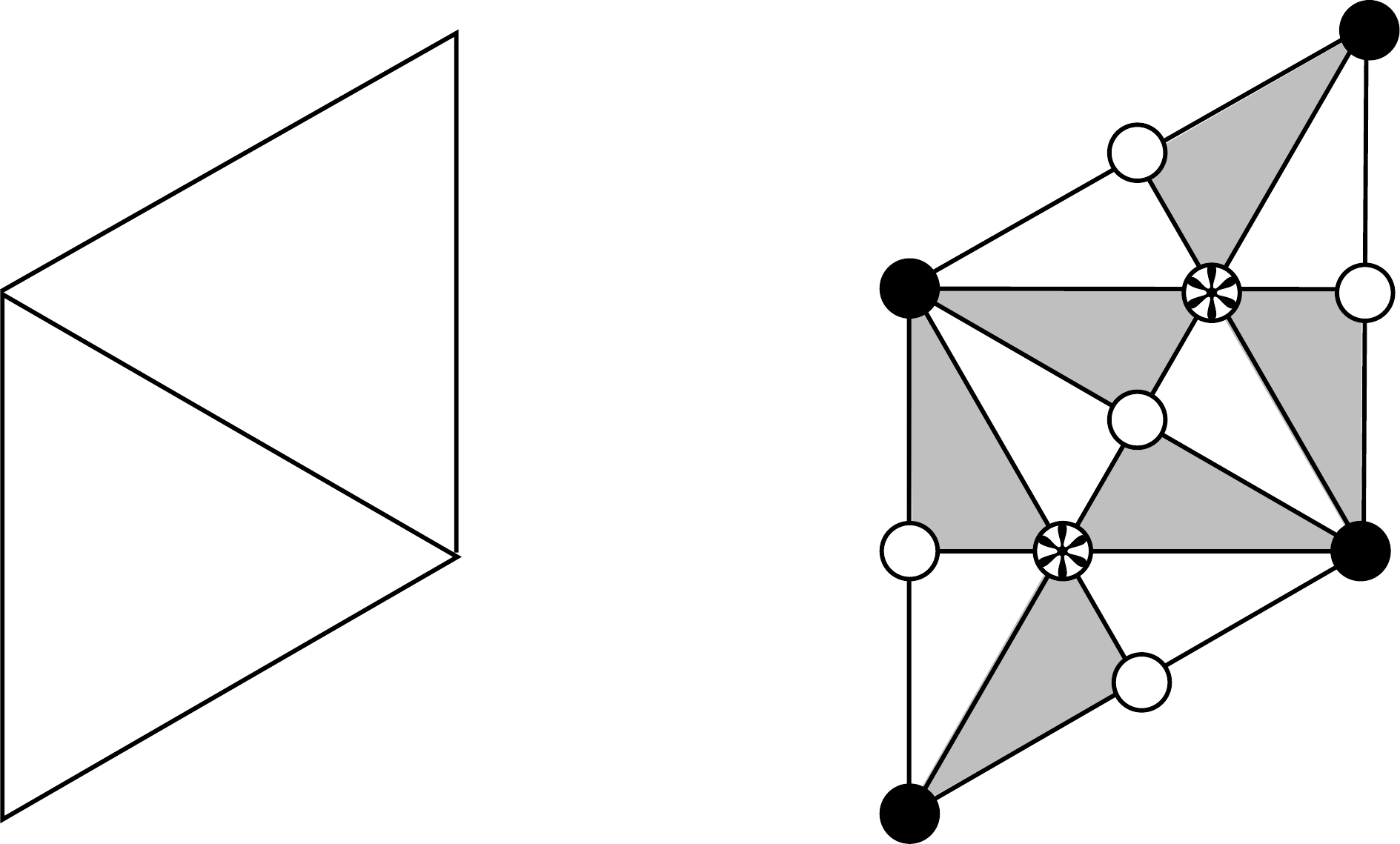}
\end{center}
\caption{\label{fig:barycentric}Any equilateral triangulation
can be refined by barycentric subdivision into a new equilateral triangulation that is bipartite and 3-coloured, as described after the proof of Proposition~\ref{prop:equivalence}.}
\end{figure}
     
  \begin{lem}[Equilateral triangulations and reflections]\label{lem:reflection}
     A generalised triangulation of $X$ 
     is equilateral if and only if the two 
     triangles adjacent to a given edge are related by reflection.      
     That is, suppose that the triangles
      $T$ and $\tilde{T}$ are both adjacent to an edge $e$.
      Then there exists an antiholomorphic homeomorphism
      $\iota\colon T\to\tilde{T}$ that fixes $e$ pointwise and maps the third vertex
      of $T$ to the corresponding vertex of $\tilde{T}$. 
  \end{lem}
  \begin{proof}
       Let $e$, $T$ and $\tilde{T}$ be as in the statement, and let
       $\phi_T$ and $\phi_{\tilde{T}}$ be as defined above. Suppose that $\phi_{\tilde T}|_e = R \circ \phi_{T}|_e$, where
     $R$ is a reflection symmetry of $\Delta$. Then 
     \[ \iota \defeq \phi_{\tilde{T}}^{-1} \circ R \circ \phi_T \]
     is an antiholomorphic bijection as in the statement of the observation. 
     
     Conversely, suppose $\iota$ is such a bijection. Then
        $R \defeq \phi_{\tilde{T}}\circ \iota\circ \phi_{T}^{-1}$ is an 
        antiholomorphic automorphism of the triangle $\Delta$, mapping vertices 
        to vertices. Thus $R$ is a reflection symmetry of $\Delta$, as required.
  \end{proof}

  \begin{rmk}\label{rmk:rigidity}
    It follows from the Schwarz reflection principle that, if a reflection 
      $\iota\colon T\to\tilde{T}$ as above exists, then $\iota$ and hence $\tilde{T}$ are
      uniquely determined by $T$. In particular, an equilateral triangulation
      $\mathcal{T}$ is uniquely determined by any given triangle $T\in\mathcal{T}$.
  \end{rmk}
     
   The equivalence of Theorems~\ref{thm:triangulation} and~\ref{thm:belyi} is
     a consequence of the following fact.
     
    \begin{prop}[Triangulations and Belyi functions]\label{prop:equivalence}
      A Riemann surface $X$ is equilaterally triangulable if and only if
         there is a Belyi function on $X$. 
    \end{prop}
    \begin{proof}
       Proposition~\ref{prop:equivalence} is well-known in the compact case;
        see~\cite{voevodskiishabat}, and 
        the proof in the general case is
         the same~\cite[\S1.3]{bowersstephenson}.
          For the reader's convenience, we present it briefly. 
         First suppose that $f\colon X\to \Ch$ is 
         a Belyi function. Consider the generalised triangulation of the sphere 
         into two triangles corresponding to the upper and lower half-plane, with 
         vertices at $1$, $-1$ and $\infty$. By the Schwarz reflection principle
         and Lemma~\ref{lem:reflection}, this
         triangulation is equilateral. Since the critical values of $f$ are at the
         vertices of the triangulation, we may lift it to $X$, to obtain a generalised
         equilateral triangulation. As discussed above, a barycentric subdivision 
         leads to a triangulation in the stricter sense, and the proof of the 
         ``if'' direction is complete.

       Now suppose that an equilateral triangulation of the surface $X$ is given. Let
         $\mathcal{T}$ be the corresponding collection of topological triangles,
         with conformal maps $\phi_T\colon T\to \Delta$ for $T\in\mathcal{T}$, as above. 
         Let $\psi\colon \Delta \to \D$ be the conformal isomorphism 
           that fixes $0$ and $1$, and
         consider the function 
            \[ f\colon X\to \Ch; \colon z\mapsto F_3(\psi(\phi_T(z))) \qquad (z\in T),\] 
           where $F_3$ is the degree $6$ rational map 
           \[
          F_3(z) := \frac{1}{2}(z^3+z^{-3}).  \]

       Let $\rho$ denote rotation by {60{\degree}} around $0$,
        and let $\sigma$ denote  
         complex conjugation. 
        Observe that $\psi$ commutes with both operations,
         and that $F_3 \circ \rho = F_3 \circ \sigma = F_3$ on $\partial \D$. 
          The group of symmetries of $\Delta$ is generated by $\rho$ and $\sigma$, 
          and thus 
          $f$ is indeed a well-defined holomorphic function on $X$. Clearly $f$ is
          a branched covering with no critical values outside of $-1$, $1$ and $\infty$; so
          $f$ is a Belyi function.
    \end{proof}
    \begin{rmk}\label{rmk:3colouring}
     The generalised 
       equilateral triangulation obtained from the Belyi function $f$ in the above 
       proof is \emph{3-colourable}: Its vertices may be coloured with the three colours
       $\{-1,1,\infty\}$ in such a way that adjacent vertices have different colours. 
       Conversely, suppose $\mathcal{T}$ is a generalised equilateral triangulation
       together with a 3-colouring of its vertices; let us call this a \emph{3-coloured}
       triangulation. Then the three vertices of any triangle $T\in \mathcal{T}$ may
       be coloured with the three different colours $-1$, $1$ and $\infty$, and 
       we may map $T$ 
       conformally to either the upper or lower half-plane in such a way that 
       each vertex corresponds to the point indicated by its colour.
       By Schwarz reflection the collection of these conformal maps
       extends to a Belyi function on $X$. Hence the Belyi functions on $X$ are
       in one-to-one correspondence with the 3-coloured generalised equilateral
       triangulations on $X$. 
       
      Not every equilateral triangulation $\mathcal{T}$ 
       (generalised or otherwise) can be 3-coloured;
        consider, for example, the triangulation of the sphere into four 
        congruent spherical equilateral triangles. However, the barycentric 
        subdivision of $\mathcal{T}$ is always $3$-colourable; indeed, we may 
        mark the original vertices with the colour $1$, the new vertices added on
        existing edges with the colour $-1$, and the vertices added in each face 
        with $\infty$ (Figure~\ref{fig:barycentric}). This yields precisely the triangulation corresponding
        to the Belyi function in the ``only if'' direction of 
        Proposition~\ref{prop:equivalence}.
    \end{rmk}

\begin{figure}
\begin{center}
\subcaptionbox{The complex plane $\C$\label{fig:triangular-lattice}}{\includegraphics[height=0.17\textheight]{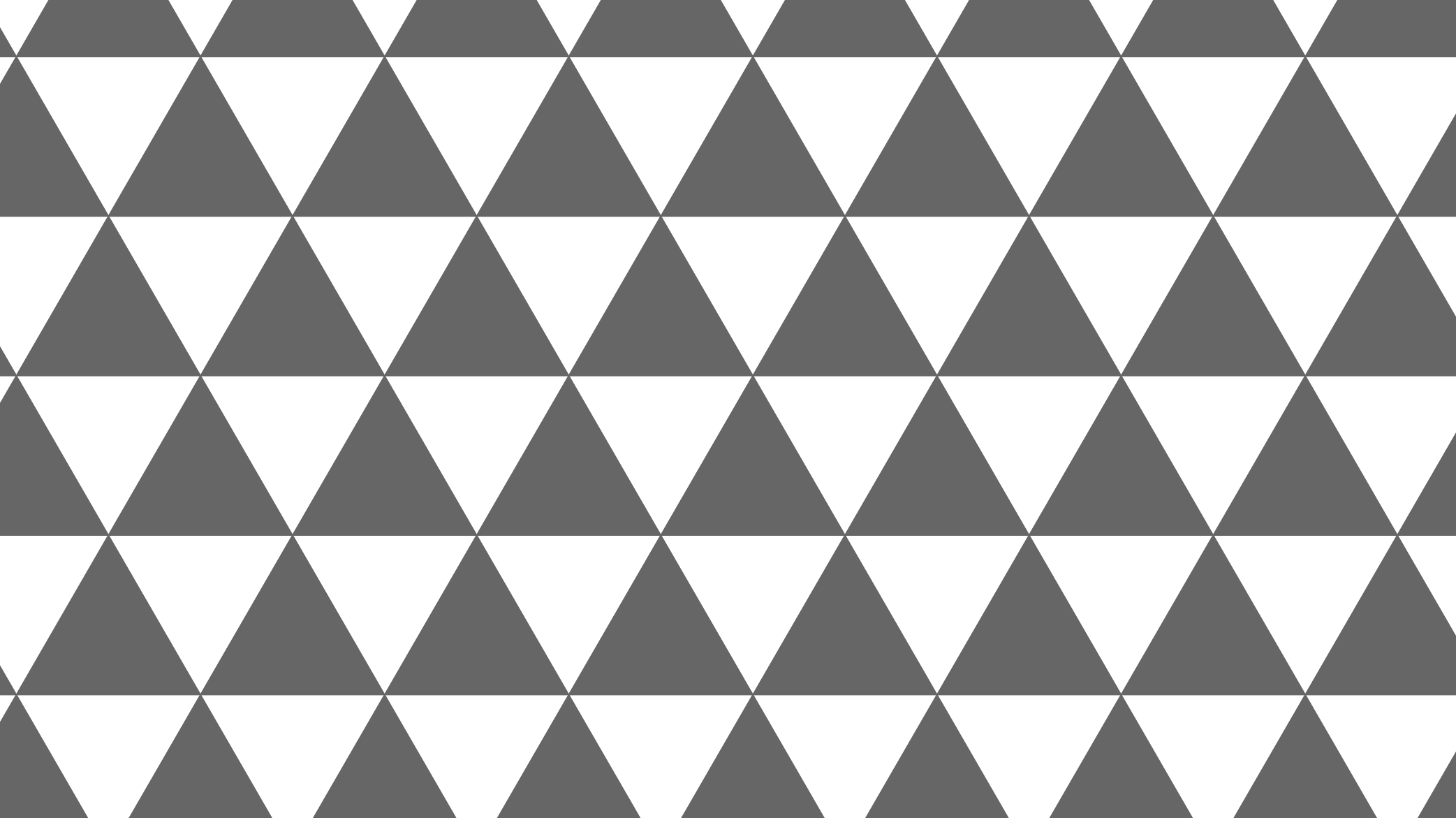}}%
\hfill%
\subcaptionbox{The punctured plane $\C^*$}{\includegraphics[height=0.17\textheight]%
{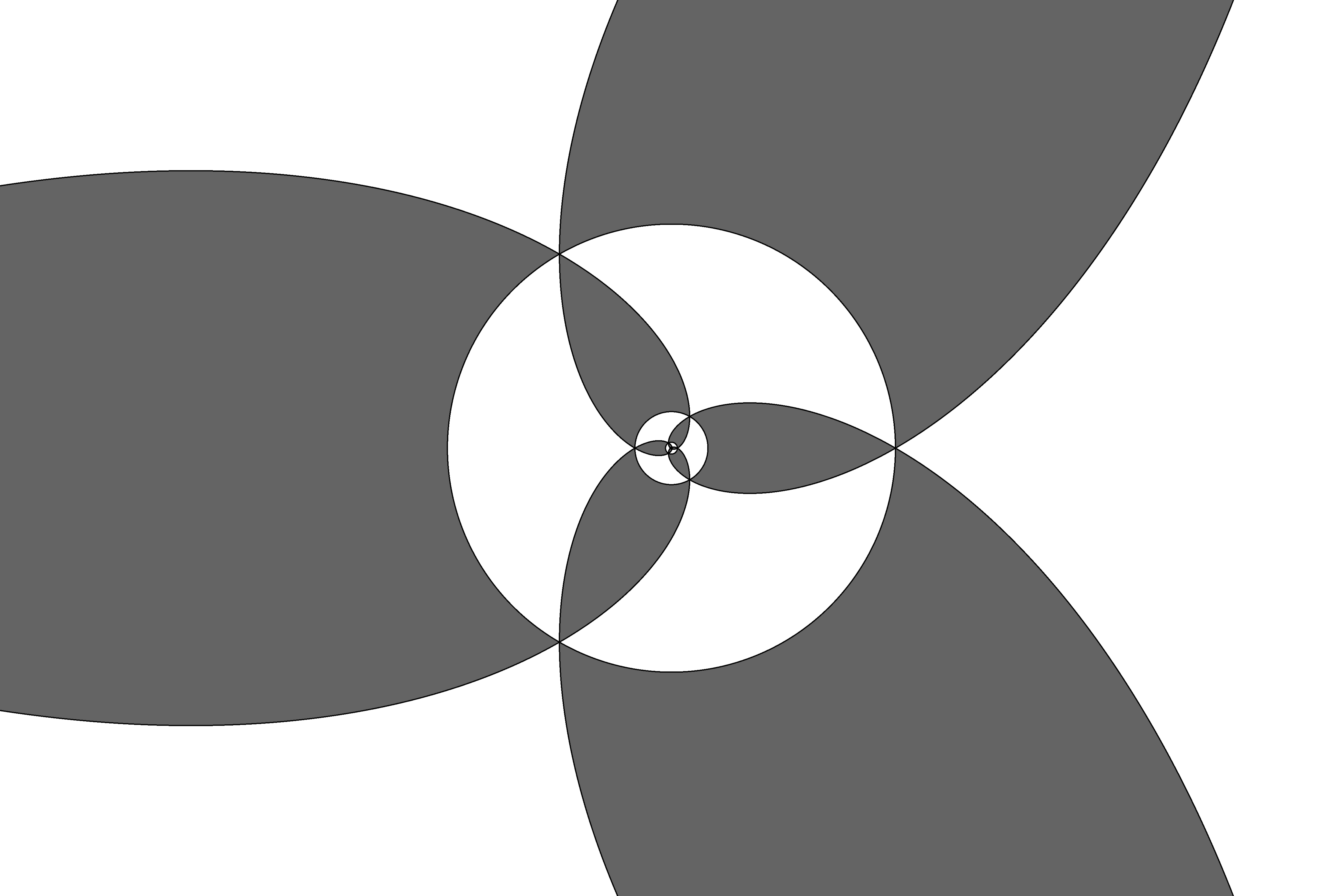}}
\end{center}\mbox{}\\ \vspace{.5cm} 
\begin{center}
\subcaptionbox{The three-punctured sphere\label{fig:3-punctured}}%
{\includegraphics[height=0.26\textheight]{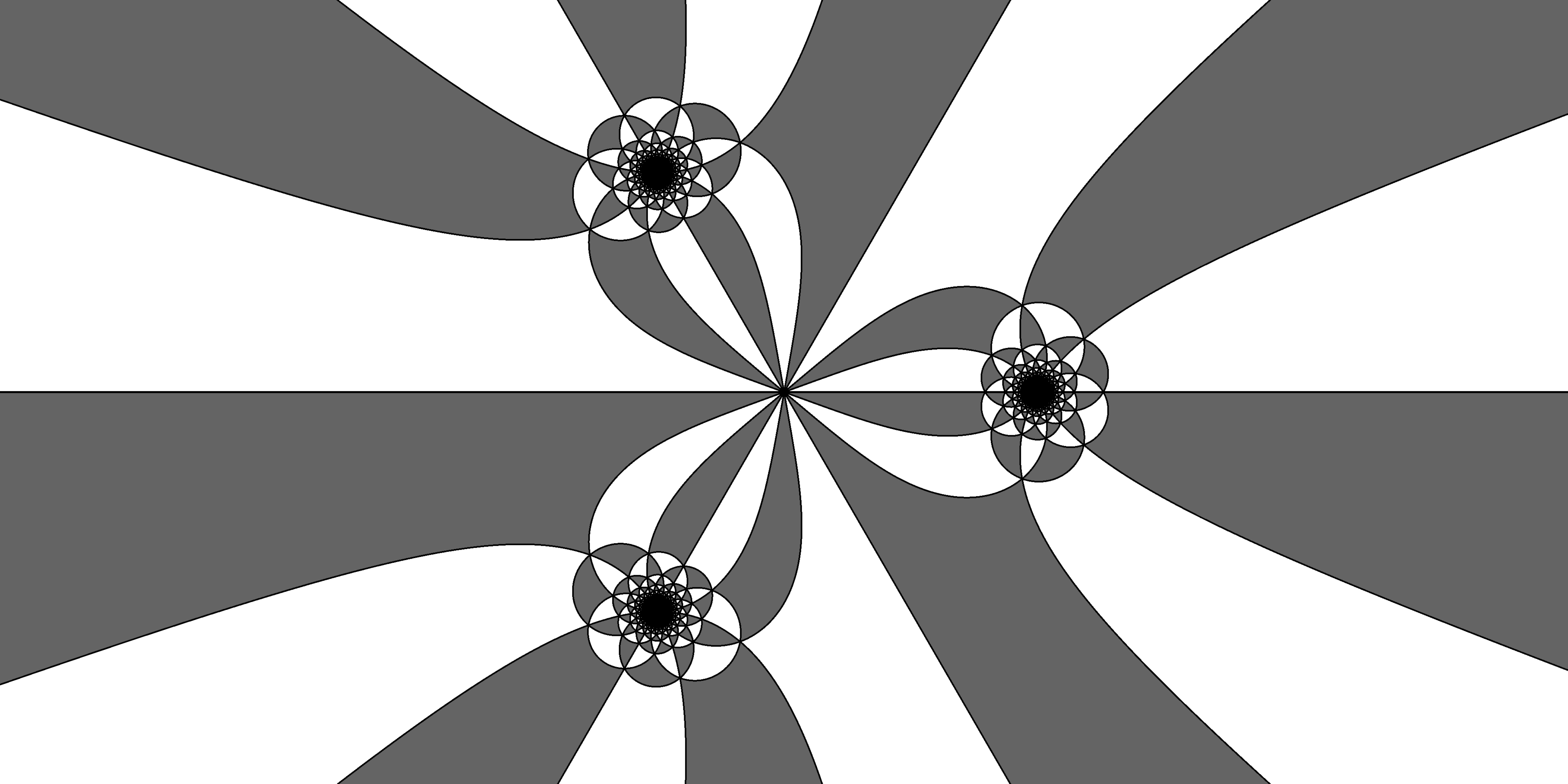}}%
\end{center}\mbox{}\\ \vspace{.5cm}
\begin{center}
\subcaptionbox{The unit disc $\DD$\label{fig:disc}}{\includegraphics[height=0.3\textheight]{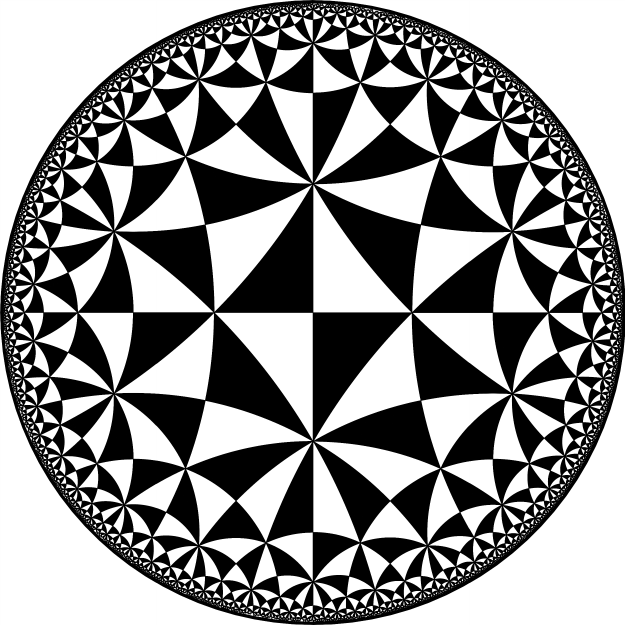}}\hfill
\subcaptionbox{The punctured disc $\DD^*$\label{fig:punctureddisc}}{\includegraphics[height=0.3\textheight]{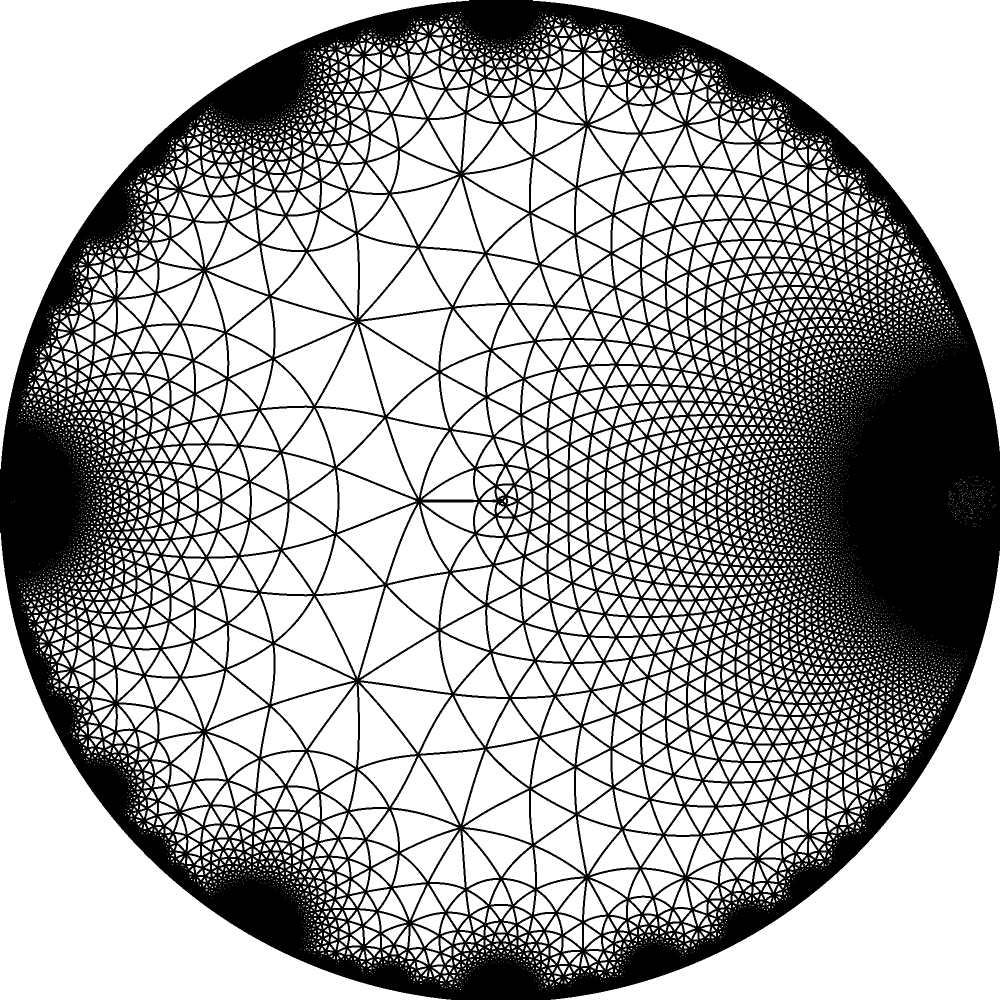}}
\end{center}
\caption{\label{fig:trivialcases}Equilateral triangulations of non-compact Riemann surfaces with trivial moduli spaces.}
\end{figure}

 \subsection*{Elementary cases of Theorem~\ref{thm:triangulation}}    
    The triangular lattice, which tesselates the
      plane into equilateral triangles, provides an equilateral
      triangulation of both the plane and the bi-infinite cylinder; i.e., the punctured
      plane. This triangulation is $3$-colourable; 
      the corresponding Belyi function is elliptic, and can
      be described as the universal orbifold covering map of the sphere with signature
      $(3,3,3)$; see~\cite[Appendix~E]{milnordynamics}. 
      
   The unit disc $\DD$ is equilaterally triangulated by the classical hyperbolic triangle 
     groups. We may obtain an equilateral triangulation of the punctured disc $\DD^*=\DD\setminus\{0\}$ as follows. 
     The \emph{Klein $j$-invariant} $j\colon \HH^+\to\C$ is a branched covering map from the 
     upper half-plane $\HH^+$ to the complex plane which is invariant under the modular group and has
     only two branched values, which we may arrange to be $0$ and $1$. In particular,
     $j(\zeta +1)=j(\zeta)$ for all $z\in\HH^+$, and hence $J\colon\DD^*\to\C; z\mapsto j( \log z / (2\pi i) )$ is a well-defined 
     branched covering map with branched values $0$ and $1$.   
     Let $\T$ be a triangulation of the complex plane for which $0$ and $1$ are vertices 
     (for example, the triangular lattice $\T_{\eucl}$ discussed above, chosen such that $[0,1]$ is the edge of one of the triangles). 
     Then the preimage of $\T$ under $J$ is an equilateral triangulation of $\D^*$; see Figure~\ref{fig:punctureddisc}.

    A similar construction leads to triangulations of multiply-punctured spheres. Note that 
      \[  g\colon \Ch\to\Ch; \quad z\mapsto \frac{z^n}{z^n-1} \] 
      is a degree $n$ branched covering of the sphere, branched over $0$ and $1$. 
      The preimage under $g$ of an equilateral triangulation $\T$ of $\Ch\setminus \{1\}$ (for example, 
      the image of the triangular lattice under $z\mapsto (z+1)/z$)
      is an equilateral triangulation
      of the sphere punctured at the $n$-th roots of unity; see 
       Figure~\ref{fig:3-punctured}.

   In particular, the thrice-punctured sphere is equilaterally triangulable, 
     and there exist equilaterally triangulable $n$-punctured spheres for all $n$. 
     However, 
     for $n>3$, we have equilaterally triangulated only 
     one specific member of the moduli space of 
     $n$-punctured spheres, which has positive dimension. We may 
     obtain others by modifying the construction, e.g.\ by using different
     degree $d$ covering maps whose critical values lie in the triangular lattice. 
     Nonetheless, this yields at most countably many different
     surfaces among the uncountably many possible choices. 

\section{Triangulations of hemmed Riemann surfaces}\label{sec:foldingnew}
  Let $U$ be a hemmed Riemann surface, in the sense of Definition~\ref{defn:hemmed}. Our goal in this section
   is to show that 
   there is a triangulation of $U$ that is close to an equilateral triangulation, in 
   a quasiconformal sense. Moreover, in boundary coordinates, the triangulation
   will simply subdivide each boundary circle $\gamma$ into a large
   number $d^{\gamma}$ of equal arcs, where the $d^{\gamma}$ can be chosen
   independently of each other as long as they are sufficiently large. This will later allow us to
   glue together triangulations of different finite pieces of a given Riemann surface. 
   
 To make this statement precise, we use the following notion. 
   \begin{defn}
       Let $d\geq 1$, and let 
     \[ \Xi_d \defeq \{e^{2\pi i j/d}\colon
                      j\in\Z\} \]
        denote the set of all $d$-th roots of unity. We call $\Xi_d$ the 
        \emph{standard partition} of $S^1$ of size $d$; the 
        intervals of $S^1\setminus \Xi_d$ are called the
        \emph{edges} of the partition.
   \end{defn}
   
\begin{prop}[Triangulations on hemmed Riemann surfaces]\label{prop:hemmedtriangulation}
  There are $K_0 >1$, $s_0\geq 6$ and 
    a function $\Deltafn\colon (1,\infty)\to \N$, with the following property. 
  
  Let $U$ be a hemmed Riemann surface with 
     boundary coordinates 
   \[ \phi^{\gamma} \colon \AA_-(R^{\gamma})  \to A^{\gamma}. \] 
  Denote the set of all boundary curves by $\Gamma$, and let $\rho$ be a conformal metric on $U$. 
    Fix
     $d^{\gamma}\geq \Deltafn(R^{\gamma})$ for each $\gamma\in\Gamma$, and let $\eta>0$.
  
  Then there is a homeomorphism
    $g$ from $\overline{U}$ to a finite equilateral surface-with-boundary $E$ such 
    that the following hold. 
  \begin{enumerate}[(a)] 
    \item Every vertex of $E$ is incident to at most $s_0$ edges.\label{item:angles} 
    \item For $\gamma\in\Gamma$, 
         the map $g\circ \phi^{\gamma}\colon S^1\to g(\gamma)$ maps each edge of
           $\Xi_{d^{\gamma}}$ to a boundary edge of $E$ in length-respecting fashion.\label{item:length-respecting} 
    \item $g$ is $K_0$-quasiconformal on $U$.\label{item:gqc}
    \item The  complex dilatation of $g$ is supported on the
    union $\bigcup_{\gamma} A^{\gamma}$, together with a set
  that has measure at most $\eta$ with respect to the metric $\rho$. 
   \end{enumerate} 
 \end{prop} 
 \begin{remark}[Remark 1]
    A map \emph{respects length} if it changes distances by a constant
      factor~\cite[\S4]{BishopFolding}. 
In other words, if $e$ and $f$ are two rectifiable arcs, then a homeomorphism 
$\psi\colon e \to f$ is 
length-respecting if, for any measureable set $X \subset e$, we have 
$\ell(\psi(X)) = \ell(X) \cdot \ell(f)/\ell(e)$, where $\ell$ denotes 
arc-length measure. There are exactly two such maps between $e$ and $f$, 
depending on how the endpoints of $e$ are mapped to the endpoints of $f$. 
If $e$ and $f$ have the same endpoints and these are fixed by $\psi$, 
then $\psi$ is unique.
      
    Here length on $E$ is measured with respect to the       
      natural distance inherited from its representation by equilateral triangles. This is a flat conformal
     metric, except possibly for cone singularities at the vertices of the triangulation.

   We may rephrase~\ref{item:length-respecting} more explicitly as follows.  Let $e$ be an edge of $\Xi_{d^{\gamma}}$.
       Then $g(\phi^{\gamma}(e))$ is a boundary edge of $E$; let $T$ be 
        the unique adjacent face. 
        By the definition of an equilateral surface, $T$ is a copy of
     a planar equilateral triangle; in these coordinates, 
      $g\circ\phi^{\gamma}\circ \exp$, restricted to
      a component of $\log e$, is required to be 
      the restriction of a complex affine map.
 \end{remark}
 \begin{remark}[Remark 2]
   It is crucial that the number 
    $d_{\gamma}$ can be chosen arbitrarily large on each
    boundary curve $\gamma$, independently of the choice for the others. 
 \end{remark}
 
  The idea of the proof of Proposition~\ref{prop:hemmedtriangulation} can be summarised as follows.
   \begin{enumerate}[(I)]
    \item By cutting along finitely many essential curves, we may assume that $U$
        has genus $0$, and hence is a subset of the plane.\label{item:reductiontogenus0}
   \item We cover most of the domain $U$ (more precisely, 
       a domain obtained from
      $U$ by removing an annulus contained in $A^{\gamma}$ 
      for each $\gamma\in\Gamma$) with 
      small Euclidean equilateral triangles arranged in a triangular lattice.\label{item:stepII}
    \item We are left with finitely many annuli, one in each $A^{\gamma}$, between $\gamma$
      and a curve consisting of edges taken from the above lattice. We interpolate
       between the 
      partitions of these two boundaries by a triangulation that has \emph{bounded geometry}, and hence is
      quasiconformally equivalent to an equilateral triangulation.\label{item:annulitriangulation}
   \end{enumerate}
    
 We begin by developing a number of elementary lemmas 
   that will be used in the final step~\ref{item:annulitriangulation}. 
   The first goal is to show that, when we cover a surface as in
   step~\ref{item:stepII}
    by sufficiently small equilateral triangles, we obtain an equilateral surface-with-boundary with the same number
   of boundary curves, and furthermore the edges of these boundary curves are not close to being perpendicular
   to the original boundaries.
   
  The next lemma makes this idea more precise. For $\eps>0$, 
      let $\mathcal{L}_{\eps}$ be a tiling of the plane by (closed) Euclidean equilateral triangles of
      side-length $\eps$, as in Figure~\ref{fig:triangular-lattice}. 
      If $\gamma\subset\C$ is a smooth Jordan curve, define $\Omega(\gamma,\eps)$ to be 
      the union of all triangles of $\mathcal{L}_{\eps}$ that intersect $\gamma$; see Figure~\ref{fig:DefnOmega}.
      By a \emph{segment} of $\partial \Omega(\gamma,\eps)$ we mean an edge of a triangle in the triangulation
      $\mathcal{L}_{\eps}$ that lies on $\partial \Omega(\gamma,\eps)$.

\begin{figure}
  \begin{overpic}[percent, width=.6\textwidth]{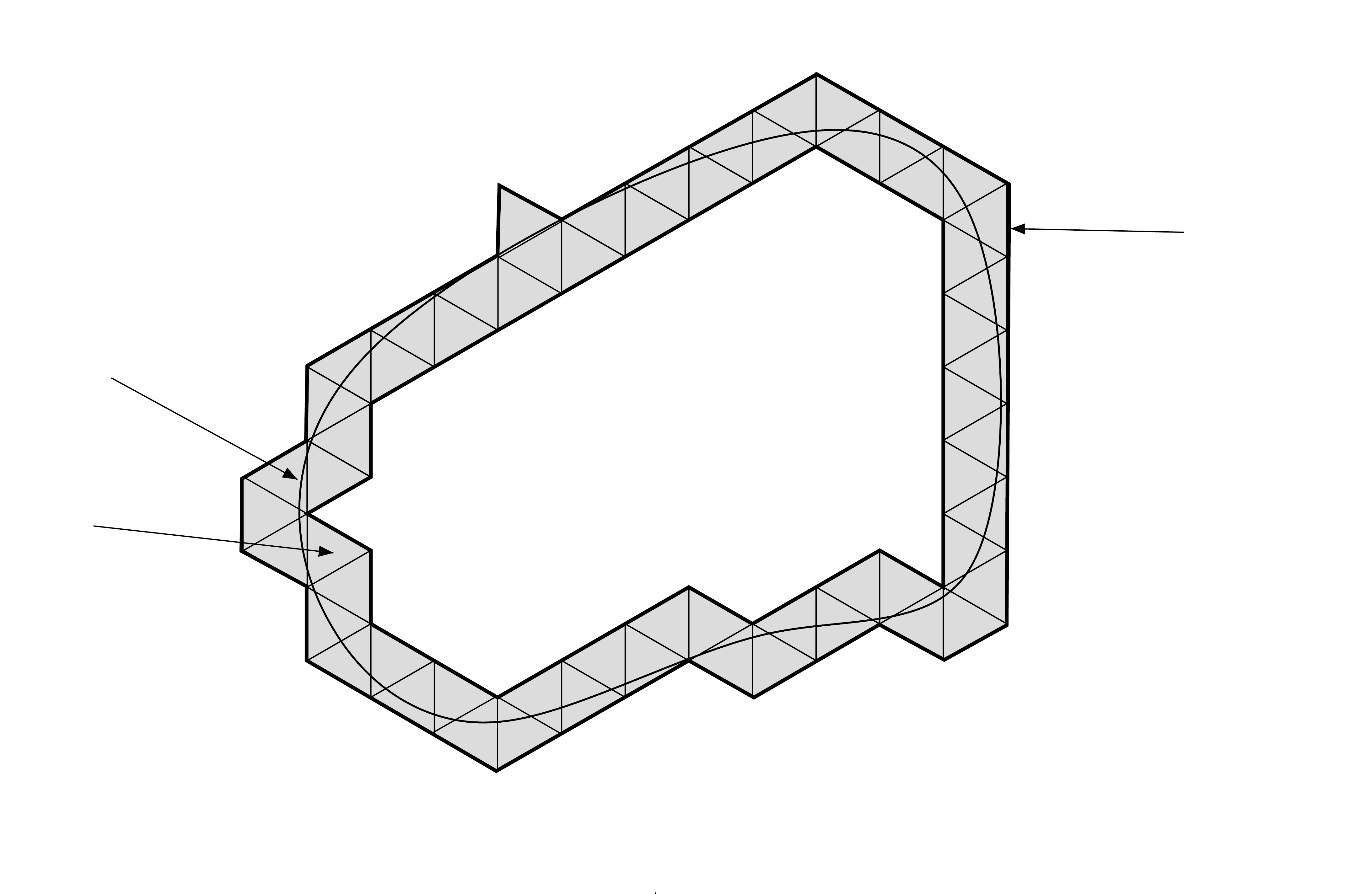}
  \put(89,47){$\partial \Omega(\gamma,\eps)$}
  \put(5,37){$\gamma$}
  \put(2,23){$\Omega(\gamma,\eps)$}
  \end{overpic}\\[2ex]
  \caption{\label{fig:DefnOmega}
        Definition of $\Omega(\gamma,\eps)$. 
        }
\end{figure}

 \begin{lem}[Boundary approximation]\label{lem:boundaryapprox}
   Let $\gamma\subset\C$ be a smooth Jordan curve. If $\eps>0$ is sufficiently small, then the following hold.
    \begin{enumerate}[(a)]
     \item 
       The equilateral surface-with-boundary $\Omega(\gamma,\eps)$ is a closed topological annulus that is bounded by two Jordan curves.
    \item Suppose that $I$ is a segment of $\partial \Omega(\gamma,\eps)$. If $x\in\gamma$ is within
    distance $\eps$ of $I$, then the tangent line of $\gamma$ at $x$ makes an angle at most $(5/12)\pi$ with
    the line containing $I$. (By convention, two parallel lines make an angle of $0$ with each other.)\label{item:angle}
    \end{enumerate}
 \end{lem}
 \begin{proof} 
  By definition $\Omega(\gamma,\eps)$ is an equilateral surface-with-boundary. 
Fix $\delta>0$ sufficiently small 
  (as we see below, $\delta < \pi/12$ suffices). 
Since $\gamma$ is a smooth Jordan curve, for sufficiently small
$\eps\in (0,\diam(\gamma))$ the following holds. If $x,y\in \gamma$ 
and $\lvert x - y\rvert < 3\eps$, then the tangent lines of $\gamma$ at $x$ and $y$ 
differ by an angle that is less than $\delta$. Furthermore, if $\eps$ is small enough, then 
$\gamma\cap D(x,3\eps)$ is an arc for all $x\in\gamma$. 
This arc is then contained in the union of two sectors of opening angle $2\delta$, 
 centered at the tangent line to $\gamma$ at $x$. 

 Suppose that $I$ is a segment of $\partial \Omega(\gamma,\eps)$. Then $I$ is a side of a triangle
  $T$ of $\mathcal{L}_{\eps}$ that hits $\gamma$, while the other triangle that
  has $I$ as an edge does not. In particular, $I\cap \gamma =\emptyset$. Fix a point 
  $x\in \partial T\cap \gamma$. 
   
 Let $\theta$ be an isometry of the plane that maps $x$ to $0$ and the tangent line to $\gamma$ 
   through $x$ to the real axis. Then $\theta$ maps the arc $\gamma\cap D(x,3\eps)$ to points
   with arguments in $(-\delta,\delta)$ and $(\pi-\delta,\pi+\delta)$. Since $I$ does not intersect
   $\gamma$, this means that we may choose $\theta$ (composing with a rotation or reflection) 
      such that all points of $\theta(I)$ have
   	arguments between $-\delta$ and $\pi+\delta$. (See the left-hand side of 
   	Figure~\ref{fig:NotPerp}.) If $\theta(I)$ made an angle 
   	greater than $(5/12)\pi$ with the real axis, then the lowest side of 
   	$\theta(T)$ would make an angle greater than
   	$(5/12)\pi - \pi/3 = \pi/12 > \delta$ with the real axis. But this easily implies
   	that $\theta(T)$ lies entirely in the sector at arguments in $(-\delta,\pi+\delta)$, contradicting the
   	fact that $0\in \theta(T)$. (See the right-hand side of Figure~\ref{fig:NotPerp}.)  This concludes the proof
   	of~\ref{item:angle}. 
   	
  If $T\in\mathcal{L}$ is a triangle that intersects $\gamma$ in some point $x$, then triangles
   adjacent to $T$ can intersect $\gamma$ only in the segment $\gamma\cap D(x,2\eps)$, which is
   again contained in the union of two thin sectors. 
   It follows that every segment of $\partial \Omega(\gamma,\eps)$ is adjacent
   to exactly two other such segments. So every connected component of $\partial \Omega(\gamma,\eps)$ is
   a Jordan curve. The same fact implies that every boundary component intersects a $2\eps$-neighbourhood
   of every point of $\gamma$. It follows that there are only two such components, one inside and one
   outside of $\gamma$. 
\end{proof}

\begin{figure}
\begin{center}
  \begin{overpic}[percent,grid=false,height=2in]{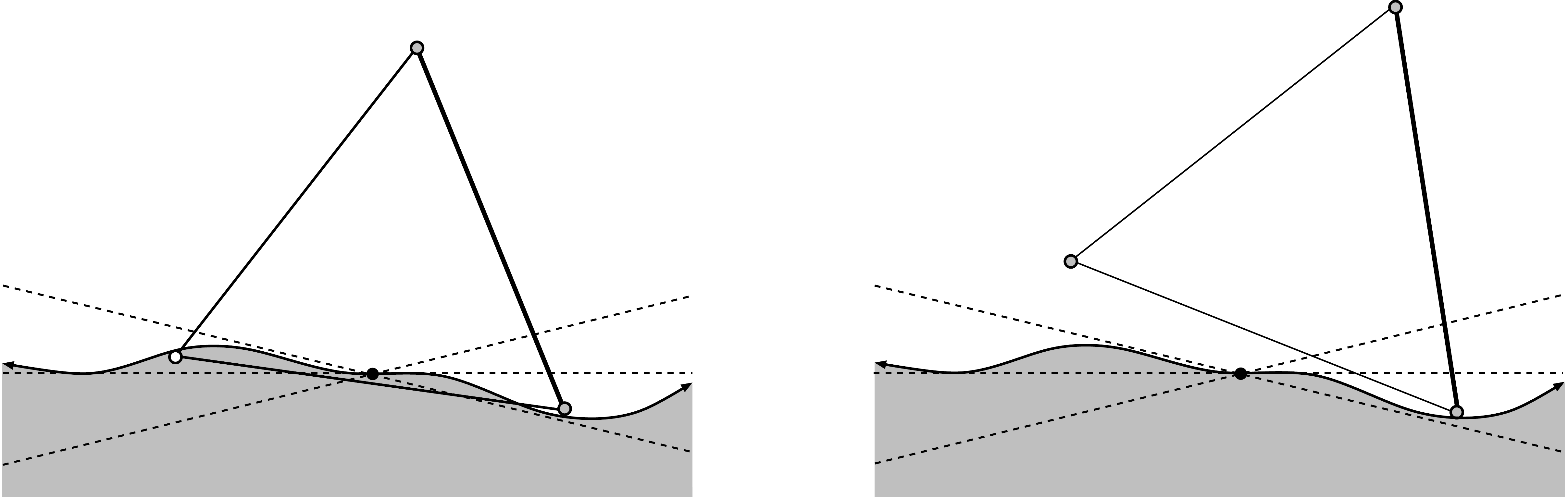}
  \put(22.8,9){\Large $0$}
  \put(33,13.7){\Large$\theta(I)$}
  \put(78.5,4){\Large $0$}
  \put(92.5,13.7){\Large$\theta(I)$}
  \end{overpic}
\end{center}
	\caption{\label{fig:NotPerp}  
	Proof that $\alpha(\gamma)$ is not close to being perpendicular to $\gamma$.}
\end{figure}

We wish to apply Lemma~\ref{lem:boundaryapprox} 
  to the core curve 
 of an annulus $A$ (where $A$ is one of the hems $A^{\gamma}$ of the Riemann surface $U$ in Proposition~\ref{prop:hemmedtriangulation}). 
  More precisely, let 
 $\Phi\colon \AA_+(R)\to A\subset\C$ be a conformal isomorphism (for some $R>1$). 
  Let $\gamma$ be the 
  connected component of the boundary of $A$ (in $\Ch$) 
  corresponding to the limit set of 
  $\Phi(z)$ as $\lvert z\rvert \to 1$. (In our applications, 
  $\Phi$ extends analytically to $\AA_+(R)\cup\overline{\DD}$, and in particular
  $\gamma  = \Phi(S^1)\subset\C$ is an analytic curve. However, we do not require this
  property here.) Let
  $\tilde{\gamma}$ be the core curve of $A$; i.e., $\tilde{\gamma} \defeq \Phi(\partial D(0,\sqrt{R}))$. Suppose
  that 
  $\eps$ is small enough so that the conclusion of
  Lemma~\ref{lem:boundaryapprox} holds for $\Omega(\gamma,\eps)$, and let 
  $\alpha$ be the connected
  component of $\partial \Omega(\gamma,\eps)$ that separates $\gamma$ from $\tilde{\gamma}$. 
  Our goal is to relate the annulus bounded by $\gamma$ and $\alpha$ to that
  bounded by $\gamma$ and $\tilde{\gamma}$ using a suitable quasiconformal homeomorphism. 
 
  It is useful to lift our picture by the exponential map. 
  Set
\[ \Sigma_+ \defeq \exp^{-1}(\Phi^{-1}(\alpha)); \]
see Figure~\ref{fig:DefnS}.
Let $V_+$ and $E_+$ denote the sets of vertices and edges
of $\Sigma_+$; that is, the preimages of the vertices and edges of the polygonal curve $\alpha$ under $\Phi\circ\exp$. 
Also let $\Sigma_-$ denote the imaginary axis, and let $S$ denote the domain bounded by $\Sigma_-$ on the left and $\Sigma_+$ on the right. So 
$\Phi\circ \exp$ maps the topological strip $S$ 
to the annulus bounded by $\alpha$ and $\gamma$ as a 
universal covering map. Also define the vertical strip 
\[ \tilde{S} \defeq \{ x + iy\colon 0 < x < \rho \defeq (\log R)/2 \}, \]
and let $\tilde{\Sigma}_+ = \rho + i\R$ be its right boundary. 
Note that $\Phi(\exp(\tilde{S}))$ is bounded by $\gamma$ and $\tilde{\gamma}$, and in particular $S\subset \tilde{S}$.  
(See Figure~\ref{fig:DefnS}.) 

All of the above objects depend on $\eps$, $R$ and $\Phi$. We show the following.

\begin{figure}
\begin{center} 
 \begin{overpic}[percent,grid=false,height=2in]{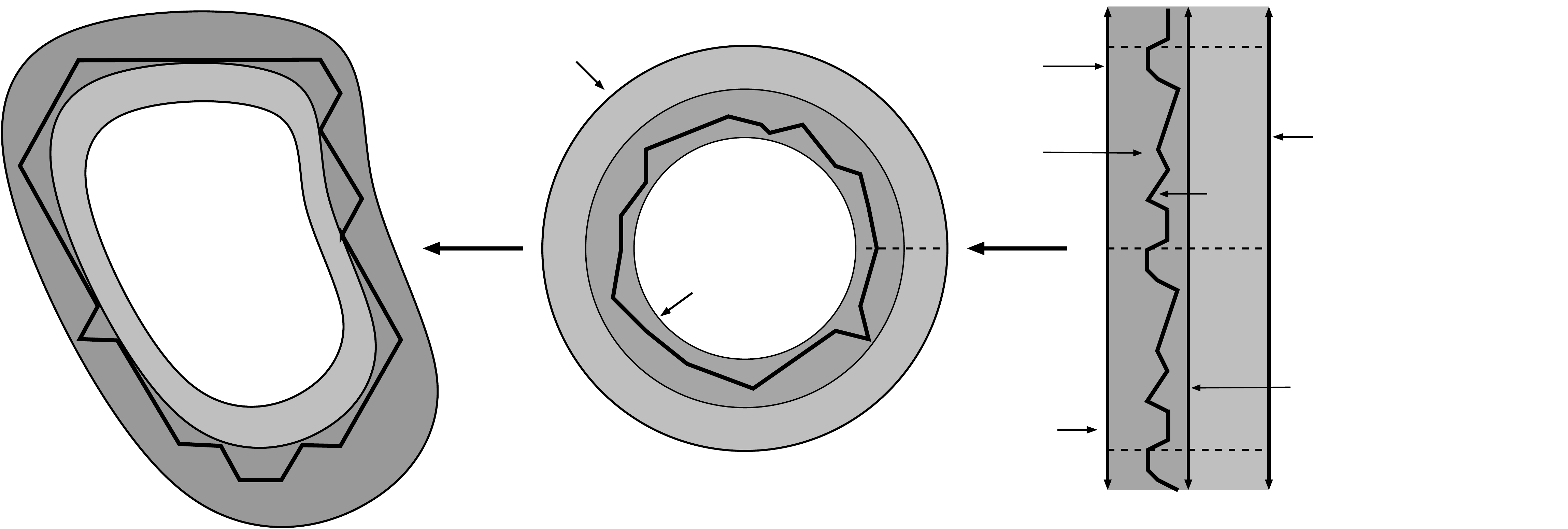}
	 \put(27,31){$\{|z|=R^\gamma$\}}
  \put(29,19){$\Phi$}
  \put(42,16){$\{|z|=1\}$}
  \put(63,19){exp}
  \put(63,29){$\Sigma_-$}
  \put(63,23.5){$\Sigma_+$}
  \put(58.0,5.8){$\{x=0\}$}
  \put(71.5,13){$S$}
  \put(78,20.5){$\widetilde S$}
  \put(84,24.5){$\{x=2\rho\}$}
  \put(83,8.5){$\widetilde \Sigma_+=\{x=\rho\}$}
  \end{overpic}
\end{center} 
  \caption{\label{fig:DefnS}
	The definition of the strips $S$ and $\tilde{S}$.
  }
\end{figure}

\begin{lem}[Quasiconformal map straightening $S$]\label{lem:qcstripmap}
There are universal constants $K_1>1$ and $\lambda_1>1$ with the following property. Let $R>0$, let 
$\Phi\colon \AA_+(R)\to A\subset\C$ be a conformal isomorphism, and let 
$\eps>0$ be sufficiently small. Then (using the notation introduced above) there is 
a $K_1$-quasiconformal homeomorphism 
$ \Psi \colon  S\to \tilde{S}  $
such that:
\begin{enumerate}[(a)]
\item $\Psi(z+2\pi i) = \psi(z)+2\pi i$ for all $z$;\label{item:commuteswithtranslation}
\item $\Psi(z)= z$ for $z\in \Sigma_-$;\label{item:identityonleft}
\item $\Psi(V_{+})$ contains $\rho=(\log R)/2$;\label{item:real}
\item the length of the intervals in $\tilde{E}_+\defeq \{\Psi(e)\colon e\in E_+\}$
is bounded above by $\rho$;\label{item:lengthupperbound}
\item the lengths of adjacent intervals in $\tilde{E}_+$ differ at
most by a factor of $\lambda_1$;\label{item:boundedgeometry}
\item Let $\tilde{e}\in \tilde{E}_+$. Then 
$\Phi\circ \exp \circ \Psi^{-1}$ respects length  on 
$\tilde{e}$.\label{item:lengthrespectinginneredge}
          \end{enumerate}
\end{lem}

The basic idea of the proof of the lemma is straightforward: We first map each horizontal line of $S$ linearly to the corresponding
horizontal line of $\tilde{S}$. Lemma~\ref{lem:boundaryapprox} ensures that this gives a quasiconformal map 
$h\colon S\to\tilde{S}$ whose maximal
dilatation
is bounded by a universal constant. (If $\eps$ is small enough, the edges of $\Sigma_+$ are almost straight line 
segments, by Koebe's theorem.) The resulting map satisfies all desired properties
 except~\ref{item:real} and~\ref{item:lengthrespectinginneredge}; the latter two are easily ensured by further 
 quasiconformal changes arbitrarily close to the identity. 
  To provide the details, we require the following elementary result.
 
\begin{lem}[Bi-Lipschitz map between strips] \label{lem:biLip_est} 
Suppose $0 < r < 1$ and $L < \infty$ and suppose 
$\chi\colon {\mathbb R}  \to  {\mathbb R}$ is $L$-Lipschitz and satisfies 
$ r \leq \chi(y)   \leq  1/r $ for all $y\in\R$.
Set $S_1\defeq\{ (x,y)\in\R^2\colon 0< x < \chi(y)\}$ and 
	$S_2\defeq\{ (x,y)\in\R^2\colon 0 < x < 1\}$. 
Then the map $\Theta\colon S_1 \to S_2$ defined by $\Theta(x,y) = 
(x/\chi(y),y)$ is bi-Lipschitz and hence quasiconformal. The 
bi-Lipschitz constant and the maximal dilatation of
$\Theta$ are bounded by a constant depending only on $r$ and $L$. 
\end{lem}
\begin{proof}
	Observe that 
\begin{eqnarray*}
	|\Theta(x,y) - \Theta(a,b)| 
	&\leq & |y-b| + \left|\frac{x-a}{\chi(y)}\right| 
	+ \left|\frac {(\chi(y)-\chi(b))a} {\chi(y) \chi(b) }\right| \\
	&\leq & |y-b| + \frac{\lvert x-a\rvert }{r}
	+ \frac {L| y-b| }{r^3} ,
\end{eqnarray*} 
and hence $\Theta$ is Lipschitz with a constant 
depending only on $r$ and $L$. The inverse of 
$\Theta$ is defined on $S_2$ by the analogous formula for 
$\psi = 1/\chi$, which satisfies $r\leq \psi \leq 1/r$
and is $L/r^2$-Lipschitz. So the same calculation as above shows that the 
inverse is also Lipschitz. Thus $\Theta$ is bi-Lipschitz.

Since $\Theta$ is  bi-Lipschitz with constant depending only on 
$r$ and $L$, it is also quasiconformal with a constant
depending only on these two quantities. Alternatively, 
one can  compute $\Theta_{\overline z}$ and  $\Theta_z$ and 
prove directly that 
$|\mu| = |\Theta_{\overline z}/\Theta_z|$ is bounded by a constant strictly less than $1$.
\end{proof}

\begin{cor}[Map of horizontal segments]\label{lem:bilipschitzstripmap}
 In the setting of Lemma~\ref{lem:qcstripmap}, if $\eps>0$ is 
  sufficiently small, then 
  there is a bi-Lipschitz homeomorphism $h\colon S \to \tilde S$
  that maps horizontal segments to horizontal segments at the  
  same height, that is the identity on $\Sigma_-$, and whose bi-Lipschitz constant
  (and hence maximal dilatation) is bounded by a universal constant.  
  
  Moreover, set $\hat{E}_+\defeq \{h(e)\colon e\in E_+\}$. Then (for sufficiently small $\eps$) 
    the length of the intervals in $\hat{E}$ is bounded above by $\rho$, and the lengths of
    adjacent intervals in $\hat{E}_+$ differ at most by a universal multiplicative constant $\lambda_1$. 
\end{cor} 
\begin{proof} 
If $\eps$ is sufficiently small, then 
the map  $\exp^{-1} \circ \Phi^{-1}$ is conformal on 
a disk $D(x,M \eps)$ around any point $x \in \alpha$,
where $M$ is fixed, but may be taken as large as we wish if 
$\eps$ is sufficiently small. If $M$ is large enough then by 
Koebe's distortion theorem,  the map $\exp^{-1} \circ \Phi^{-1}$ 
is as close to a complex affine map on $D(x, 2 \epsilon)$ as we wish. In particular, 
the angle between any segment in $\alpha(\gamma)$  and  the tangent 
line of $\partial \tilde U$ at a point $y \in \partial \tilde U \cap 
D(x,2 \eps) $ is nearly preserved by the map. 
Since $\partial \tilde U$ is 
mapped to a vertical line, we deduce that  $\alpha$ is mapped to a
piecewise analytic curve 
whose tangents (where they exist) deviate from vertical by at most $(5/12)\pi+o(1) $.
The second term can be made as small as we wish by taking $\eps$ small enough, and hence $M$ large enough. 
In particular, we can choose $M$ and $\eps$ so that this angle is less than $ (11/24)\pi$. 
It follows that $\Sigma_+$ is the graph of an $L$-Lipschitz function, where $L = \tan((11/24)\pi)$ is a universal constant. 

For $y\in\R$ let $\chi(y)\leq 1$ be the number such 
 that $\rho\cdot (\chi(y)+iy)$ is the unique point of 
$\alpha$ at imaginary part $\rho y$. As we just saw, $\chi$ is $L$-Lipschitz.
Clearly also 
$\chi(y)\geq 1/2$ for all $y\in\R$ if $\eps$ is small enough. 
Let $\Theta$ be the map from Lemma~\ref{lem:biLip_est} and set
$h(z)\defeq \rho\cdot \Theta(z/\rho)$. Then $h$ is a homeomorphism from
 $S$ to $\tilde{S}$. Since $\Theta$ is bi-Lipschitz with universal bi-Lipschitz constant and universally bounded maximal dilation, the same is true of
  $h$. 
  
 Moreover, the lengths of the edges of $\Sigma_+$ tend to zero as
  $\eps\to 0$, and lengths of adjacent edges are comparable up to a universal factor
  by Koebe's distortion theorem. 
  Since $h$ is bi-Lipschitz, the same is true for the images of these 
  edges under $h$; i.e., the elements of
  $\hat{E}_+$.
\end{proof}

\begin{proof}[Proof of Lemma~\ref{lem:qcstripmap}]
Let $h\colon S\to\tilde{S}$ be the bi-Lipschitz map from Lemma~\ref{lem:bilipschitzstripmap} and 
$\hat{E}_+\defeq \{h(e)\colon e\in E_+\}$. The map already satisfies \ref{item:commuteswithtranslation},
\ref{item:identityonleft}, \ref{item:lengthupperbound} and~\ref{item:boundedgeometry}. We first
discuss how to modify $h$ as to obtain also property~\ref{item:lengthrespectinginneredge} on 
$\hat{E}_+$, by pre-composing $h$ with a suitable quasiconformal homeomorphism.

Let $\hat{e}$ be an element of $\hat{E}_+$. 
Since $\hat{e}$ and $\Phi(\exp(h^{-1}(\hat{e})))$ are both straight line
 segments, there is a unique affine map that maps the former to the latter and
 agrees with $\Phi\circ\exp\circ h^{-1}$ on the endpoints of $\hat{e}$. Define
 $\psi_1\colon \hat{e}\to \hat{e}$ to be the unique homeomorphism such that
 $\Phi\circ \exp \circ h^{-1}\circ \psi_1^{-1}$ agrees with this affine map, and hence
 respects length, on $\hat{e}$. 
 This defines a homeomorphism 
 $\psi_1\colon \tilde{\Sigma}_+\to\tilde{\Sigma}_+$. 
Applying Koebe's theorem again, we see that the derivative of 
 $\psi_1$ on $\tilde{e}$ tends to $1$ uniformly as $\eps\to 0$. 
Extend $\psi_1$ to a map $\psi_1\colon \tilde{S}\to \tilde{S}$ that agrees with the 
identity on $\Sigma_-$ and is affine on each horizontal segment
of $\tilde{S}$. By the above fact on the derivative of $\psi_1$ on $\tilde{\Sigma}_+$, 
the maximal dilatation of the extension tends to $1$ as $\eps\to 0$. 
          
Finally, let $\psi_2\colon \tilde{S}\to\tilde{S}$ be the real-affine
map that is the identity on $\tilde{\Sigma}_-$ and a translation 
on $\tilde{\Sigma}_+$ that maps the point of $h(V_+)$ with
smallest positive imaginary part to $\rho$. As $\eps\to 0$, the maximal
dilatation of this map tends to $1$. Hence, if $K_1>K_2$, then the composition
$\Psi\defeq \psi_2\circ \psi_1\circ h$ is $K_1$-quasiconformal 
for sufficiently small $\eps$. 
          
Since each of the maps 
$\psi_2$, $\psi_1$ and $h$ satisfies~\ref{item:commuteswithtranslation} 
and~\ref{item:identityonleft}, so does $\Psi$. Claim~\ref{item:real} holds
by choice of $\psi_2$. We have $\tilde{E}_+ \defeq \Psi(E_+) = \psi_2(\hat{E}_+)$. Since
$\hat{E}_+$ satisfies~\ref{item:lengthupperbound} and~\ref{item:boundedgeometry}, its translate
$\tilde{E}$ does also. 
Finally,~\ref{item:lengthrespectinginneredge} holds by definition of 
$\psi_1$. This concludes the proof.
\end{proof}

Finally, we require an elementary fact about extending partitions of the boundary of a rectangle to 
 a triangulation of its interior.

  \begin{defn}[Bounded-geometry partition of a rectangle boundary]\label{def:partition}
    Let $R$ be a Euclidean rectangle. By a  \emph{boundary partition} of $R$ we mean
     a finite set $P$ of points on $\partial R$ that includes the four vertices of $R$
     (i.e., a union of partitions of the four sides of $\partial R$). 
     The  \emph{edges} of the partition are the connected components of
     $\partial R\setminus P$; two edges are \emph{adjacent} if they have a common
     endpoint. 
     
    We say that the boundary partition $P$ has \emph{bounded geometry} 
     with constant $\lambda>1$ if 
       \begin{enumerate}[(a)]
         \item the lengths of 
          adjacent edges differ by at most a factor of $\lambda$, and 
          \item all edges have length at most $\lambda \ell$, where $\ell$ is the length of the two shorter sides of
             $R$. 
     \end{enumerate}
  \end{defn}
  
  \begin{prop}[Triangulations of a rectangle]\label{prop:rectangle}
	  Let $\lambda>1$. Then there is a constant $\theta_0>0$ with the following
      property. If $Q$ is a rectangle and $P$ is a bounded-geometry boundary partition
      with constant $\lambda$, then there is a triangulation $T$ 
      of the closed rectangle $Q$ into finitely many Euclidean triangles such that 
      \begin{enumerate}[(1)]
        \item all angles in all triangles in $T$ are bounded below by $\theta_0$;
        \item the vertices of $T$ on $\partial Q$ are precisely the elements of $P$. 
      \end{enumerate}  
  \end{prop}
  Since we are not
    aware of a reference, we give a proof of Proposition~\ref{prop:rectangle} in an appendix
     (Section~\ref{sec:appendixtriangles}.) We remark that the result can also be obtained using
      the (much more general) methods used in~\cite{quadrilateralmeshes}.
      
With these preparations, we are now ready to prove Proposition~\ref{prop:hemmedtriangulation}.
  
  \begin{proof}[Proof of Proposition~\ref{prop:hemmedtriangulation}]
         Set $\Deltafn(R) \defeq 1/\log R$.  As mentioned in~\ref{item:reductiontogenus0}, we prove the proposition first when $U$ has genus $0$. In this case, it turns out that
	  the complex dilatation is supported only
           on the annuli $A^{\gamma}$, so we  can even take $\eta=0$. 
           
           For each $\gamma\in\Gamma$, 
             glue a copy $D^{\gamma}$ 
          of the closed disc $\overline{D(\infty,1)}=\Ch\setminus \D$ into the surface $U$
           at $\gamma$.  More precisely,
      we obtain a Riemann surface structure on $U\cup \bigcup_{\gamma} D^{\gamma}$
      by using the original charts of $U$, and adding the charts (with values in $\Ch$)
         \[ z\mapsto \begin{cases} z & \text{if } z\in D^{\gamma} \\
                      (\phi^{\gamma})^{-1}(z) & \text{if }z\in A^{\gamma} \end{cases} \]
       on $D^{\gamma}\cup A^{\gamma}$. (For simplicity of notation,
        we use $z$ to denote both the point of $D^{\gamma}$ and the one of
     $\overline{D(\infty,1)}$ that it represents.) 
             
     The result is a compact Riemann surface of genus $0$, and hence 
       conformally equivalent to the Riemann sphere $\Ch$. 
       In other words, we are now
        in the following situation: 
        \[U = \Ch\setminus \bigcup_{\gamma\in\Gamma} \Phi^{\gamma}(\overline{\D}) \]
         is an analytically bounded surface, 
        bounded by the curves
         $\gamma = \Phi^{\gamma}(S^1)$. Here each $\Phi^{\gamma}$ is a 
         conformal map defined on the disc $D(0,R^{\gamma})$, and the images of
         these functions have disjoint closures.  Note that
         $\Phi^{\gamma}(z)=\phi^{\gamma}(1/z)$ on $\AA_+(R^{\gamma})$.
         We may choose coordinates
         on the sphere such that $\Phi^{\gamma}(0)=\infty$ 
         for some $\gamma$, so that $\overline{U}\subset\C$. 
        Let \[ \tilde{\gamma} \defeq \Phi^{\gamma}( \partial(D(0, \sqrt{R^{\gamma}}))) \] denote the core curve of 
      the annulus $A^{\gamma}$, and let $\tilde{U}$ denote the subdomain of $U$ bounded by the curves $\tilde{\gamma}$. 
     (In Figure~\ref{fig:DefnAlpha}, $\tilde{U}$ is the union of the light grey annulus and the white region.)

\begin{figure}
  \begin{overpic}[percent, width=.6\textwidth]{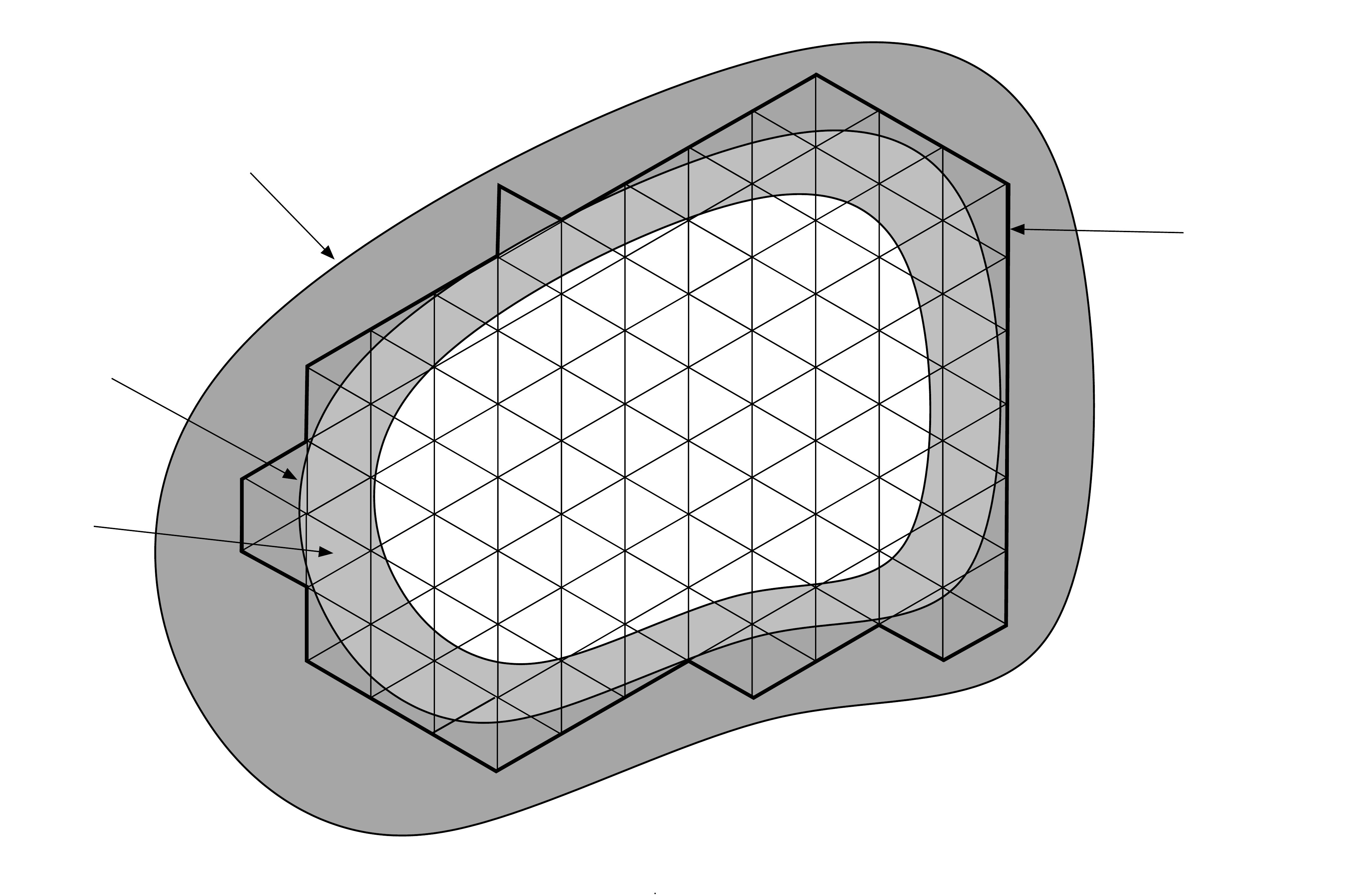}
  \put(89,47){$\alpha(\gamma)$}
  \put(15,53){$\gamma$}
  \put(5,37){$\widetilde\gamma$}
  \put(20,10){$U$}
  \put(2,25){$\tilde U$}
  \end{overpic}\\[2ex]
  \caption{\label{fig:DefnAlpha}
        Definition of $\alpha(\gamma)$. 
	For simplicity, the figure is drawn when $U$ is 
	simply connected, with a single boundary curve $\gamma$.
        }
\end{figure}

      Recall that $\mathcal{L}=\mathcal{L}_{\eps}$ is a tiling of the plane by equilateral triangles of
       side-length $\eps$. Let $\tilde{E}$ be the union of all triangles of $\mathcal{L}$ that intersect  $\tilde{U}$. 
       By Lemma~\ref{lem:boundaryapprox}, if $\eps>0$ is small enough, $\tilde{E}$ is a finite equilateral 
       Riemann surface with boundary, 
        with one boundary curve $\alpha(\gamma)$ contained in each $A^{\gamma}$ and homotopic to $\gamma$. See Figure~\ref{fig:DefnAlpha}.

    For each boundary curve $\gamma\in\Gamma$, we are in the situation described by Lemma~\ref{lem:qcstripmap}
     (where $A=A^{\gamma}$, $R=R^{\gamma}$, $\Phi=\Phi^{\gamma}|_{\AA_+(R^{\gamma})}$ and so 
     on).
     We assume that
     $\eps$ is chosen sufficiently small such that the conclusion of the lemma holds for each $\gamma$.

 We now explain how to extend the triangulation of $\tilde{E}$ to a triangulation 
  of $U^{\gamma}\setminus \tilde{E}$ for each $\gamma\in\Gamma$, using Lemma~\ref{lem:qcstripmap} and Proposition~\ref{prop:rectangle}. Fix $\gamma\in\Gamma$ and let the
  strips $S$ and $\tilde{S}$, the boundary $\Sigma_+$ and its vertex set $V_+$, and the map
  $\Psi=\Psi^{\gamma}\colon S\to\tilde{S}$ be as in Lemma~\ref{lem:qcstripmap}. Recall also that $\rho = (\log R^{\gamma})/2$.

Consider the rectangle 
\[ Q\defeq Q^{\gamma} \defeq \{a + ib \colon 0 \leq a \leq \rho \text{ and } 0\leq b \leq 2\pi\} \subset \tilde{S}. \]
        The set $P_+ \defeq \Psi(V_+)\cap Q$ is a bounded-geometry partition of the left vertical side of $Q$
        (with constant $\lambda_1>1$, which is the universal constant from Lemma~\ref{lem:qcstripmap}). Set 
            \[ P_- \defeq \{ 2\pi i j/d^{\gamma}\colon j=0,\dots,d^{\gamma} \}; \]
         this provides a partition of the right side of $Q$. 
        By Lemma~\ref{lem:qcstripmap}~\ref{item:lengthupperbound}, 
        and since $d^{\gamma} \geq \Deltafn(R^{\gamma})\geq 1/\log R^{\gamma}$, all of the edges of the two partitions have length at most
          $2\pi \log R^{\gamma}$. It follows easily that we can extend $P_+\cup P_-$ to a bounded-geometry partition of $\partial Q$
          in the sense of Definition~\ref{def:partition}, with universal constant $\lambda > \lambda_1$, 
          and where furthermore 
          the partition of the upper and lower boundary agree up to translation by $2\pi i$. 
          Now apply Proposition~\ref{prop:rectangle} to obtain 
          a triangulation $\mathcal{Q}^{\gamma}$ of $Q = Q^{\gamma}$ by Euclidean triangles, where the angles of all triangles are bounded below by $\theta_0$. Observe that,
          in particular, no vertex is incident to more than $s_1\defeq \lfloor 2\pi/\theta_0\rfloor$ edges. 

        Map $\mathcal{Q}^{\gamma}$ to an equilateral surface $E^{\gamma}$ by a homeomorphism $g^{\gamma}$ that is real-affine on each triangle. Then 
          $g^{\gamma}$ is $K_3$-quasiconformal, where $K_3$ depends only on $\theta_0$, and hence is a universal constant. 
          We form an equilateral surface $E$ as the union of $\tilde{E}$ and all $E^{\gamma}$, by identifying each 
          boundary edge $e$ of $\tilde{E}$ on $\alpha(\gamma)$ with the corresponding edge $g^{\gamma}(\Psi(\Log((\Phi^{\gamma})^{-1}(e))))$ of $E^{\gamma}$. 
          (Here $\Log$ is the branch of the logarithm taking imaginary parts between $0$ and $2\pi$.) 

          By the length-respecting property of $\Psi$, the function
           \[ g\colon \overline{U}\to E; z\mapsto \begin{cases} z & \text{if }z\in \tilde{E} \\
                                                           g^{\gamma}(\Psi^{\gamma}(\Log( (\Phi^{\gamma})^{-1}(z)))) & \text{if }z\in  \overline{U^{\gamma}}\setminus \tilde{E} \end{cases} \]
           is continuous, and hence a $K_0\defeq K_1\cdot K_3$-quasiconformal homeomorphism which is conformal on $\tilde{E}$. Every vertex of $E$ is incident to at most
             $s_2\defeq \max(6, s_1 + 4)$ edges. Finally, for any edge of $\Xi_{d^{\gamma}}$, we have
                \[ g \circ \phi^{\gamma} = g^{\gamma}\circ \Log \]
             on $E$. $\Log$ takes $e$ to one of the complementary intervals of $A_-$ in length-respecting fashion; this interval in turn is one of the edges of the
             triangulation $\mathcal{Q}^{\gamma}$. The restriction of $g^{\gamma}$ to this edge is a real-affine map, and hence length-respecting.
               This establishes~\ref{item:length-respecting} and completes the proof when $U$ has genus $0$.
      
If $U$ has positive genus $g>0$, then by definition $g$ is the largest
number such that there
are $g$ pairwise disjoint closed curves $\beta_1,\dots,\beta_g\subset U$ such that
\[ \tilde{U} \defeq U\setminus \bigcup_{i=1}^g \beta_i \] 
is connected. We may choose the $\beta_i$ to be 
analytic. Let $\psi_i\colon S^1\to \beta_i$ be analytic parameterisations,
which extend to analytic biholomorphic maps 
\[ \psi_i \colon \AA(R_i) \to A_i \subset U \]
for some $R_i>1$. We choose the $R_i$ 
sufficiently small to ensure that the closures of the annuli $A_i$ are 
pairwise disjoint, and also disjoint from the closures of the 
 $A^{\gamma}$, and that additionally their combined $\rho$-area is at most
$\eta$. 
          
Clearly $\tilde{U}$ has genus $0$. 
We can think of $\tilde{U}$ as a hemmed Riemann surface, whose boundary
curves are those inherited from $U$, together with 
two copies $\beta_i^+$ and $\beta_i^-$ of each $\beta_i$. The boundary
parameterisations are given by
\[ \phi^{\beta_i^-}\colon \AA_-(R_i) \to \tilde{U}; \quad
  z\mapsto \psi_i(z) \qquad\text{and}\qquad 
 \phi^{\beta_i^+}\colon \AA_-(R_i)\to\tilde{U};
  z\mapsto \psi_i(1/z). \]

 Now apply Proposition~\ref{prop:hemmedtriangulation} to the genus $0$ surface $\tilde{U}$, 
         where we take 
         $d^{\beta_i^+} = d^{\beta_i^-} \geq \Deltafn(R_i)$
         for each $i$. We obtain an equilateral surface-with-boundary $\tilde{E}$ and a quasiconformal map
         $\tilde{g}\colon \tilde{U}\to\tilde{E}$. For each edge $e$ of the partition of $\beta_i$ given by
         $\psi_i(\Xi_{d^{\beta_i}})$, there are two corresponding intervals $e_+$ and $e_-$ on 
         $\beta_i^+$ and $\beta_i^-$. Identifying the edges $g(e_+)$ and $g(e_-)$ on $\tilde{E}$, for every edge $e$,
          we obtain a new equilateral surface-with-boundary $E$. Every vertex of $E$ is incident to at most
          $s_0\defeq 2s_2-2$ edges.
          Conditions~\ref{item:length-respecting} and~\ref{item:gqc} for $\tilde{g}$ ensure
           that $\tilde{g}$ induces a homeomorphism 
         $g\colon U\to E$ that also satisfies this conditions. The complex dilatation of $g$ 
         is supported on the union of 
         $\bigcup_{\gamma} A^{\gamma}$ and $\bigcup_{i=1}^g A^{\beta_i^+}\cup A^{\beta_i^-}$. The latter 
         set has area at most $\eta$, as required.
 \end{proof}

 \section{Corrections on Riemann surfaces}\label{sec:teichmueller}
 
  As previously mentioned, our goal is to 
      build the desired triangulations of the non-compact surface $X$
     piece by piece on finite pieces (in the sense of Definition~\ref{defn:finitepiece}) 
     of $X$, 
     applying the construction of the preceding
      section. Recall that we may straighten these triangulations by a quasiconformal
      map to obtain an equilateral triangulation. This straightening 
      changes the surface on which the triangulation is defined, but  
      by Proposition~\ref{prop:hemmedtriangulation} the  maximal dilatation of
      the quasiconformal maps
      in question is bounded and supported
      on sets of small area. Our goal is now to justify that 
      the change to the complex structure 
      is so small that the resulting perturbed piece can be re-embedded
      into our original 
       surface $X$. 
       
\begin{prop}[Realising quasiconformal changes] \label{prop:straightening}
Let $X$ be a Riemann surface, equipped with a conformal metric $\rho$,
and let $S\subsetneq X$ be 
an analytically bounded finite piece of $X$. 
Let $K\geq 1$, and let $\delta>0$. Then there is a constant
$\eta>0$ with the following property.
Let $\mu$ be a Beltrami form on $S$ whose support $\supp(\mu)$ has area
at most $\eta$ (with respect to the metric
on $X$) and whose maximal 
dilatation is bounded by $K$. Then there is a quasiconformal homeomorphism
 \[ \psi\colon S\to \psi(S)\subset X\]
whose complex dilatation is $\mu$, which is isotopic to the identity and which satisfies
\[ \dist(z,\psi(z))<\delta \]
for all $z\in S$.
\end{prop}
 
  \begin{lem}\label{lem:reduction}
   To establish Proposition~\ref{prop:straightening}, it is
    sufficient to prove it in the special case where 
     $X$ is compact and hyperbolic, and $\rho$ is the 
     hyperbolic metric on $X$.
  \end{lem}
  \begin{proof}
    If $X$ is not compact, 
     let $\tilde{S}$ be a larger finite  piece of $X$, 
     extending $S$ by a small annulus at each boundary curve; so 
     $\overline{S}\subset \tilde{S}\subset X$.
     Now form a new, compact, Riemann surface $\tilde{X}$ by glueing, 
     into each boundary curve of $\tilde{S}$, a compact Riemann surface with a disc removed. 
     By choosing at least one of these surfaces to have genus at least $2$, we ensure that $\tilde{X}$ is hyperbolic. 

   Let $\tilde{\rho}$ be the hyperbolic
     metric on $\tilde{X}$. Since $\rho$ and $\tilde{\rho}$ are comparable on the 
     closure of $\tilde{S}$, there is $\tilde{\delta}>0$ with the following property.
     If $z\in S$ and $w\in \tilde{X}$ are such that 
     $\dist_{\tilde{\rho}}(z,w)<\tilde{\delta}$, then
     $w\in \tilde{S}$ and
     $\dist_{\rho}(z,w)<\delta$.
     
     Suppose that Proposition~\ref{prop:straightening} has been
        proved for the compact surface $\tilde{X}$; we apply it with 
        $S$, $K$ and
        $\tilde{\delta}$ to obtain a number $\tilde{\eta}>0$. Let $\eta>0$
        be so small that any subset of $S$ 
        of $\rho$-area at most $\eta$
        has $\tilde{\rho}$-area at most $\tilde{\eta}$. (Again, this is possible by
        comparability of the
        Riemannian metrics.) Then $\eta$ satisfies
        the conclusion of Proposition~\ref{prop:straightening} for $X$, $S$, $K$ and 
        $\delta$.
   \end{proof}
        
   So it remains to establish Proposition~\ref{prop:straightening} for $X$ compact and hyperbolic\footnote{%
        The requirement that $X$ be hyperbolic is made purely for convenience. Everything that
          follows is true in a suitable sense also for tori and the Riemann sphere, but assuming hyperbolicity
          means that we can avoid normalisation assumptions in the statements and considerations of special cases in the proofs.}.
     To do so, we 
    require some well-known results from the theory of Riemann surfaces, 
    quasiconformal
     mappings and
     Teichm\"uller spaces. Let us begin with two simple facts related to the compactness
     of quasiconformal mappings. 
     
 \begin{lem}[Compactness of quasiconformal mappings]\label{lem:compactness}
   Let $X$ be a compact hyperbolic Riemann surface, let $K\geq 1$, and let $\psi_n\colon X\to X$ be
     a sequence of $K$-quasiconformal self-maps of $X$.  
   Then there is a subsequence $(\psi_{n_k})_{k=0}^{\infty}$ that converges
     uniformly to a quasiconformal map $\psi\colon X\to X$. Moreover, if 
	 the complex  dilatations $\mu_{n_k}$ of $\psi_{n_k}$ 
     converge in measure to some Beltrami differential
	 $\mu$, then $\mu$ is the complex dilatation of $\psi$. 
  \end{lem}
  \begin{proof}
    According to~\cite[Theorem~4.4.1]{hubbardteichmuller}, the family of $K$-quasiconformal self-maps 
     of $X$ is equi\-continuous. Since $X$ is compact and the inverse of a $K$-quasiconformal map is $K$-quasiconformal,
     the family is indeed compact, proving the first claim. 

    The second claim follows from \cite[Theorem~I.4.6]{lehtounivalent} by lifting the maps to the disc via the universal
      covering map $\pi\colon\DD\to X$. (Recall that, if $\mu_{n_k}\to \mu$ in measure, then there is a subsequence along which it converges almost everywhere.)
  \end{proof}
  
\begin{lem}[Area distortion]\label{lem:areadistortion}
 Let $X$ be a compact hyperbolic Riemann surface, with its hyperbolic
    metric $\rho_X$, and let $K\geq 1$ and 
    $\theta>0$. 
    Then there is $\eta>0$ with the following property: 
      If $E\subset X$ is compact with
      $\area_X(E)\leq \eta$, then $\area_X(\psi(E))\leq \theta$ for all $K$-quasiconformal maps $\psi\colon X\to X$. 
 \end{lem}
 \begin{proof}
   It was first observed 
     by Bojarski  \cite{bojarskibeltrami} that
     $K$-quasiconformal mappings, suitably normalised, 
     distort area by a power depending only on $K$; 
      see the first paragraph of~\cite{gehringreicharea}. Also compare~\cite{astalaarea,eremenkoarea} for
      the optimal result. 
       These results are normally stated for self-maps of the unit disc fixing the origin. 
       In particular, the statement of Lemma~\ref{lem:areadistortion}
       holds when $X$ is replaced by $\DD$, equipped with the \emph{Euclidean} metric $\rho_{\C}$, 
       and $\psi\in \Psi_{\DD}$, where $\Psi_{\DD}$ consists of all $K$-quasiconformal self-maps of $\DD$ fixing the origin. 

      Now let $X$ be compact and hyperbolic, and let $\pi\colon\D\to X$ be a universal covering. 
        Let $A\subset \D$ with $0\in A$ be a fundamental hyperbolic polygon for the deck transformations of $\pi$. 
       If $\psi\colon X\to X$ is $K$-quasiconformal, then we may lift $\psi$ to a quasiconformal map
       $\tilde{\psi}\colon \D\to\D$ with $\pi\circ\tilde{\psi} = \psi\circ \pi$, and such that 
       $\tilde{\psi}(0)\in A$. Let $\alpha\colon\D\to\D$ be the M\"obius transformation that maps 
        $\tilde{\psi}(0)$ to $0$; then
       \[ \phi \defeq \alpha\circ\tilde{\psi}\in \Psi_{\DD}. \]

      The set $\Psi_{\DD}$ is compact by~\cite[Corollary~4.4.3]{hubbardteichmuller}; it follows that 
       there is $r$, depending only on $A$ and $K$, such that 
        $\phi(A)\subset D(0,r)$. The Euclidean and hyperbolic metrics are comparable on
         $\overline{D(0,r)}$ by a factor of at most $C\defeq 2/(1-r^2)$.

      Let $\theta>0$. By Bojarski's observation, there is $\eta>0$ (depending on $K$ and $r$) such that 
       \begin{equation}\label{eqn:distortionproof}
          \area_{\C}(\phi(\tilde{E})) \leq \frac{\theta\cdot (1-r^2)^2}{4} 
      \end{equation}
      whenever $\tilde{E}\subset \D$ has area at most $\theta$. 

    Now let $E\subset X$ have hyperbolic area at most $\eta$, and let $\tilde{E} = \pi^{-1}(E)\cap \overline{A}$. 
      Then 
        \begin{align*} &\area_{\C}(\tilde{E}) < \area_{\DD}(\tilde{E}) = \area_X(\tilde{E}) \leq \eta \qquad\text{and hence}\\
              &\area_X(\psi(E)) = \area_{\DD}(\tilde{\psi}(\tilde{E})) = 
               \area_{\DD}(\phi(\tilde{E})) \leq \frac{4}{(1-r^2)} \cdot \area_{\C}(\phi(\tilde{E})) \leq \theta \end{align*}
     by~\eqref{eqn:distortionproof}. 
    \end{proof}

     If $X$ is a hyperbolic Riemann surface, we denote by $\Teich(X)$ the Teichm\"uller
       space of $X$. Recall that $\Teich(X)$ can be defined as the set of
       equivalence classes $[\mu]_{T}$ of bounded measurable Beltrami differentials with
       $\lVert \mu\rVert_{\infty} < 1$ \cite[Proposition~6.4.11]{hubbardteichmuller}. Here two such differentials $\mu$ and $\nu$
      are equivalent
       if there is a quasiconformal homeomorphism $\psi\colon X\to X$,
       isotopic to the identity relative the ideal boundary of $X$, 
       such that $\psi^*(\nu)=\mu$ \cite[Proposition~6.4.11]{hubbardteichmuller}.    
       Here $\psi^*(\nu)$ is the pull-back of the differential $\nu$ by $\psi$;
       see \cite[Definition~4.8.10 and Formula~4.8.34]{hubbardteichmuller}.
        
       Alternatively, 
       lift $\mu$ and $\nu$ to $\DD$ via the universal covering map. Then $\nu\in[\mu]_T$ 
      if and only if the solutions $\phi_{\mu},\phi_{\nu}\colon \D\to\D$ of the 
      corresponding Beltrami equations, normalised to fix $0$ and $1$, agree on $\partial \D$. 
             $\Teich(X)$ is a complex Banach manifold, which is finite-dimensional
        if and only if $X$ is a compact surface with at most 
        finitely many punctures removed; see \cite[Section~6.5]{hubbardteichmuller}. 

  \begin{lem}\label{lem:teichmueller}
     Let $X$ be a compact Riemann surface, let $K\geq 1$, and let $(\mu_n)_{n=0}^{\infty}$ 
	  be Beltrami differentials on $X$ of maximal dilatation at most $K$. 

     Then $[\mu_n]_T \to [0]_T$ in Teichm\"uller space if and only if 
        there are representatives $\nu_n\in [\mu_n]_T$ that converge to $0$ in measure.
  \end{lem}
  \begin{remark}[Remark 1]
    We shall only require the ``if'' direction. Note that this direction is false 
      when $\Teich(X)$ is infinite-dimensional; 
         compare~\cite[Section~7]{gardinerapproximation}.
  \end{remark}
  \begin{proof}
    We use the Teichm\"uller metric on $\Teich(X)$; see~\cite[Proposition~and~Definition~6.4.4]{hubbardteichmuller}. 
     With respect to this metric, the distance between $[\mu_n]_T$ and $[0]_T$ is 
	  $\log K$, where $K$ is the infimum of the maximal
	  dilatations of $\mu\in [\mu_n]$. In particular, 
       if $[\mu_n]_T \to [0]_T$, then there are representatives of $[\mu_n]$ whose maximal dilatation converges to $1$. 
       Hence these differentials converge to $0$ in measure. 

     For the ``if'' direction, note that the points having Teichm\"uller distance at most $\log K$ from $[0]_T$ is
      compact. (It is here that we use the fact that our Teichm\"uller space is 
       finite-dimensional.) Now lift the Beltrami differentials $\mu_n$ to the universal
       cover and solve the Beltrami equation, obtaining $K$-quasiconformal maps $\phi_{\mu_n}\colon\D\to\D$ fixing $0$ and $1$. 
       By \cite[Theorem I.4.6]{lehtounivalent}, the only limit function of $\phi_{\mu_n}$ as $n\to\infty$ is given by the identity, showing that
      indeed $[\mu_n]_T\to[0]_T$. 
  \end{proof}


 We also require a result concerning the tangent space of $\Teich(X)$ at $X$, which is 
     represented by \emph{infinitesimal classes} 
       $[\mu]_B$ of
       bounded measurable Beltrami differentials. By definition, $\mu\in[0]_B$ if 
           \begin{equation}\label{eqn:pairing}
             \langle \mu , q\rangle \defeq \int_X \mu\cdot q = 0 \end{equation}
      for all $q\in A^1(X)$, and $\mu\in [\nu]_B$ are infinitesimally equivalent if $\mu-\nu\in[0]_B$. 
       Here $A^1(X)$ is the Bergman space of
       integrable holomorphic quadratic differentials on $X$. 
       In fact, the pairing~\eqref{eqn:pairing} induces an isomorphism between
       the tangent space to Teichm\"uller space and the dual space of
       $A^1(X)$~\cite[Proposition~6.6.2]{hubbardteichmuller}. 

\begin{lem}\label{lem:tangent}
 Let $X$ be a compact hyperbolic Riemann surface, $D\subset X$ a non-empty 
  sub-surface, and let
   $\mu$ be a Beltrami differential on $X$. Then there is 
   $\nu \in [\mu]_B$ such that
    $\nu=0$ a.e.\ on $X\setminus D$. 
\end{lem}
\begin{proof}
    Let $\phi$ be the linear functional on $A^1(X)$ induced by $\mu$ via the 
       pairing~\eqref{eqn:pairing}. 

    The 
    restriction of any element of $A^1(X)$ to $D$ is an element of $A^1(D)$. 
    So we can think of $A^1(X)$ as a finite-dimensional linear subspace of 
    $A^1(D)$. Since the space is finite-dimensional, the linear functional
    $\phi$ is continuous
    also with respect to the norm on $A^1(X)$ induced from that of $A^1(D)$. 
    By the Hahn--Banach theorem, $\phi$ extends to a continuous
    linear map $\tilde{\phi}\colon A^1(D)\to \C$. 
    By \cite[Proposition~6.6.2]{hubbardteichmuller}, this functional
    $\tilde{\phi}$ is generated by some Beltrami differential $\tilde{\nu}$ on $D$. 
    
   Extend $\tilde{\nu}$ to $X$ by setting it to be $0$ outside of $D$. Then 
     $\tilde{\nu}$ is in the same infinitesimal class as $\mu$ by construction,
     and we are done. 
\end{proof}

 Now we are ready to prove Proposition~\ref{prop:straightening}. 

 \begin{proof}[Proof of Proposition~\ref{prop:straightening}]
   By Lemma~\ref{lem:reduction}, we may assume that $X$ is compact and hyperbolic, and endowed with
    the hyperbolic metric. Let $D$ be an open disc in 
     $X\setminus \overline{S}$. 
     Let $\hat{\mathcal{V}}$ denote the set of Beltrami differentials supported on
	 $D$ and whose maximal dilatation is bounded by $K$, and let 
     $\mathcal{V}\subset \Teich(X)$ be the corresponding subset of
     Teichm\"uller space. 
     The projection map $\pi\colon \mu\to[\mu]_T$ from Beltrami differentials to
     Teichm\"uller space is analytic \cite[Theorem~6.5.1]{hubbardteichmuller}. 
     The derivative at $[0]_T$ of this map is precisely the projection
     $\mu\to [\mu]_B$ \cite[Corollary~6.6.4]{hubbardteichmuller}. 
     Hence Lemma~\ref{lem:tangent} implies that the restriction 
     $\pi\colon \hat{\mathcal{V}}\to \mathcal{V}$ 
     is a submersion near $[0]_T$, and therefore 
     covers a neigbourhood of $[0]_T$ in $\Teich(X)$. 
     
     Indeed, recall that $\Teich(X)$ is
     finite-dimensional, so by Lemma~\ref{lem:tangent}
     there are Beltrami differentials 
     $\mu_1\dots,\mu_n\in \hat{\mathcal{V}}$ whose infinitesimal classes
     form a basis of the tangent space of $\Teich(X)$ at $[0]_T$. Consider the 
     finite-dimensional subset 
        $\mathcal{U}= \langle \mu_1,\dots,\mu_n\rangle \cap \hat{\mathcal{V}}$; 
     then the derivative at $[0]_T$ of 
     $\pi \colon \mathcal{U}\to \mathcal{V}$ is invertible, and the claim follows by the
     inverse mapping theorem.

  By Lemma~\ref{lem:teichmueller}, if $\eta$ is sufficiently small,
    then $[\mu]_T\in \mathcal{V}$ for 
      any Beltrami differential $\mu$ on $X$ which has maximal
      dilatation at most $K$ and is supported on a set of measure less than 
      $\eta$. So for any such $\mu$, there is a Beltrami differential
         $\nu\in \hat{\mathcal{V}}$ and an at most
      $K^2$-quasiconformal map $\psi\colon X\to X$, isotopic to the identity,
      such that 
      $\psi^*(\nu)= \mu$. 
      
    Let $\mu_n$ be a sequence of Beltrami differentials on $S$ 
	 of maximal dilatation bounded by $K$, and such that the area of the support
	 of the  complex dilatation tends to $0$ as $n\to\infty$. Furthermore,
      let $D_k\subset X\setminus \overline{S}$ be a shrinking sequence of discs
      whose area tends to zero. 
      
     For $n$ sufficiently large, we can construct a map $\psi_n$ as
       above, using $D=D_{k(n)}$, with $k(n)\to\infty$ as $n\to\infty$. 
	 Then the support of the  complex dilatation $\tilde{\mu}_n$ of $\psi_n$
      is contained in the union of the support of $\mu_n$ (whose
       area tends to zero) 
             and the set $\psi_n^{-1}(D_{k(n)})$. 
      By Lemma~\ref{lem:areadistortion}, the 
      area of the latter set also tends to zero as $n\to\infty$. 
            
   By Lemma~\ref{lem:compactness}, every limit function
      of $(\psi_n)$ as $n\to\infty$ is a conformal automorphism of $X$;
      since each $\psi_n$ is isotopic to the identiy, so are the limit functions.
      But a non-trivial conformal isomorphism $\phi$ of $X$ cannot be isotopic
      to the identity (this result is usually attributed to Hurwitz). 
      Indeed, if we lift $\phi$ to the universal cover, we obtain
      a M\"obius transformation $M$ on the disc; $\phi$ is isotopic to the identity
      if and only if the boundary values of $M$, and therefore $M$ itself,
      agree with the identity; compare \cite[Proposition~6.4.9]{hubbardteichmuller}.  

  So $\psi_n$ converges to the identity. It follows that, by choosing
    $\eta$ sufficiently small in the statement of the proposition, the map $\psi$ 
    we have constructed can be chosen as close to the identity as desired. 
    In particular, we can ensure that
    $\psi^{-1}(D)\cap S = \emptyset$, and the restriction
    $\psi|_S$ solves the Beltrami equation for $\mu$, as desired.     
 \end{proof}
 \begin{remark}
    Recent work of Kahn, Pilgrim and Thurston \cite{kahnpilgrimthurston}
    more generally describes 
    when a topologically finite Riemann surface can be embedded into another,
    using an extremal length criterion. This can also be used to
    deduce Proposition~\ref{prop:straightening}, but the approach above  is more
    elementary. 
 \end{remark}

 Finally, we record the following version of Lemma~\ref{lem:areadistortion},
    for application on compact subsets of non-compact surfaces. 
\begin{prop}[Area distortion]\label{prop:areadistortion}
 Let $X$ be a Riemann surface, equipped with a conformal metric $\rho$, and let 
   $S\neq X$ be a finite piece of $X$. Let $K\geq 1$ and let 
     $B\subset S$ be compact. 
     Then there is $\eps>0$ and a function $\theta\colon (0,\infty)\to(0,\infty)$ with 
      $\theta(t)\to 0$ as $t\to 0$, such that the
    following holds.
   Suppose that $\psi$ is a $K$-quasiconformal mapping from $S$ into $X$ such that
     $\dist_{\rho}(\psi(z),z) \leq \eps$ for all $z\in S$. 
     Then, for all $A\subset B$,
      \[ \area_{\rho}(\psi(A)) \leq \theta_{\rho}(\area(A)). \]
\end{prop}
\begin{proof}
    We can deduce the claim by applying
      Lemma~\ref{lem:areadistortion} to a compact hyperbolic Riemann surface $\tilde{X}$ containing $S$, obtained 
      exactly as in the proof of Lemma~\ref{lem:reduction}. 

    Let $\hat{S}\supset B$ be a slightly smaller finite piece $\hat{S}\subset S$. 
     If $\eps$ is chosen sufficiently small, we have $\psi(\hat{S}) \subset S$ and we may extend
      $\psi|_{\hat{S}}$ to a $\tilde{K}$-quasiconformal map $\tilde{X}\to\tilde{X}$ which is the
      identity off $S$. Furthermore~-- again for sufficiently small $\eps$~-- the constant 
      $\tilde{K}$ is independent of $\psi$. Now the claim follows from Lemma~\ref{lem:areadistortion}. 
\end{proof}

\section{Construction of equilateral triangulations}     

  Our proof of Theorem~\ref{thm:triangulation} relies on a decomposition of our non-compact Riemann surface $X$ into analytically bounded finite pieces;
    see Figure~\ref{fig:decomposition}
    
\begin{figure}
 \begin{center}
  \includegraphics[width=\textwidth]{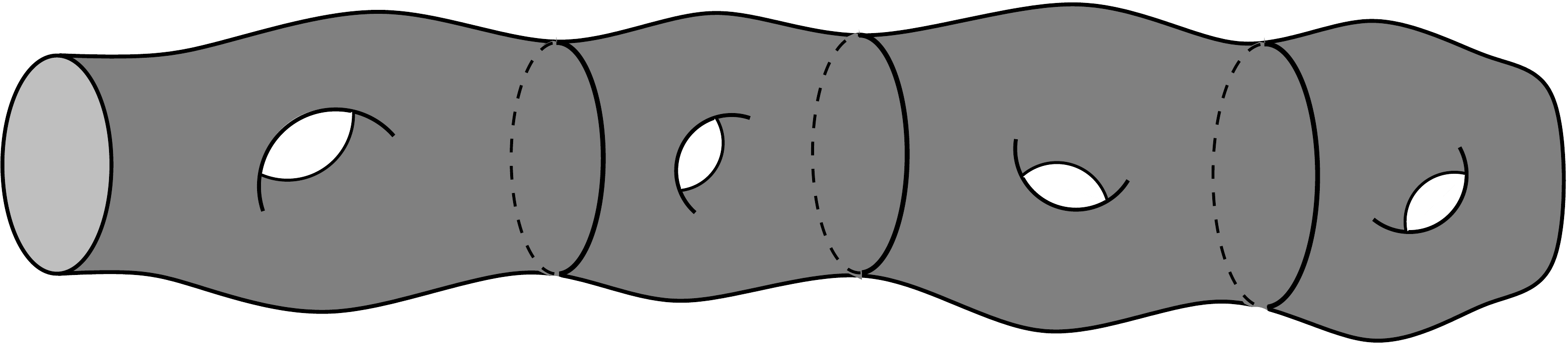}
  \end{center}
 \caption{A decomposition as in
   Proposition~\ref{prop:decomposition}, for an infinite-genus Riemann surface
    $X$.\label{fig:decomposition}}
\end{figure}

   \begin{prop}\label{prop:decomposition}
     Every non-compact Riemann surface $X$ can be written as 
        \[ X = \bigcup_{j=0}^{\infty} \overline{U_j}, \]
     where the $U_j$ are pairwise disjoint analytically bounded 
     finite pieces of $X$, such that every boundary curve 
     $\gamma$ of $U_j$ is also a boundary curve of exactly one other 
     piece $U_{j'}$ ($j'\neq j$).
  \end{prop}

   Proposition~\ref{prop:decomposition} is a purely topological consequence of Rad\'o's 
     theorem. Since we are not aware of a modern elementary
     account of this nature, we give the simple deduction below. 
     The existence of a decomposition appears to have been first 
     observed~-- for general open, triangulable, not necessarily orientable surfaces~-- 
     by Ker\'ekj\'art\'o in 1923 \cite[\S5.1,~pp 166--167]{Kerekjarto1923}. However,
      for his application (the topological classification of open surfaces), 
      Ker\'ekj\'art\'o requires additional properties of the decomposition,
      which 
      means that some additional care is required in the construction.
          
     Though favourably reviewed by Lefschetz in 1925 \cite{lefschetzreview}, in subsequent 
     years K\'erek\-j\'art\'o's work has been criticised,
     sometimes harshly~\cite{kerekjartoencyclopedia},
      for a lack of 
     rigour. In particular, 
     Richards \cite{richardsclassification} observes that the justification for 
     Ker\'ekj\'art\'o's classification theorem 
     contains gaps (which Richards fills). 
     Nonetheless, K\'erekj\'arto's argument for the 
     existence of the decomposition is correct, if somewhat informal. 
      Of course, much more precise
     statements are known, particularly in the case of Riemann surfaces; see e.g.\ \cite{structuretheorems}\footnote{Observe that Theorem~1.1 of \cite{structuretheorems},
    for topological surfaces, also follows from the earlier work of K\'erekj\'arto and Richards.}. 
    
\begin{proof}[Proof of Proposition~\ref{prop:decomposition}]
  It is equivalent to show that $X$ can be written as the increasing union of analytically
   bounded finite
   pieces $(X_j)_{j=0}^{\infty}$ with $\overline{X_j}\subset X_{j+1}$. 
   Indeed, the desired decomposition then consists of $X_0$ together with the 
   connected components of $X_{j+1}\setminus \overline{X_j}$, which are
   themselves finite pieces of $X$. 
   
  Let $\mathcal{T}$ be a triangulation of $X$, which exists by Rad\'o's theorem. 
  Fix a triangle $K_0\in\mathcal{T}$; recall that $K_0\subset X$ is compact. 
  We inductively define a sequence $(K_j)_{j=0}^{\infty}$ of compact, connected sets
   by 
   \[ K_{j+1} \defeq \bigcup\{T\in\mathcal{T}\colon T\cap K_j \neq\emptyset\}. \] 
   Then $\bigcup K_j = X$, and each interior $\interior(K_j)$ is connected, contains
   $K_{j-1}$, and is a finite piece of $X$. 
   Hence we may shrink $K_j$ (whose boundary may not be analytic) 
   slightly to obtain an analytically bounded finite piece
   $X_j$ that still contains $K_{j-1}$. 
\end{proof}

\begin{proof}[Proof of Theorem~\ref{thm:triangulation}]
 Let $X$ be a non-compact Riemann surface; we shall construct an equilateral triangulation 
   on $X$. Let $\rho$ be a complete conformal metric on $X$; for example, 
     a metric of constant curvature. 
    As mentioned in the introduction, Theorem~\ref{thm:triangulation} is trivial 
        when $X$ is Euclidean (and hence either the plane or the punctured plane). 
        So we could assume that $X$ is hyperbolic, and $\rho$ the hyperbolic metric.
        However, our 
        construction works equally well regardless of the nature of the metric,
        so we shall not require this assumption.
     
      For the remainder of the
       section, fix a decomposition $(U_j)_{j=0}^{\infty}$ of $X$ into 
       analytically bounded finite pieces, as in 
       Proposition~\ref{prop:decomposition}.

      Let $\Gamma$ be the set of all boundary curves of the $U_j$.
       For every $\gamma\in\Gamma$, there are unique $j_1<j_2$ such that 
       $\gamma$ is on the boundary of $U_{j_1}$ and of $U_{j_2}$. We 
       say that 
       $\gamma$ is an \emph{outer curve} of 
       $U_{j_1}$ and an \emph{inner curve} of $U_{j_2}$, and write 
       $\iota_-(\gamma)\defeq j_1$ and 
           $\iota_+(\gamma)\defeq j_2$.
        For $j\geq 0$, let $\Gamma^-(U_j)$ denote the set of 
         inner boundary curves of $U_j$, and let 
          $\Gamma^+(U_j)$ denote the set of
        all outer boundary curves of $U_j$.

         We may assume that the pieces $U_j$ are numbered such that 
         \[ X_j = \bigcup_{k=0}^j U_k  \cup \bigcup\{\gamma\in\Gamma\colon \iota_+(\gamma)\leq j\} \]
            is connected for all $j\geq 0$; hence $X_j$ is a finite piece of $X$. 
       Let 
       $\Gamma(X_j)$ denote the boundary curves of $X_j$; that is, 
       \[ \Gamma(X_j) = \{ \gamma\in\Gamma\colon
                                         \iota_-(\gamma) \leq j < \iota_+(\gamma)\}. \]
See Figure~\ref{fig:DefnX}.

\begin{figure}
 \begin{center}
\includegraphics[width=\textwidth]{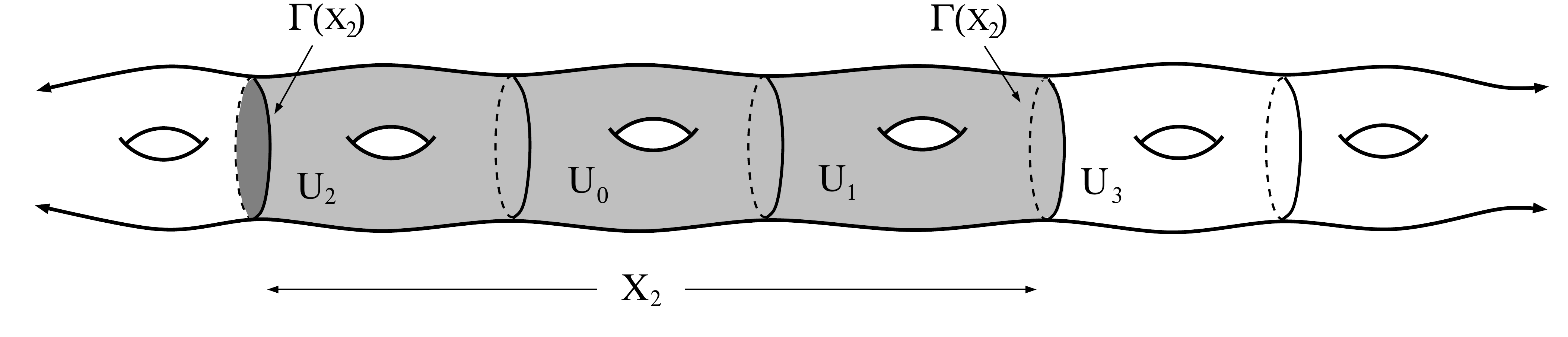}
 \end{center}
\caption{\label{fig:DefnX}The definition of $X_j$ and its set of boundary curves $\Gamma(X_j)$.}
\end{figure}
                      
     For each $\gamma$, we fix an analytic parameterisation
   $\phi^{\gamma}\colon S^1\to \gamma$. Let $\hat{R}^{\gamma}>1$ be so small that 
       $\phi^{\gamma}$ extends to a conformal isomorphism from
       $\AA(\hat{R}^{\gamma})$ onto an annulus $\hat{A}^{\gamma}$;
        we may assume that different $\hat{A}^{\gamma}$ have 
        pairwise disjoint closures.
                Set 
           \[ \hat{A}_{+}^{\gamma} \defeq \phi^{\gamma}(\AA_{+}(\hat{R}^{\gamma})) \qquad\text{and}\qquad \hat{A}_{-}^{\gamma} \defeq \phi^{\gamma}(\AA_{-}(\hat{R}^{\gamma})).\]
            Precomposing by $z\mapsto 1/z$ and decreasing $\hat{R}_{\gamma}$ if necessary, we can ensure that 
             $A_+^{\gamma} \subset U_{\iota_+(\gamma)}$
         and $A_-^{\gamma} \subset U_{\iota_-(\gamma)}$. 
         For $R\leq \hat{R}_{\gamma}$, we also define 
             \[ A^\gamma(R) \defeq \phi_\gamma(\AA(R)) \qquad \text{and}\qquad A^\gamma_{\pm}(R)\defeq \phi_\gamma(\AA_{\pm}(R)). \]
           For example, $\hat{A}^{\gamma} = A^{\gamma}(\hat{R}_{\gamma})$. 

      We use these annuli to define annular extensions of
       $\overline{X_j}$ in $X$ as follows. 
       Let $\hat{X}_j$ be the union of $\overline{X_j}$ and the annuli
          $\hat{A}^{\gamma}$ for all boundary curves $\gamma$ of $X_j$; i.e., 
            \[ \hat{X}_j = X_j \cup \bigcup\bigl\{
                                  ( \gamma \cup \hat{A}^{\gamma}_+)\colon 
                                  \iota_-(\gamma)\leq j < \iota_+(\gamma) \bigr\}. \]
           Then $X_j$ an analytically bounded finite piece of $\hat{X}_j$. 

     Fix the constant $K_0$ from  Proposition~\ref{prop:hemmedtriangulation}. 
      We define the desired triangulation piecewise, through an inductive
      construction. The underlying strategy can be described as follows.

     Apply Proposition~\ref{prop:hemmedtriangulation} to 
         construct a $K_0$-quasiconformal function $g_0\colon U_0\to E_0$, where $U_0$ is 
         considered as a 
         hemmed surface with boundary parameterisations $\phi^{\gamma}$, and $E_0$ is an equilateral surface-with-boundary. 
         In the following, we shall use without comment the properties described in the conclusion of
			Proposition~\ref{prop:hemmedtriangulation}. In particular, the equilateral triangulation of $E_0$ has local degree bounded by $s_0$, and $g_0\circ\phi^{\gamma}$ 
			maps every edge of the 
			partition $\Xi_{d^{\gamma}}$ to an edge of $E_0$ in length-respecting fashion. 
         If the degrees $d^{\gamma}$ are sufficiently large, then 
	 the  complex dilatation of $g_0$ is supported on a set of small area,
         and by Proposition~\ref{prop:straightening}, there is a quasiconformal map
         $\psi_0$ from $\hat{X}_0$ into $X$ such that 
         $f_0 \defeq g_0\circ \psi_0^{-1}$ is conformal, and $\psi_0$ is close
         to the identity. 

       Thus we have obtained an equilateral triangulation of   
       the finite piece $\tilde{X}_0 \defeq \psi_0(X_0)$ of
        $X$, which is bounded by the curves $\psi_0(\gamma)$ for 
        $\gamma\in\Gamma(X_0)$. Consider the piece $\tilde{U}_1$ whose outer
        boundary curves are the outer boundary curves of $U_1$, and whose inner
        boundary curves are given by $\psi_0(\gamma)$ for $\gamma\in\Gamma^-(U_1)$. 
        Then $\tilde{U}_1$ is a hemmed Riemann surface, where for the inner curves
        we use the boundary correspondence given by 
        \[ \phi_1^{\gamma}(\zeta) \defeq \psi_0\left(\phi^{\gamma}\left(\frac{1}{\zeta}\right)\right), \] 
         defined on some annulus $\AA_-(R^{\gamma})$. 
         Observe that
             \[ f_0 \circ \phi_1^{\gamma}  \]
            is length-respecting on $S^1$, for each $\gamma$. 
            
         We may apply Proposition~\ref{prop:hemmedtriangulation} to this hemmed surface, using the same values
           $d^{\gamma}$ on the inner boundary curves of $\tilde{U}_1$~-- assuming they were 
           chosen sufficiently large in
           step $0$. We obtain a map $g\colon \tilde{U}_1 \to E_1$. By the length-preserving 
           properties of $g$ and $f_0$, it follows that $g$ extends 
           $f_0$ continuously to a quasiconformal map  $g_1$ from 
              \[ Y_1 \defeq \tilde{X}_0 \cup \tilde{U}_1 \cup \bigcup_{\gamma\in \Gamma^-(U_1)} \psi_0(\gamma) \subset \hat{X}_1 \] 
              to an equilateral surface $\mathcal{E}_1$, which is the union of $E_0$ and $E_1$, glued along corresponding boundary curves. 
            Again, \emph{assuming that all degrees are sufficiently large}, we straighten 
            $g_1$ using a quasiconformal map $\psi_1$ from $\hat{X}_1$ into $X$.
            The result is an equilateral triangulation of the finite piece
            $\tilde{X}_1 \defeq \psi_1(Y_1)$, and we continue inductively. 
            
      More formally, the construction depends on 
         a collection of numbers $(R^{\gamma})_{\gamma\in\Gamma}$, with 
         $1<R^{\gamma}<\hat{R}^{\gamma}$, and 
         positive integers $(d^{\gamma})_{\gamma\in\Gamma}$ with 
         $d^{\gamma} \geq \Deltafn(R^{\gamma})$. (Here $\Deltafn$ is the 
         function from Proposition~\ref{prop:hemmedtriangulation}.) 
       After the $(j-1)$-th stage of the construction, we will have constructed
         the following objects. 
          \begin{enumerate}[(1)]
             \item $\tilde{X}_{j-1}$ is a finite piece of $X$, homotopic to $X_{j-1}$ 
                  and contained in $\hat{X}_{j-1}$.\label{item:Xj} 
             \item\label{item:Psij} 
                   For each boundary curve $\gamma\in\Gamma(X_{j-1})$, the corresponding
                   boundary curve of $\tilde{X}_{j-1}$ is the image of $\gamma$ under 
                   a $K_0$-quasiconformal
                   map $\Psi^{\gamma}_{j-1}$. This map is defined on 
                   $A^{\gamma}(R^{\gamma})$ and 
                   conformal on $A^{\gamma}_+(R^{\gamma})$; furthermore, 
                   \[ \Psi^{\gamma}_{j-1}(A_-^{\gamma}(R^{\gamma})) \subset \tilde{X}_{j-1}
                    \qquad\text{and}\qquad 
                    \Psi^{\gamma}_{j-1}(A_+^{\gamma}(R^{\gamma})) \cap \tilde{X}_{j-1}=
                      \emptyset.\]
            \item $\Psi^{\gamma}_{j-1}(A^{\gamma}(R^{\gamma})) \subset \hat{A}^{\gamma}$
              for each $\gamma$ as in~\ref{item:Psij}.\label{item:Psijimage}
             \item $f_{j-1}\colon \closure(\tilde{X}_{j-1})\to \mathcal{E}_{j-1}$ is a homeomorphism that is conformal on $\tilde{X}_{j-1}$, where
                $\mathcal{E}_{j-1}$ is a finite equilateral surface-with-boundary. For $\gamma\in\Gamma(X_{j-1})$, the map 
                   \[ f_{j-1} \circ \Psi^{\gamma}_{j-1} \circ \phi^{\gamma} \] 
                   maps each edge of the partition $\Xi_{d^{\gamma}}$ to a boundary edge of $\mathcal{E}_{j-1}$ in length-preserving fashion.\label{item:fj}
            \item In $\mathcal{E}_{j-1}$, every inner vertex is incident to at most $2s_0-2$ edges, and every
                boundary vertex is incident to at most $s_0$ edges.\label{item:Eboundeddegree}
          \end{enumerate}
        For $j=0$, we use the convention that 
          $\tilde{X}_{-1} = \Gamma(X_{-1}) = \mathcal{E}_{-1} = \emptyset$, so that the hypotheses are
          trivial.
           
        The inductive construction proceeds as follows. 

        \textbf{Step 1.} 
        We define $\tilde{U}_j$ to be the finite piece of $X$ bounded by 
          the curves in $\Gamma^+(U_j)$ and the curves
          $\Psi^{\gamma}_{j-1}(\gamma)$ for 
          $\gamma\in\Gamma^-(U_j)$. This piece becomes
          a hemmed surface when equipped with the boundary parameterisations
          $\phi^{\gamma}$ for the boundary curves $\gamma\in \Gamma^+(U_j)$ and 
             \[ \phi^{\gamma}_{j-1}(\zeta) \defeq 
                   \Psi^{\gamma}_{j-1}\left( \phi^{\gamma}(1/\zeta)\right) \]
         for the others. 
         
         \textbf{Step 2.} 
            We apply Proposition~\ref{prop:hemmedtriangulation} to obtain a quasiconformal map 
            \[ g_j \colon \closure(\tilde{U}_j) \to E_j, \] 
             where $E_j$ is a finite equilateral surface-with-boundary, and every vertex of $E_j$ has local degree at most $s_0$. 
            For each $\gamma\in\Gamma^-(U_j)$, the function
            $g_j \circ \Psi^{\gamma}_{j-1}\circ \phi^{\gamma}$ maps each edge of $\Xi_{d^{\gamma}}$ to 
            an edge of $E_j$ in length-respecting fashion. (Note that the map $\zeta\mapsto 1/\zeta$ is itself length-respecting on $S^1$.) 
                            
      \textbf{Step 3.} Next, we apply Proposition~\ref{prop:straightening}, 
        where $S=\hat{X}_j$ and $\mu$ is the Beltrami differential of $g_j$ on $\tilde{U}_j$,
        and $0$ elsewhere. We obtain a quasiconformal homeomorphism
        $\psi_j\colon S\to \psi(S)\subset X$, isotopic to the identity. Of course, we can
        only apply Proposition~\ref{prop:straightening} if the support of $\mu$ is
        sufficiently small; we show below that it is possible to ensure
        this by choosing 
        the sequence $(R^{\gamma})_{\gamma\in\Gamma}$ appropriately. 
        
      \textbf{Step 4.} Finally, we define $\tilde{X}_j$, functions
         $\Psi^{\gamma}_j$, an equilateral surface $\mathcal{E}_j$ and
         a function $f_j\colon \closure(\tilde{X}_j)\to\mathcal{E}_j$ such that~\ref{item:Xj},~\ref{item:Psij} and~\ref{item:fj}
         hold (with $j-1$ replaced by $j$). 
         
         Firstly, set 
         \[ Y_j \defeq \tilde{X}_{j-1} \cup \tilde{U}_j \cup 
                  \bigcup_{\gamma\in \Gamma^-(U_j)} \Psi^{\gamma}_{j-1}(\gamma)
                     \qquad\text{and}\qquad \tilde{X}_j \defeq \psi_j(Y_j). \] 
       Then $\tilde{X}_j$ is a finite piece of $X$, homotopic to $X_j$. 
         
         Note that 
           \[ \Gamma(X_j) = (\Gamma(X_{j-1})\setminus \Gamma^-(U_j)) \cup
                           \Gamma^+(U_j). \]
        The boundary curves of $\tilde{X}_j$ are given by 
           the curves $\Psi^{\gamma}_j(\gamma)$, where
            \begin{equation}\label{eqn:newboundaryparameterisation}
               \Psi^{\gamma}_j = \psi_j \circ \Psi^{\gamma}_{j-1} \end{equation}
            when $\gamma\in \Gamma(X_{j-1})\setminus \Gamma^-(U_j)$ and
             $\Psi^{\gamma}_j = \psi_j$ when
             $\gamma\in \Gamma^+(U_j)$.
            In~\eqref{eqn:newboundaryparameterisation}, recall that $\psi_j$ is
            conformal outside of $\tilde{U}_j$, and hence on $\hat{A}_{\gamma}$ for $\gamma\in\Gamma(X_{j-1})\setminus \Gamma^-(U_j)$. So
            $\Psi^{\gamma}$ is indeed $K_0$-quasiconformal on 
            $A^{\gamma}(R^{\gamma})$ and conformal on 
            $A^{\gamma}_+(R^{\gamma})$. It follows that~\ref{item:Psij} holds for
            our maps $\Psi^{\gamma}_j$. 
                    
        Finally, let $\gamma\in\Gamma^-(U_j)$, let $e$ be an edge of $\Xi_{d^{\gamma}}$, and consider
         $\tilde{e}\defeq \Psi_{j-1}^{\gamma}(\phi^{\gamma}(e))$. Then 
          $f_{j-1}(\tilde{e})$ is a boundary edge of $\mathcal{E}_{j-1}$, and $g_j(\tilde{e})$ is a boundary
          edge of $E_j$. We form an equilateral surface-with-boundary $\mathcal{E}_j$ by identifying these two boundary edges for
          each $\gamma$ and each $e$. 
         We identify $\mathcal{E}_{j-1}$ and $E_j$ with their corresponding subsets of 
          $\mathcal{E}_j$.         
          Every boundary vertex of $\mathcal{E}_j$ is a boundary vertex of $\mathcal{E}_{j-1}$ or of
             $E_j$, and therefore has local degree at most $s_0$. Every inner vertex of $\mathcal{E}_j$ is either an inner vertex of
             $\mathcal{E}_{j_1}$ or of $E_j$, or it is a common boundary vertex of both. In the latter case, the vertex is connected to at most 
             $s_0-2$ inner edges of $\mathcal{E}_{j-1}$, at most $s_0-2$ inner edges of $E_j$, and two common boundary edges of the two. 
             This establishes~\ref{item:Eboundeddegree} for $\mathcal{E}_j$. 
          
       Both $f_{j-1}$ and $g_j$ take values in $\mathcal{E}_j$. Let $\gamma$, $e$  and $\tilde{e}$ be as above, and define
          $\hat{e}=f_{j-1}(\tilde{e}) = g_j(\tilde{e})$. 
          By~\ref{item:fj} and the observation on $g_j$ in Step~2, the map $g_j\circ f_{j-1}^{-1}$ is an isometry of the edge $\hat{e}$.
          Keeping in mind that $f_{j-1}$ and $g_j$ are orientation-preserving, and take values on opposite sides of $\hat{e}$ in $\mathcal{E}_j$, 
          it follows that $g_j\circ f_{j-1}^{-1} = \id$ on $\hat{e}$. Thus
             \[ g \colon \closure(Y_j) \to \mathcal{E}_j; \quad z\mapsto \begin{cases}
                                                                                                          f_{j-1}(z) &\text{if } z\in \closure(\tilde{X}_{j-1}) \\
                                                                                                          g_{j}(z) &\text{if } z\in \closure(\tilde{U}_{j})\end{cases} \]
        is a well-defined homeomorphism. The function $f_{j-1}$ is $K_0$-quasiconformal on $\tilde{X}_{j-1}$, and $g_j$ is $K_0$-quasiconformal on 
           $\tilde{U}_j$. Since the common boundary curves
          $(\Psi_{j-1}^{\gamma}(\gamma))_{\gamma\in\Gamma^-(U_j)}$ are quasicircles, 
          $g$ is $K_0$-quasiconformal on all of $Y_j$. 
        
          Now define       
             \[ f_j \defeq g \circ \psi_j^{-1} \colon \closure(\tilde{X}_j) \to\mathcal{E}_j. \] 
               Then $f_j$ is conformal on $\tilde{X}_j$ and satisfies~\ref{item:fj}. 

       It remains to see that 
       Proposition~\ref{prop:straightening} can always be
        applied in Step 3, and that $\psi_j$ is sufficiently close to the
        identity that~\ref{item:Psijimage}, and therefore~\ref{item:Xj}, hold. 
	This requires that the  complex dilatation of the map $g_j$ can be chosen
        to be supported on a sufficiently small set. By
        Proposition~\ref{prop:hemmedtriangulation}, this dilatation is supported
        on the annuli $\Psi^{\gamma}_{j-1}(A_+^{\gamma}(R^{\gamma}))$ 
        for inner curves of $U_j$ and on the 
        annuli $\phi^{\gamma}(A_-^{\gamma}(R^{\gamma}))$  for outer curves of
        $U_j$, together with a set of negligible area. The area of the latter annuli 
        can be made small simply by choosing $R^{\gamma}$ small enough. 
        
      For the 
        former annuli, on the other hand, we must be slightly more careful. Indeed,
        the map $\Psi^{\gamma}$ is the composition of
        $\psi_{j-1}$, $\psi_{j-2}$, \dots, $\psi_{\iota^-(\gamma)}$. The last of these
        depends on $d^{\gamma}$, which in turn depends on 
        $R^{\gamma}$. So $R^{\gamma}$ must be chosen so that the
        image $A_-^{\gamma}(R^{\gamma})$ under $\Psi^{\gamma}$ is small,
        independently of the choices that determine $\Psi^{\gamma}$. Happily, 
        since the dilatation of $\Psi^{\gamma}$ is uniformly bounded, we can do so
        using the area distortion of quasiconformal mappings 
        (Proposition~\ref{prop:areadistortion}). 
            
       To make all of this precise, for each $\gamma\in\Gamma$ choose annuli 
       $\hat{A}^{\gamma}_1$ and $\hat{A}^{\gamma}_2$ with
         \[ \gamma \subset
           \hat{A}^{\gamma}_1, 
           \quad \closure(\hat{A}^{\gamma}_1) \subset \hat{A}^{\gamma}_2, \quad\text{and}\quad
           \closure(\hat{A}^{\gamma}_2) \subset \hat{A}^{\gamma}.\]
        We set 
     \[ \eps^{\gamma}_1 \defeq 
               \dist(\hat{A}^{\gamma}_2 , 
                   \partial \hat{A}^{\gamma}). \]
         Also let $\eps^{\gamma}_2$ be the constant $\eps$ from                    
           Proposition~\ref{prop:areadistortion}, with $K=K_0$, 
             $S = \hat{A}^{\gamma}_2$, and $B = \closure(\hat{A}^{\gamma}_1)$. 
             Also let $\theta = \theta^{\gamma}\colon (0,\infty)\to (0,\infty)$ be the
             function from the same proposition. So 
             a $K_0$-quasiconformal map from $\hat{A}^{\gamma}_2$ into $X$ 
             maps sets of area at most
             $\eta$ to sets of area at most $\theta^{\gamma}(\eta)$, 
             provided that it does not
             move points by more than $\eps^{\gamma}_2$. Define
             \[ \eps^{\gamma}\defeq \min(1,\eps^{\gamma}_1,\eps^{\gamma}_2). \]
             
         Next, for $j\geq 0$, choose $\eta_j$ 
               according to 
               Proposition~\ref{prop:straightening}, 
               where we use $S=\hat{X}_j$, $K=K_0$, and
                  \[ \delta = \delta_j \defeq 
                        2^{-(j+1)}\cdot \min_{\gamma\in \Gamma(X_j)} \eps^{\gamma}. \]
               
 Finally, choose 
$R^{\gamma}$
             sufficiently close to $1$ to ensure that 
             \begin{itemize}
               \item $A^{\gamma}(R^{\gamma})\subset \hat{A}^{\gamma}_1$, 
               \item $\displaystyle{\area(A^{\gamma}_-(R^{\gamma})) \leq 
                              \frac{\eta_{\iota^-(\gamma)}}{2\# \Gamma(U_{\iota^-(\gamma)})}}$, and 
               \item $\displaystyle{\theta^{\gamma}(\area(A^{\gamma}_+(R^{\gamma}))) \leq \frac{\eta_{\iota^+(\gamma)}}{2\# \Gamma(U_{\iota^+(\gamma)})}}$.
           \end{itemize}
           
Observe that this choice of $(R^{\gamma})_{\gamma\in\Gamma}$ depends only
  on the surface $X$, its metric $\rho$ and the decomposition 
 $(U_j)_{j\geq 0}$ of $X$ into finite pieces. We claim
  that, in our inductive construction, we can ensure 
\begin{enumerate}[resume*]
   \item\label{item:Psijcontrol}
    $\Psi_{j-1}^{\gamma}$ is defined on $\hat{A}^{\gamma}_2$, 
   where it satisfies
    \[ \dist(\Psi_{j-1}^{\gamma}(z),z) \leq (1-2^{-j})\cdot \eps^{\gamma}, \]
\end{enumerate}
in addition to~\ref{item:Xj}--\ref{item:fj}.

 By choice of $\eps^{\gamma}_1$ and $R^{\gamma}$, 
\ref{item:Psijcontrol} implies 
 \begin{equation}\label{eqn:Psijcontrol}
\Psi_{j-1}^{\gamma}(A^{\gamma}(R^{\gamma})) \subset
 \Psi_{j-1}^{\gamma}(\hat{A}^{\gamma}_1) \subset 
 \Psi_{j-1}^{\gamma}(\hat{A}^{\gamma}_2)\subset \hat{A}^{\gamma}. 
\end{equation}
In particular,~\ref{item:Psijimage} and~\ref{item:Xj} follow.
                       
In order to  obtain~\ref{item:Psijcontrol}, we use 
 $\eta = \eta_j/2$ when applying Proposition~\ref{prop:hemmedtriangulation}
in Step 2 of the inductive construction.
 The complex dilatation of $g_j$ is then supported on the union of 
 \begin{enumerate}[(a)]
 \item a set of area at most $\eta$;
 \item the annuli $A^{\gamma}_-(R^\gamma)$ for the outer curves 
   of $U_j$; i.e., those $\gamma\in\Gamma$ for which
  $\iota^-(\gamma)=j$;\label{item:outerannuliarea}
\item the annuli $\Psi_{j-1}^{\gamma}(A^{\gamma}_+(R^{\gamma}))$
   for the inner curves of $U_j$, i.e. those $\gamma\in\Gamma$ for which
   $\iota^+(\gamma)=j$.\label{item:innerannuliarea}
\end{enumerate}
  By choice of $R^{\gamma}$ and $\theta^{\gamma}$, and 
by~\eqref{eqn:Psijcontrol}, we see that each of the annuli
  in~\ref{item:outerannuliarea} and~\ref{item:innerannuliarea} has area at most
 \[   \frac{\eta_{j}}{2\# \Gamma(U_j)}. \]
  So the support of the dilatation has area at most $\eta_j$. 
             
            By choice of $\eta_j$, this implies that Proposition~\ref{prop:straightening}
              can indeed by applied in Step 2, and $\psi_j$ moves points at most
              a distance of $\delta_j$. Now, using~\ref{item:Psijcontrol} for $\Psi_{j-1}$,
              it follows from the definition
              of $\Psi_j^{\gamma}$ that~\ref{item:Psijcontrol} also holds for $\Psi_j$. 
      The inductive
      construction is complete.
  
   To complete the proof, we claim that
    the functions $f_j$ converge to a conformal isomorphism $f$ between $X$ and an equilateral surface $\mathcal{E}$.
       To show this,   
    fix $j\geq 0$ and define
      \[ \alpha_n \defeq \psi_n \circ \psi_{n-1} \circ \dots \circ \psi_j  \]  
      for $n\geq j$. Then $\alpha_n$ is a quasiconformal map on a neighbourhood of $\closure(\tilde{U}_j)$. Furthermore, 
        \[ \dist\bigl(\alpha_n(z), \alpha_{n+1}(z)\bigr) =
            \dist\bigl(\alpha_n(z), \psi_{n+1}(\alpha_n(z))\bigr) \leq \delta_{n+1} \leq 1/2^{n+2}. \]
            So the maps $\alpha_n$ 
      form a Cauchy sequence, and converge to a non-constant function $\alpha$ on $\closure(\tilde{U}_j)$.
      
We claim that the maximal dilatation of $\alpha_n$ is bounded by $K_0$, which is independent of $j$ and $n$. 
     Recall that, for $k\geq 0$, the 
      complex dilatation of $\psi_k$ is supported on $\tilde{U}_k$; in particular, $\psi_k$ is 
      conformal on $\tilde{X}_{k-1}$ if $k\geq 1$. Since $\psi_k(\tilde{U}_k) \subset \psi_k(Y_k) = \tilde{X}_j$, 
      it follows inductively that $\psi_n\circ\dots\ \psi_{j+1}$ is conformal on 
      $\psi_j(\tilde{U}_j)$, and hence the maximal dilatation of $\alpha_n$ on $\tilde{U}_j$ is the same
      as that of $\psi_j$, which is bounded by $K_0$.
      As a uniform limit of $K_0$-quasiconformal
      maps, $\alpha$ is also $K_0$-quasiconformal. Moreover, $\alpha_n^{-1}\to \alpha^{-1}$.
        By definition of $f_n$, we have 
          \[ f_n \circ \alpha_n|_{\tilde{U}_j} = 
             f_{n-1} \circ \alpha_{n-1}|_{\tilde{U}_j} = \dots = 
             g_j, \] 
         and hence $f_n\to g_j\circ \alpha^{-1}$ uniformly on $\closure(\tilde{U}_j)$. 
            
     So the partially defined conformal maps $f_n$ converge locally uniformly to a global conformal 
      function 
         \[ f\colon X\to \mathcal{E} \defeq \bigcup_{j=0}^{\infty} \mathcal{E}_j. \]
       Hence $X$ is conformally equivalent to the (infinite) equilateral surface $\mathcal{E}$, and the proof
       of Theorem~\ref{thm:triangulation} is complete.
 \end{proof}      
 
 \begin{proof}[Proof of Theorems~\ref{thm:belyi} and~\ref{thm:boundeddegree}]
   By Theorem~\ref{thm:triangulation}, there is an equilateral triangulation $\mathcal{T}$ on $X$. By Proposition~\ref{prop:equivalence}, there is a Belyi function 
      $f$ on $X$. This proves Theorem~\ref{thm:belyi}.

    Moreover, the triangulation $\mathcal{T}$ has the property that no vertex is incident to more than
     $2s_0-2$ edges (recall~\ref{item:Eboundeddegree} in the proof of Theorem~\ref{thm:triangulation}).
     The Belyi function constructed in the proof of Proposition~\ref{prop:equivalence} has the property that every 
     preimage of $-1$ has degree $2$, every preimage of $\infty$ has degree $3$. Furthermore, the preimages of 
     $1$ are precisely the vertices of $\mathcal{T}$, and the components of
	 $f^{-1} ([-1,1))$ are the 
     edges of $\mathcal{T}$. So every critical point of 
     $f$ has degree at most $2s_0-2$. 
 \end{proof}

\begin{figure}
\begin{center}
\subcaptionbox{\label{fig:split_triangle}Subdivision of a boundary triangle}{\includegraphics[height=0.26\textheight]{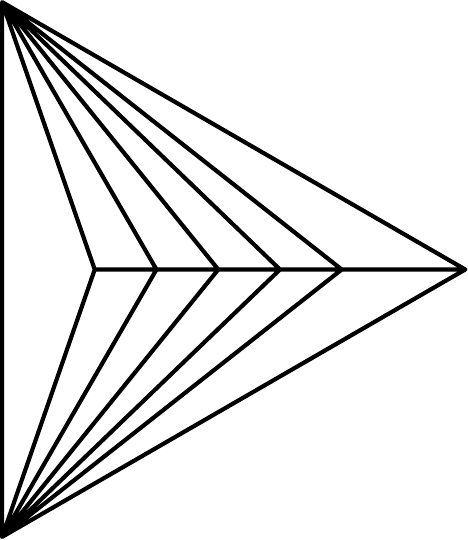}}%
\hfill%
\subcaptionbox{Subdividing all boundary triangles\label{fig:split_triangle_domain}}{\includegraphics[height = 0.26\textheight]{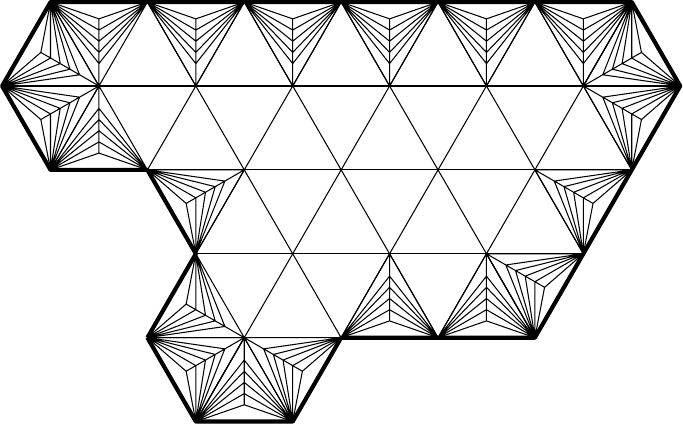}}%
\end{center}
\caption{\label{fig:subdivision}Proof of Lemma~\ref{lem:subdivision}.}
\end{figure}

 It is intuitively clear that our proof of Theorem~\ref{thm:triangulation} involves infinitely many independent choices, leading to uncountably many different 
   combinatorially
   different triangulations. To make this precise, and hence to prove Corollary~\ref{cor:uncountable}, we will use the following 
   strengthening of Proposition~\ref{prop:hemmedtriangulation}. 
      \begin{lem}\label{lem:subdivision}
         In Proposition~\ref{prop:hemmedtriangulation}, we may replace~\ref{item:angles} by 
           \begin{enumerate}[(A)]
              \item  There are universal constants $D_0\geq D_1 > 4$ with the following
                property. Every boundary vertex of $E$ has degree at least $D_1$ and at most $D_0$, and every inner vertex of $E$ has degree less than $D_1$. 
          \end{enumerate}
      \end{lem}
      \begin{proof}
          Let $\Delta$ be an equilateral triangle with vertices $A$, $B$, $C$. We may triangulate $T$ by adding $s_0$ vertices $v_1,\dots,v_{s_0}$ inside $T$, 
             where each $d_i$ is connected to $A$ and $B$ and also
             to $d_{i-1}$, with the convention that $v_0=C$. (See Figure~\ref{fig:split_triangle}.) Mapping these triangles in an affine manner to equilateral triangles, we obtain
             a quasiconformal map $h\colon T\to E_0$, where $E_0$ is an equilateral surface-with-boundary. On this surface, the two boundary vertices corresponding to
             $A$ and $B$ have degree $s_0+2$, while $C$ has degree $3$ and the interior vertices all have degree $3$ or $4$. 
      
         Let $\tilde{E}$ be the equilateral surface obtained in Proposition~\ref{prop:hemmedtriangulation}, and let $T$ be a boundary triangle; i.e., a triangle 
           in $\tilde{E}$ that has an edge on
            $\partial\tilde{E}$. We may identify $T$ with $\Delta$ such that the boundary edge corresponds to the edge $AB$. 
             We assume that $C$ is an interior vertex of $\tilde{E}$. (This is always true if we follow
            the construction in the proof of Proposition~\ref{prop:hemmedtriangulation}, but the argument
             is easily adapted if this is not the case.) 

         We can glue a copy of $E_0$ into $\tilde{E}$ in place of $T$, for every such triangle $T$. The result is 
            a new equilateral surface $E$, and a quasiconformal homeomorphism $h_1\colon \tilde{E}\to E$, whose  maximal dilatation coincides
            with that of $h$. Every boundary vertex of $\tilde{E}$ belongs to exactly two boundary triangles.
            Hence, on $E$, each of these vertices has local degree at least $D_1\defeq 2 + 2s_0\geq 14$, and at most 
            $D_0 \defeq 3 s_0$. On the other hand, any interior vertex of $\tilde{E}$ belongs to at most $s_0$ triangles. Thus it arises as the vertex
            $C$ in the above construction for at most $s_0$ different triangles, and has degree at most $2s_0 < D_1$ in $E$. Any new vertices
            in $E$ have degree at most $4 < D_1$. This completes the proof. 
      \end{proof}

\begin{proof}[Proof of Corollary~\ref{cor:uncountable}]
   First suppose that $X$ is non-compact. Let $\mathcal{T}$ be an equilateral triangulation on $X$, and 
      let $f\colon X\to\Ch$ be the corresponding Belyi function from Proposition~\ref{prop:equivalence}. 
      The vertices and edges of $\mathcal{T}$ are given by $f^{-1}(-1)$ and $f^{-1}([-1,1))$, respectively. Hence it is enough to show that 
      the proof of Theorem~\ref{thm:triangulation} can produce uncountably many different triangulations of $X$, no two of which agree up to a conformal isomorphism of $X$. 
      
   We use the notation from the proof of Theorem~\ref{thm:triangulation}, but at each stage of the construction,
     we apply the modified version of Proposition~\ref{prop:hemmedtriangulation} from Lemma~\ref{lem:subdivision}. 
     Let $\gamma\in\Gamma$, set $j \defeq \iota_-(\gamma)$, and let
     $\alpha=\alpha^j\colon \hat{X}_j\to X$ be the quasiconformal map obtained
     at the conclusion of the proof. Then 
     $\alpha^j(\gamma)$ consists of a cycle of 
     $d^{\gamma}$ edges of $\mathcal{T}$, with all vertices on this cycle having degree at least $2D_1-2 > D_1$. 
     On the other hand, any vertex of $\mathcal{T}$ that does not lie on one of these curves has degree strictly less than $D_1$.
     
   It follows 
       that the sets
       \begin{align*} \mathcal{D} &\defeq \{d^{\gamma}\colon \gamma\in\Gamma\} 
            \qquad\text{and}\\ 
            \Pi(\mathcal{D}) &\defeq \{p \text{ prime: $p$ divides $d$ for
                some $d\in\mathcal{D}$} \} \end{align*}
       are uniquely determined by the combinatorial structure of $\mathcal{T}$
       as an abstract graph.  For any infinite set $P$ 
   of prime numbers, we can choose a sequence $(d^{\gamma})_{\gamma\in\Gamma}$ 
   in such a way 
   that $d^{\gamma} \geq \Deltafn(R^{\gamma})$
     and such that $\Pi(\mathcal{D}) = P$. 
   So there are uncountably many different equilateral triangulations 
    on $X$. 
    
    On the other hand, the number of compact equilateral Riemann surfaces with $n$ faces is clearly 
      finite for every $n$, so the number of compact equilateral Riemann surfaces is countable.
      As mentioned in Remark~\ref{rmk:3colouring}, up to pre-composition by a conformal isomorphism,
      a Belyi function on a Riemann surface $X$ is uniquely determined
      by an equilateral Riemann surface together with a 3-colouring of its triangulation.
\end{proof} 
    
\section{Appendix: Triangulations of rectangles}\label{sec:appendixtriangles}
 \begin{proof}[Proof of Proposition~\ref{prop:rectangle}]
 Let $\lambda>1$, let $Q$ be a rectangle, and let $P$ be a bounded-geometry partition with constant $\lambda$.
	Since an affine stretch $x+yi \mapsto x + a y i$, for 
	$ 1 \leq  a\leq 2$, only changes angles by a bounded 
	amount, we may assume 
	   \[ Q= \{ x + iy\colon 0\leq x \leq m, 0\leq y \leq 1 \} \] for some 
	natural number $m$. Thus $\interior(Q)$ is a union of dyadic 
	squares as shown in Figure~\ref{fig:whitney} for 
	a unit square; in general, the decomposition 
	consists of the  $8m-4$ 
	dyadic squares of side length $1/4$ that don't touch
	$\partial Q$, surrounded by rings of progressively 
	smaller dyadic squares of side length $1/8, 1/16, \dots$.
	
	\begin{figure}
 \subcaptionbox{\label{fig:whitney}Whitney decomposition of a square}{\includegraphics[trim= 50 0 50 0, clip,height=2.5in]{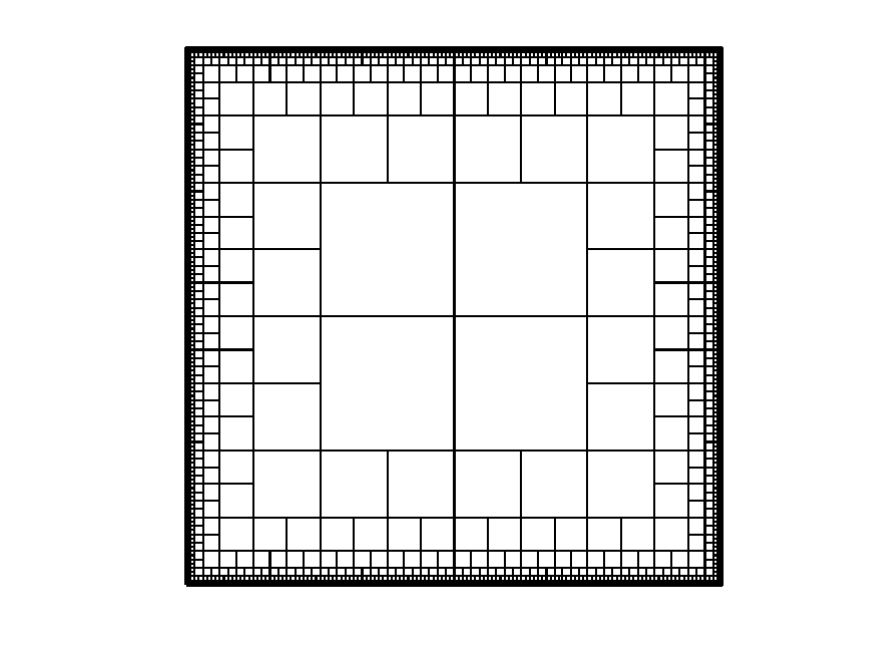}}\hfill%
 \subcaptionbox{\label{fig:sigmabdy}Polygonal arcs approximating $\partial Q$}{\includegraphics[trim= 50 0 50 0, clip,height=2.5in]{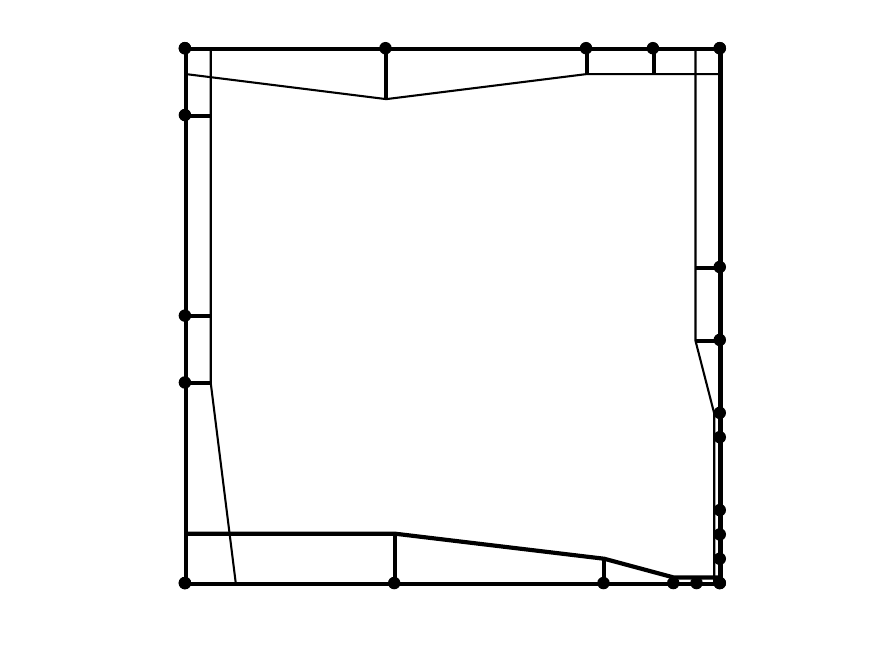}}\\
  \subcaptionbox{\label{fig:gamma}Whitney squares separated from $\partial Q$ by the arcs from~\subref{fig:sigmabdy}}{\includegraphics[trim=50 0 50 0,clip,height=2.5in]{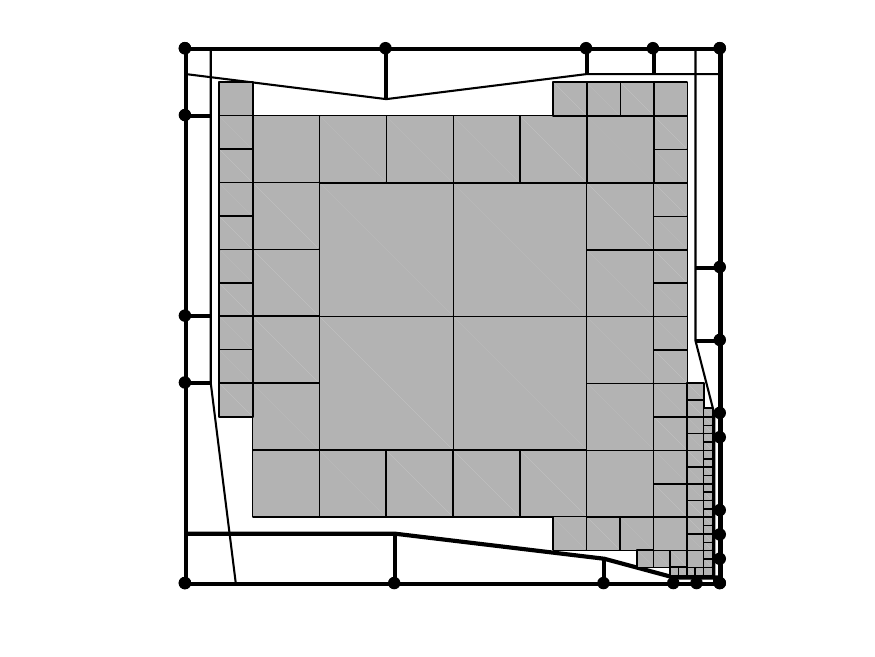}}\hfill%
  \subcaptionbox{\label{fig:triangulation}Triangulation of the region between the squares and $\partial Q$}{\includegraphics[trim=50 0 50 0,clip,height=2.5in]{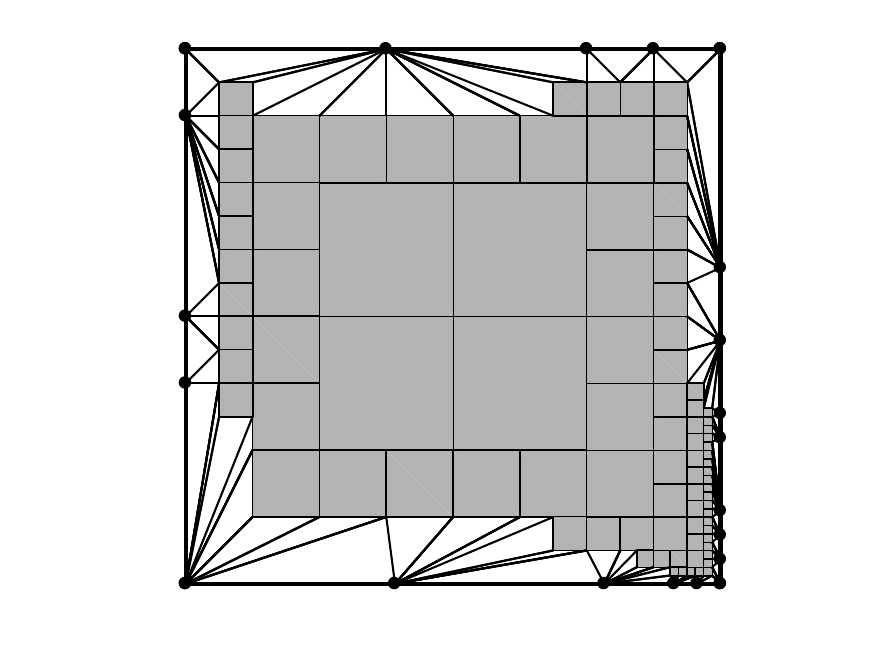}}
	\caption{\label{tri_algorithm} 
	Illustration of the proof of Proposition~\ref{prop:rectangle}}
\end{figure}

Let the $N\defeq \#P$ points of the partition be labelled as 
$x_0, x_1,\dots , x_N = x_0$ in positive orientation on $\partial Q$,
where $x_0=0$ is the lower left corner of $Q$. Indices are 
considered modulo $N$.
For each partition point $x_k$, set 
$D_k \defeq \min(|x_k-x_{k+1}|, |x_k-x_{k-1}|)$. 
Since the shorter side of $Q$  has length $1$, 
the bounded geometry assumption says that, for all indices $k$, 
 we have  $D_k\leq \lambda$ and 
    \[ \frac{1}{\lambda} \leq \frac{D_k}{D_{k+1}}\leq \lambda. \]

In particular, $D_k/(8 \lambda) \leq 1/8$, and so $D_k/(8\lambda)$ belongs to
a dyadic interval of the form $(2^{-j-1}, 2^{-j}]$
for some $j \geq 3$. Let $d_k = \frac{3}{4} 2^{-j}$ be the 
center of this interval.  Note that $d_k$ and 
$D_k/(8 \lambda)$ are 
comparable within a factor of $2$, 
so $d_k \leq D_k/(4 \lambda) \leq \min(\frac{1}{4}, D_k /4)$.

If $0=x_0 < x_1 < \dots < x_n =m$
are the partition points along the bottom edge of 
$Q$  let $z_k = x_k + i d_k$, $k=0, \dots, m$
and consider the polygonal arc  $\sigma$ with these vertices. (See Figure~\ref{fig:sigma}.)
Note that this arc connects the two vertical sides 
of $Q$ and stays within $1/4$ of the bottom edge. Moreover, every 
segment has slope between $-1/4$ and $1/4$, since
   \[ \frac{\lvert d_k - d_{k+1}\rvert}{\lvert x_k - x_{k+1}\rvert} \leq \frac{\max(d_k,d_{k+1})}{\lvert x_k - x_{k+1}\rvert}
	        \leq \frac{1}{4}\cdot \frac{ \max(D_k,D_{k+1})}{\lvert x_k - x_{k+1}\rvert} \leq \frac{1}{4}. \]  

Our choice of $d_k$ means that $z_k$ is at a 
height that is half way between the top and bottom
edges of the dyadic square $S$ that contains it. 
Since the segments of $\sigma$ have small slope, $\sigma$ leaves 
 $S$ through the two vertical side of $S$ 
and this also holds for the dyadic squares to the 
left and right of $S$.

\begin{figure}
\subcaptionbox{\label{fig:sigma}The curves $\sigma$ and $\gamma$}{\includegraphics[width=\textwidth]{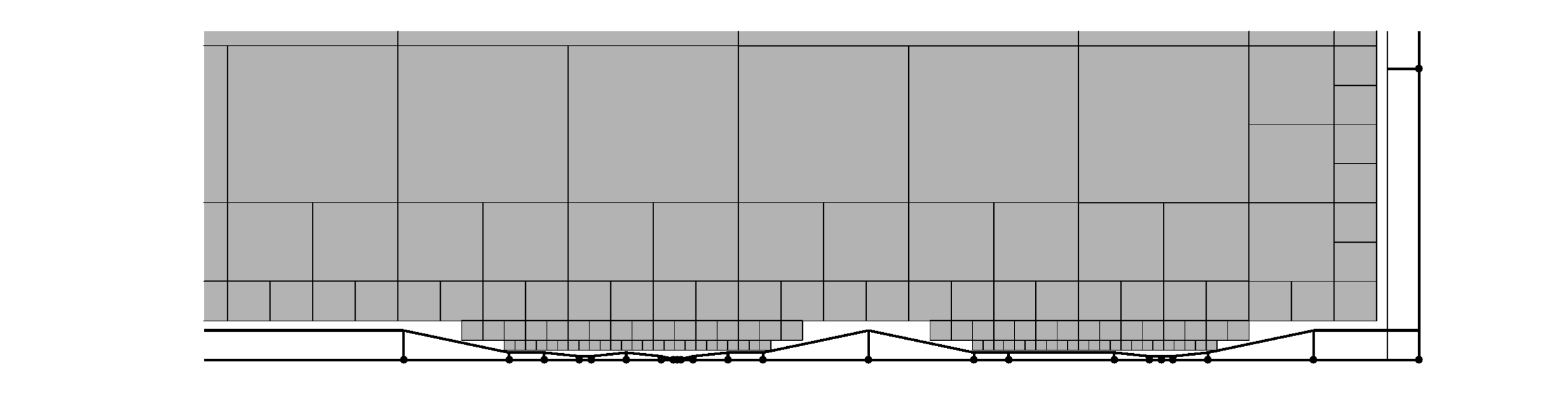}}
\subcaptionbox{\label{fig:gammatriangles}Triangulating the region between $\gamma$ and $\partial Q$}{\includegraphics[width=\textwidth]{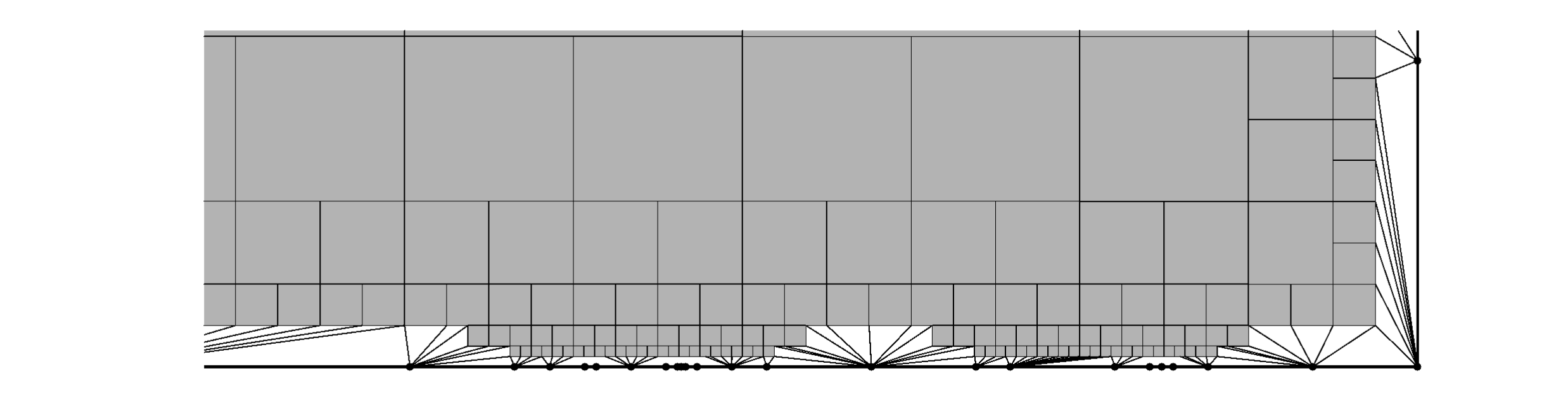}}
\caption{\label{Blowup} An enlargement of  the 
curves $\sigma$ (small slopes)  and $\gamma$  
(axis parallel boundary of boxes). 
near the boundary. Above a partition point $x_k$ the 
curve $\gamma$  is parallel to the boundary on length comparable 
to $d_k$, and  above each partition segment $\gamma$ is 
monotone (either it  is flat or forms a steps that are
all increasing or all decreasing). This makes it 
easy to verify that the region between $\gamma$ and 
$\partial Q$ can be triangulated with a lower angle 
bound and without adding vertices to $\partial Q$ or 
$\gamma$. In this picture the 
vertical scale is exaggerated to make $\sigma$ easier 
to see.}
\end{figure}

Making the same construction for each side we obtain 
four polygonal arcs $\sigma_0,\dots,\sigma_3$, each approximating one side of the 
rectangle; see Figure \ref{fig:sigmabdy}. Consider a corner point of the rectangle, say $x_0 = 0$ to fix our ideas.
The curves $\sigma_0$ and $\sigma_3$ reach the boundary at the points $i d_0$ and $d_0$, respectively,
and by the bounds on their slope, intersect in a single point within the dyadic square with centre 
	$d_0 + id_0$.

Now take the 
union of dyadic Whitney squares whose interiors do not hit the
curves $\sigma_j$ and are separated from $\partial Q$ by them 
(Figures~\ref{fig:gamma} and~\ref{fig:sigma}).
This union is itself bounded by an axis-parallel
polygon $\gamma$, which is the union of four polygonal arcs
$\gamma_0,\dots,\gamma_3$: The arc $\gamma_0$ begins at the 
upper right corner of the dyadic square centred at $d_0 + id_0$ 
(which contains the intersection point $\sigma_0$ and $\sigma_3$), and ends
similarly at the upper left corner of the square centred at $m - d_n + id_n$. (Recall
that $x_n$ is the lower right corner of the rectangle.) The arcs $\gamma_1,\dots,\gamma_3$
are characterised similarly.
	 
If we consider the polygonal arc $\sigma=\sigma_0$  corresponding to the 
bottom edge of $Q$, then the portion of $\gamma_0$ above 
each  partition arc $I$ is monotone and has a uniformly
bounded number of 
vertices, depending only on $\lambda$. Because of the monotone property, all the 
vertices  in the polygonal arc can be connected to the 
same endpoint of $I$ without hitting $\gamma$ and 
the angles between these  connecting segments is bounded
uniformly away from zero. (Figure~\ref{fig:gammatriangles}.) 
	
Moreover, for every 
partition point $x_k$ on the bottom edge, except the two corners, 
$\gamma_0$ is horizontal on some interval centered at $x_k$ 
and with length $\simeq d_k$; this is due to the property 
of $\sigma_0$  hitting only  the vertical sides of dyadic 
squares near $x_k$.
Therefore, connecting 
$x_k$ to the vertices of $\gamma_0$  whose projections are 
closest to $x_k$ to the right and left gives angles that are 
also bounded away from zero
(as mentioned above, these two points belong to the same horizontal line). 
 Do this for each side of the rectangle $Q$.
 Finally, we connect each corner to the joint endpoint of the two corresponding $\gamma_j$; 
 e.g., $0$ is connected to $4(d_0 + id_0)/3$. 
 
Now every pair $x_k$ and $x_{k+1}$ is connected to a common vertex of $\gamma$,
and likewise the two endpoints of every segment of $\gamma$ are connected to a common
vertex of our partition. Thus we have triangulated
the region between $\gamma$ and 
$\partial Q$ by triangles whose angles are bounded away from zero.        
 It is easy to triangulate the Whitney squares  so that  all 
the angles are  bounded away from zero, 
and this proves the proposition.
\end{proof}

\begin{remark}
 The method described above actually produces a triangulation 
with $O(N)$ elements where $N$ is the number of partition points 
we start with. It is simple to implement in practice; Figures
\ref{Blowup} and \ref{tri_algorithm} were produced using 
such an implementation in MATLAB.
\end{remark}

\providecommand{\bysame}{\leavevmode\hbox to3em{\hrulefill}\thinspace}
\providecommand{\MR}{\relax\ifhmode\unskip\space\fi MR }
\providecommand{\MRhref}[2]{%
  \href{http://www.ams.org/mathscinet-getitem?mr=#1}{#2}
}
\providecommand{\href}[2]{#2}


\begin{thebibliography}{BKKM86}

\bibitem[AR04]{structuretheorems}
Venancio \'{A}lvarez and Jos\'{e}~M. Rodr\'{i}guez, \emph{Structure theorems
  for {R}iemann and topological surfaces}, J. London Math. Soc. (2) \textbf{69}
  (2004), no.~1, 153--168.

\bibitem[AR15]{MR3342664}
Omer Angel and Gourab Ray, \emph{Classification of half-planar maps}, Ann.
  Probab. \textbf{43} (2015), no.~3, 1315--1349.

\bibitem[AS03]{AngelSchramm}
Omer Angel and Oded Schramm, \emph{Uniform infinite planar triangulations},
  Comm. Math. Phys. \textbf{241} (2003), no.~2-3, 191--213.

\bibitem[Ast94]{astalaarea}
Kari Astala, \emph{Area distortion of quasiconformal mappings}, Acta Math.
  \textbf{173} (1994), no.~1, 37--60.

\bibitem[BCP21]{BCP-2019}
Thomas Budzinski, Nicolas Curien, and Bram Petri, \emph{The diameter of random
  {B}ely\u i\ surfaces}, Algebr. Geom. Topol. \textbf{21} (2021), no.~6,
  2929--2957.

\bibitem[BE19]{BergweilerEremenkoQCSurgery}
Walter Bergweiler and Alexandre Eremenko, \emph{Quasiconformal surgery and
  linear differential equations}, J. Anal. Math. \textbf{137} (2019), no.~2,
  751--812.

\bibitem[Bel79]{belyi}
G.~V. Bely\u{\i}, \emph{Galois extensions of a maximal cyclotomic field}, Izv.
  Akad. Nauk SSSR Ser. Mat. \textbf{43} (1979), no.~2, 267--276, 479.

\bibitem[Bel02]{belyi2}
\bysame, \emph{A new proof of the three-point theorem}, Mat. Sb. \textbf{193}
  (2002), no.~3, 21--24.

\bibitem[Bis10]{quadrilateralmeshes}
Christopher~J. Bishop, \emph{Optimal angle bounds for quadrilateral meshes},
  Discrete Comput. Geom. \textbf{44} (2010), no.~2, 308--329.

\bibitem[Bis15]{BishopFolding}
Christopher~J. Bishop, \emph{Constructing entire functions by quasiconformal
  folding}, Acta Mathematica \textbf{214} (2015), no.~1, 1--60.

\bibitem[BKKM86]{boulatovetal}
D.~V. Boulatov, V.~A. Kazakov, I.~K. Kostov, and A.~A. Migdal, \emph{Analytical
  and numerical study of a model of dynamically triangulated random surfaces},
  Nuclear Phys. B \textbf{275} (1986), no.~4, 641--686.

\bibitem[BM04]{MR2152911}
Robert Brooks and Eran Makover, \emph{Random construction of {R}iemann
  surfaces}, J. Differential Geom. \textbf{68} (2004), no.~1, 121--157.

\bibitem[BM17]{MR3627425}
J\'{e}r\'{e}mie Bettinelli and Gr\'{e}gory Miermont, \emph{Compact {B}rownian
  surfaces {I}: {B}rownian disks}, Probab. Theory Related Fields \textbf{167}
  (2017), no.~3-4, 555--614.

\bibitem[Boy55]{bojarskibeltrami}
B.~V. Boyarski\u{\i}, \emph{Homeomorphic solutions of {B}eltrami systems},
  Dokl. Akad. Nauk SSSR (N.S.) \textbf{102} (1955), 661--664.

\bibitem[BS97]{bowersstephensonpentagonal}
Philip~L. Bowers and Kenneth Stephenson, \emph{A ``regular'' pentagonal tiling
  of the plane}, Conform. Geom. Dyn. \textbf{1} (1997), 58--68.

\bibitem[BS17]{bowersstephenson}
\bysame, \emph{Conformal tilings {I}: foundations, theory, and practice},
  Conform. Geom. Dyn. \textbf{21} (2017), 1--63.

\bibitem[BS19]{bowersstephensonII}
\bysame, \emph{Conformal tilings {II}: local isomorphism, hierarchy, and
  conformal type}, Conform. Geom. Dyn. \textbf{23} (2019), 52--104.

\bibitem[Bud20]{MR4076778}
Thomas Budzinski, \emph{Infinite geodesics in hyperbolic random
  triangulations}, Ann. Inst. Henri Poincar\'{e} Probab. Stat. \textbf{56}
  (2020), no.~2, 1129--1161.

\bibitem[CE18]{cheritatepstein}
Arnaud Ch\'{e}ritat and Adam~Lawrence Epstein, \emph{Bounded type {S}iegel
  disks of finite type maps with few singular values}, Sci. China Math.
  \textbf{61} (2018), no.~12, 2139--2156.

\bibitem[Cur16]{MR3520011}
Nicolas Curien, \emph{Planar stochastic hyperbolic triangulations}, Probab.
  Theory Related Fields \textbf{165} (2016), no.~3-4, 509--540.

\bibitem[DRV16]{DRV-2015}
Fran\c~cois David, R\'emi Rhodes, and Vincent Vargas, \emph{Liouville quantum
  gravity on complex tori}, J. Math. Phys. \textbf{57} (2016), no.~2, 022302,
  25.

\bibitem[EH95]{eremenkoarea}
A.~Eremenko and D.~H. Hamilton, \emph{On the area distortion by quasiconformal
  mappings}, Proceedings of the American Mathematical Society \textbf{123}
  (1995), no.~9, 2793--2797.

\bibitem[Eps93]{adamthesis}
Adam~Lawrence Epstein, \emph{Towers of finite type complex analytic maps},
  ProQuest LLC, Ann Arbor, MI, 1993, Thesis (Ph.D.)--City University of New
  York.

\bibitem[{Ere}04]{eremenkobelyi}
A.~{Eremenko}, \emph{{Transcendental meromorphic functions with three singular
  values}}, {Ill. J. Math.} \textbf{48} (2004), no.~2, 701--709.

\bibitem[FGJ15]{MR3339086}
N\'{u}ria Fagella, S\'{e}bastian Godillon, and Xavier Jarque, \emph{Wandering
  domains for composition of entire functions}, J. Math. Anal. Appl.
  \textbf{429} (2015), no.~1, 478--496.

\bibitem[For91]{forsterriemannsurfaces}
Otto Forster, \emph{Lectures on {R}iemann surfaces}, Graduate Texts in
  Mathematics, vol.~81, Springer-Verlag, New York, 1991, Translated from the
  1977 German original by Bruce Gilligan, Reprint of the 1981 English
  translation.

\bibitem[Fre73]{kerekjartoencyclopedia}
Hans Freudenthal, \emph{Ker\'ekj\'art\'o, {B}\'ela}, Dictionary of Scientific
  Biography, Volume VII, 1973.

\bibitem[Gar84]{gardinerapproximation}
Frederick~P. Gardiner, \emph{Approximation of infinite-dimensional
  {T}eichm\"{u}ller spaces}, Trans. Amer. Math. Soc. \textbf{282} (1984),
  no.~1, 367--383.

\bibitem[GN67]{gunningnarasimhan}
R.~C. Gunning and Raghavan Narasimhan, \emph{Immersion of open {R}iemann
  surfaces}, Math. Ann. \textbf{174} (1967), 103--108.

\bibitem[GO08]{goldbergostrovskii}
Anatoly~A. {Goldberg} and Iossif~V. {Ostrovskii}, \emph{{Value distribution of
  meromorphic functions. Transl. from the Russian by Mikhail Ostrovskii. With
  an appendix by Alexandre Eremenko and James K. Langley}}, vol. 236,
  Providence, RI: American Mathematical Society (AMS), 2008.

\bibitem[GR66]{gehringreicharea}
F.~W. Gehring and E.~Reich, \emph{Area distortion under quasiconformal
  mappings}, Ann. Acad. Sci. Fenn. Ser. A I No. \textbf{388} (1966), 15.

\bibitem[Gro97]{grothendieck}
Alexandre Grothendieck, \emph{Esquisse d'un programme}, Geometric {G}alois
  actions, 1, London Math. Soc. Lecture Note Ser., vol. 242, Cambridge Univ.
  Press, Cambridge, 1997, With an English translation on pp. 243--283,
  pp.~5--48.

\bibitem[Hub06]{hubbardteichmuller}
John~Hamal Hubbard, \emph{Teichm\"{u}ller theory and applications to geometry,
  topology, and dynamics. {V}ol. 1}, Matrix Editions, Ithaca, NY, 2006,
  Teichm\"{u}ller theory, With contributions by Adrien Douady, William Dunbar,
  Roland Roeder, Sylvain Bonnot, David Brown, Allen Hatcher, Chris Hruska and
  Sudeb Mitra, With forewords by William Thurston and Clifford Earle.

\bibitem[JW16]{joneswolfart}
Gareth~A. Jones and J\"{u}rgen Wolfart, \emph{Dessins d'enfants on {R}iemann
  surfaces}, Springer Monographs in Mathematics, Springer, Cham, 2016.

\bibitem[KPT22]{kahnpilgrimthurston}
Jeremy Kahn, Kevin~M. Pilgrim, and Dylan~P. Thurston, \emph{Conformal surface
  embeddings and extremal length}, Groups Geom. Dyn. \textbf{16} (2022), no.~2,
  403--435.

\bibitem[{Lan}02]{langleycriticalvalues}
James~K. {Langley}, \emph{{Critical values of slowly growing meromorphic
  functions}}, {Comput. Methods Funct. Theory} \textbf{2} (2002), no.~2,
  539--547.

\bibitem[Laz17]{MR3579902}
Kirill Lazebnik, \emph{Several constructions in the {E}remenko-{L}yubich
  class}, J. Math. Anal. Appl. \textbf{448} (2017), no.~1, 611--632.

\bibitem[Lef25]{lefschetzreview}
S.~Lefschetz, \emph{Book review: Ker\'ekj\'art\'o, {V}orlesungen \"uber
  {T}opologie. i. fl\"achentopologie}, Bull. Amer. Math. Soc. \textbf{31}
  (1925), 176.

\bibitem[Leh87]{lehtounivalent}
Olli Lehto, \emph{Univalent functions and {T}eichm\"{u}ller spaces}, Graduate
  Texts in Mathematics, vol. 109, Springer-Verlag, New York, 1987.

\bibitem[LG07]{MR2336042}
Jean-Fran\c{c}ois Le~Gall, \emph{The topological structure of scaling limits of
  large planar maps}, Invent. Math. \textbf{169} (2007), no.~3, 621--670.

\bibitem[LG19]{MR4007665}
\bysame, \emph{Brownian geometry}, Jpn. J. Math. \textbf{14} (2019), no.~2,
  135--174.

\bibitem[LV73]{lehtovirtanen}
O.~Lehto and K.~I. Virtanen, \emph{Quasiconformal mappings in the plane},
  second ed., Springer-Verlag, New York-Heidelberg, 1973, Translated from the
  German by K. W. Lucas, Die Grundlehren der mathematischen Wissenschaften,
  Band 126.

\bibitem[LZ04]{landozvonkin}
Sergei~K. Lando and Alexander~K. Zvonkin, \emph{Graphs on surfaces and their
  applications}, Encyclopaedia of Mathematical Sciences, vol. 141,
  Springer-Verlag, Berlin, 2004, With an appendix by Don B. Zagier,
  Low-Dimensional Topology, II.

\bibitem[Mar13]{mathoverflow}
J.~Martel, \emph{{B}elyi functions on non-compact surfaces; or: {B}uilding
  {R}iemann surfaces from equilateral triangles}, MathOverflow, 2013,
  URL:https://mathoverflow.net/q/146503 (version: 2013-10-31).

\bibitem[Mie14]{Miermont-aspects}
Gr{\'e}gory Miermont, \emph{Aspects of random maps}, Lecture notes of the 2014
  {Saint-Flour Probability Summer School}.

\bibitem[Mil06]{milnordynamics}
John Milnor, \emph{Dynamics in one complex variable}, third ed., Annals of
  Mathematics Studies, vol. 160, Princeton University Press, Princeton, NJ,
  2006.

\bibitem[Mir13]{MR3080483}
Maryam Mirzakhani, \emph{Growth of {W}eil-{P}etersson volumes and random
  hyperbolic surfaces of large genus}, J. Differential Geom. \textbf{94}
  (2013), no.~2, 267--300.

\bibitem[MPS20]{martipeteshishikura}
David Mart\'{\i}-Pete and Mitsuhiro Shishikura, \emph{Wandering domains for
  entire functions of finite order in the {E}remenko-{L}yubich class}, Proc.
  Lond. Math. Soc. (3) \textbf{120} (2020), no.~2, 155--191.

\bibitem[MS20]{MR4050102}
Jason Miller and Scott Sheffield, \emph{Liouville quantum gravity and the
  {B}rownian map {I}: the {${\rm QLE}(8/3,0)$} metric}, Invent. Math.
  \textbf{219} (2020), no.~1, 75--152.

\bibitem[MS21a]{MS-MapII}
\bysame, \emph{Liouville quantum gravity and the {B}rownian map {II}:
  {G}eodesics and continuity of the embedding}, Ann. Probab. \textbf{49}
  (2021), no.~6, 2732--2829.

\bibitem[MS21b]{MS-MapIII}
\bysame, \emph{Liouville quantum gravity and the {B}rownian map {III}: the
  conformal structure is determined}, Probab. Theory Related Fields
  \textbf{179} (2021), no.~3-4, 1183--1211.

\bibitem[OS16]{MR3525384}
John~W. Osborne and David~J. Sixsmith, \emph{On the set where the iterates of
  an entire function are neither escaping nor bounded}, Ann. Acad. Sci. Fenn.
  Math. \textbf{41} (2016), no.~2, 561--578.

\bibitem[{Rad}25]{Rado1925}
T.~{Rad\'o}, \emph{{\"Uber den Begriff der Riemannschen Fl\"ache.}}, {Acta
  Litt. Sci. Szeged} \textbf{2} (1925), 101--121 (German).

\bibitem[Rem09]{hypdim}
Lasse Rempe, \emph{Hyperbolic dimension and radial {J}ulia sets of
  transcendental functions}, Proc. Amer. Math. Soc. \textbf{137} (2009), no.~4,
  1411--1420.

\bibitem[{Rem}16]{arc-like}
Lasse {Rempe}, \emph{Arc-like continua, {J}ulia sets of entire functions, and
  {E}remenko's {C}onjecture}, 2016; Preprint arXiv:1610.06278. 

\bibitem[Ric63]{richardsclassification}
Ian Richards, \emph{On the classification of noncompact surfaces}, Trans. Amer.
  Math. Soc. \textbf{106} (1963), 259--269.

\bibitem[Sch94]{schnepscollection}
Leila Schneps (ed.), \emph{The {G}rothendieck theory of dessins d'enfants},
  London Mathematical Society Lecture Note Series, vol. 200, Cambridge
  University Press, Cambridge, 1994, Papers from the Conference on Dessins
  d'Enfant held in Luminy, April 19--24, 1993.

\bibitem[vK23]{Kerekjarto1923}
B.~von {Ker\'ekj\'art\'o}, \emph{{Vorlesungen \"uber Topologie. I.:
  Fl\"achentopologie. Mit 80 Textfiguren.}}, vol.~8, Springer, Berlin, 1923
  (German).

\bibitem[VS89]{voevodskiishabat}
V.~A. Voevodski\u{\i} and G.~B. Shabat, \emph{Equilateral triangulations of
  {R}iemann surfaces, and curves over algebraic number fields}, Dokl. Akad.
  Nauk SSSR \textbf{304} (1989), no.~2, 265--268.

\bibitem[Wil03]{noncompactpackings}
G.~Brock Williams, \emph{Noncompact surfaces are packable}, J. Anal. Math.
  \textbf{90} (2003), 243--255.

\bibitem[{Wit}55]{wittich}
Hans {Wittich}, \emph{{Neuere Untersuchungen \"uber eindeutige analytische
  Funktionen}}, vol.~8, Springer-Verlag, Berlin, 1955 (German).

\end{thebibliography}
\end{document}